\documentclass[pdflatex,sn-mathphys-num]{sn-jnl}

\usepackage{pdflscape}

\usepackage{graphicx}%
\usepackage{multirow}%
\usepackage{amsmath,amssymb,amsfonts}%
\usepackage{amsthm}%
\usepackage{bm}
\usepackage{mathrsfs}%
\usepackage[title]{appendix}%
\usepackage[table]{xcolor}%
\usepackage{textcomp}%
\usepackage{manyfoot}%
\usepackage{booktabs}%
\usepackage{algorithm}%
\usepackage{algorithmicx}%
\usepackage{algpseudocode}%
\usepackage{listings}%

\theoremstyle{thmstyleone}%
\newtheorem{theorem}{Theorem}
\newtheorem{proposition}[theorem]{Proposition}%

\theoremstyle{thmstyletwo}%
\newtheorem{remark}{Remark}%

\theoremstyle{thmstylethree}%
\newtheorem{definition}{Definition}%

\newcommand{\R}{\mathbb{R}}
\newcommand{\relu}{\operatorname{ReLU}}

\newcommand{\wl}{W^{(\ell)}}
\newcommand{\bl}{b^{(\ell)}}
\newcommand{\zl}{z^{(\ell)}}
\newcommand{\xl}{\hat{x}^{(\ell)}}
\newcommand{\al}{a^{(\ell)}}
\newcommand{\Ll}{L^{(\ell)}}
\newcommand{\Ul}{U^{(\ell)}}

\begin{document}

\title[Relaxation-Informed Training]{Relaxation-Informed Training of \linebreak Neural Network Surrogate Models}

\author{\fnm{Calvin} \sur{Tsay}}\email{c.tsay@imperial.ac.uk}

\affil{\orgdiv{Department of Computing}, \orgname{Imperial College London}, \linebreak \orgaddress{\street{South Kensington}, \postcode{SW7 2AZ}, \country{United Kingdom}}}


\abstract{ReLU neural networks trained as surrogate models can be embedded exactly in mixed-integer linear programs (MILPs), enabling global optimization over the learned function. The tractability of the resulting MILP depends on structural properties of the network, i.e., the number of binary variables in associated formulations and the tightness of the continuous LP relaxation. These properties are determined during training, yet standard training objectives (prediction loss with classical weight regularization) offer no mechanism to directly control them. This work studies training regularizers that directly target downstream MILP tractability. Specifically, we propose simple bound-based regularizers that penalize the big-M constants of MILP formulations and/or the number of unstable neurons. Moreover, we introduce an LP relaxation gap regularizer that explicitly penalizes the per-sample gap of the continuous relaxation at training points. We derive its associated gradient and provide an implementation from LP dual variables without custom automatic differentiation tools. We show that combining the above regularizers can approximate the full total derivative of the LP gap with respect to the network parameters, capturing both direct and indirect sensitivities. Experiments on non-convex benchmark functions and a two-stage stochastic programming problem with quantile neural network surrogates demonstrate that the proposed regularizers can reduce MILP solve times by up to four orders of magnitude relative to an unregularized baseline, while maintaining competitive surrogate model accuracy.}

\maketitle

\section{Introduction}\label{sec1}

Neural network surrogate models have become a popular tool in mathematical optimization, enabling complicated or unknown functions to be replaced by trained parametric approximations that can then be embedded in optimization formulations~\cite{bertsimas2025global,bradley2022perspectives,misener2023formulating}. 
Feedforward neural networks with rectified linear unit (ReLU) activations are particularly attractive for this purpose: the piecewise-linear structure of the ReLU function (and thus the combined network) allows the trained model to be encoded exactly in a mixed-integer linear program (MILP), enabling branch-and-bound global optimization~\cite{grimstad2019relu,huchette2026deep,plate2026analysis}. 
Machine learning applications include NN verification/certification~\cite{botoeva2020efficient,rossig2021advances,sosnin2024certified,sosnin2026exact}, counterfactual explanations~\cite{kanamori2021ordered, tsiourvas2024manifold}, reinforcement learning~\cite{burtea2024constrained,ryu2019caql}, and model compression~\cite{benbaki2023fast,serra2021scaling}. 
Optimization applications include optimizing over black-box objectives~\cite{perakis2022optimizing,tong2024optimization,tong2025optimization}, constraint learning~\cite{fajemisin2024optimization,maragno2025mixed}, and stochastic programming~\cite{dumouchelle2023neur2ro,kronqvist2023alternating,patel2022neur2sp}. 
Toolkits such as JANOS~\cite{bergman2022janos}, OMLT~\cite{ceccon2022omlt}, and PySCIPOpt-ML~\cite{turner2025pyscipopt} have helped popularize this approach across a range of engineering domains, including process design, energy systems, and
planning~\cite{jalving2023beyond,lopez2024process,mcdonald2024mixed,schweidtmann2019deterministic}. 
We refer the reader to~\citet{huchette2026deep} for a comprehensive overview of the intersection between ReLU neural networks and MILP. 

For a given trained network, the complexity and tractability of associated MILP formulations depends is linked to its structural properties.
In the standard big-M formulation~\cite{fischetti2018deep,lomuscio2017approach,tjeng2017evaluating}, each hidden neuron with unknown activation state is encoded by introducing a binary variable.
The number of these variables effectively dictates the combinatorial search space of the branch-and-bound search.
Equally important is the strength of the continuous LP relaxation, obtained by relaxing each binary variable to a continuous variable in $[0,1]$. 
The LP provides a bound on the MILP optimum at every node of the branch-and-bound tree, and a loose relaxation forces the solver to explore more nodes before certifying global optimality. 
Both the number of binary variables and the LP relaxation gap remain important even in more sophisticated formulations~\cite{anderson2020strong,tsay2021partition}, and can be controlled by the variable bounds, which in turn depend on the network parameters~$\theta$ (obtained during training). 
Tight bounds can reduce big-M constants, stabilize neurons, and strengthen the LP relaxation, effectively resulting in more manageable MILP problems~\cite{badilla2023computational,sosnin2024scaling,zhao2024bound}.
However, computational approaches for tightening bounds often scale poorly with network size (e.g., requiring solving optimization problems for each neuron), and these strategies are applied \emph{after training}, with no mechanism to guide the network towards tractable structures during model training. 

Standard neural network training minimizes a prediction loss and may include classical weight regularization such as $\ell_1$ or $\ell_2$ penalties. 
While inclusion of these regularizers can improve downstream MILP tractability~\cite{plate2026analysis}, neither term in the training objective directly accounts for the downstream application(s) of the resulting surrogate model.
A model trained to high accuracy may have many unstable neurons or loose LP relaxation bounds, making subsequent MILP-based optimization intractable.
This decoupling of training and optimization is an important, but largely unexplored, source of inefficiency in the surrogate modeling pipeline.

This paper proposes a family of regularization terms that can be added to the standard training loss to explicitly target the factors governing MILP tractability. 
The key observation is that the pre-activation bounds are often (sub)differentiable functions of the network parameters~$\theta$ and can therefore be incorporated into regularizers for gradient-based training.
We derive the form and gradient of each proposed regularizer, and we establish formal relationships between them and the full derivative of the LP relaxation gap with respect to~$\theta$.
Furthermore, we show that the gradient of the LP relaxation itself can be computed efficiently using sensitivity of parametric linear programs~\cite{milgrom2002,fiacco1983} and incorporated into regularizers. 

The main contributions of this paper are as follows.
\begin{enumerate}
\item We derive two bound propagation-based
  regularizers: $\mathcal{R}_{\mathrm{BW}}$ (bound-width) and
  $\mathcal{R}_{\mathrm{SN}}$ (stable-neuron). We provide closed-form expressions for their subgradients via automatic differentiation through the bound propagation.
  The bound widths prescribe the big-M constants of the MILP, and their recursive structure through the network depth can be exploited for gradient computations at the cost of a single additional forward pass per training step.
\item We introduce the LP relaxation gap regularizer  $\mathcal{R}_{\mathrm{LP}}$, which directly penalizes the per-sample continuous relaxation of the MILP at each training  point.
  We derive and express its gradient in terms of the LP dual variables and the standard backpropagation gradient. An exact implementation via a straight-through estimator avoids the need for custom differentiation tools.
\item We establish a gradient decomposition  (Proposition~\ref{prop:combined}) showing that the combined regularizer $\mathcal{R}_{\mathrm{BW}} + \mathcal{R}_{\mathrm{LP}}$ approximates the   full total derivative of the LP gap with respect to~$\theta$, capturing   both the direct sensitivity through the constraint right-hand sides and   the indirect sensitivity through the big-M constants via IBP.
\item We evaluate all regularizers on benchmark surrogate functions and on large-scale stochastic programming tasks, measuring their effect on the number of unstable   neurons, LP relaxation gap, MILP node count, and solve time across a range of network architectures and regularization strengths.
\end{enumerate}

The remainder of the paper is organised as follows. We first contextualize our contribution in relation to existing work in Section~\ref{sec:related}. Section~\ref{sec:setup} then introduces neural network models, the training objective, the big-M MILP formulation, and bound propagation. Sections~\ref{sec:regularizers} and~\ref{sec:lp} derive the bound-based regularizers and the LP gap regularizer, respectively. 
Section~\ref{sec:combined} briefly analyzes the combined regularizer and its relation to the total derivative.
Computational experiments are reported in Section~\ref{sec:results}, and conclusions are drawn in Section~\ref{sec:conclusions}.

\subsection{Relation to existing work}
\label{sec:related}
The idea of training for downstream tractability MILP has parallels in the adversarial machine learning literature, where networks trained for certified robustness have been shown to exhibit fewer unstable neurons and tighter bounded output domains than their unregularized counterparts~\cite{xiao2019training,gowal2018effectiveness,mirman2018differentiable,zhang2020towards}. 
Here, `robustness' refers to provable immunity to adversarial perturbations within a norm ball around each input at inference~\cite{rossig2021advances} or training~\cite{sosnin2025abstract,sosnin2024certified} time. 
More generally, the motivation of training a model with its downstream optimization use in mind is shared with the \emph{decision-focused learning} literature~\cite{mandi2024decision}, e.g., `smart predict-then-optimize'~\cite{elmachtoub2022smart} and task-based learning~\cite{donti2017task}.

While the regularizers developed in Sections~\ref{sec:regularizers}--\ref{sec:lp} share some technical components with the verification literature (notably differentiable bounds and penalties that target neuron stability), our motivation, formulation, and context differ in several important respects detailed below. 
On the other hand, while we consider a similar application as the decision-focused learning literature, the proposed regularization strategies purely target downstream MILP tractability, rather than improving the quality of decisions. 

\textbf{Application domain.}
Certified training methods target \emph{adversarial robustness} of classifiers, typically defined as classification accuracy under worst-case perturbation in an $\epsilon$-ball around each test image. 
Note that this problem can often be solved without finding the worst-case perturbation (i.e., completely solving a MILP): it requires only a successful worst-case perturbation or safe bound~\cite{rossig2021advances, tsay2021partition}. 
On the other hand, the proposed methods primarily target \emph{surrogate models for optimization}. In this setting, the network approximates a continuous function over a known box domain~$\mathcal{X}$, and the downstream task is to solve a MILP over/involving the trained surrogate~\cite{grimstad2019relu}. 
In contrast to the verification literature, the relevant metrics therefore include optimization properties such as LP relaxation gap and MILP solve time.  

\textbf{Training for certification.}
In certified training, intermediate variable bounds are used to compute an output-level bound, rather than as targets themselves~\cite{gowal2018effectiveness,zhang2020towards}.  
Nevertheless, in MILP intermediate bounds directly influence the tightness of the overall formulation, and we therefore quantify and regularize with total bound width across all hidden neurons (an objective with no analog in the certified training literature). 
Furthermore, we propose regularizers directly targeting the relaxation gap, which is entirely specific to the MILP surrogate setting.
Most similar to the present work methodologically are regularizers targeting neuron stability, which can accelerate verification of classification networks~\cite{xiao2019training,dvijotham2018training}.
The present work studies stability regularizers for the MILP setting, where the network is embedded as an objective or constraint in a downstream optimization problem rather than verified for a fixed property. 
Finally, we note the certified training literature typically studies multi-class classifiers with cross-entropy loss and specification-based margins.
Our setting involves scalar regression surrogates trained with MSE loss, where there is no class margin to certify, and the relevant relaxation is the continuous relaxation of the surrogate MILP, not the convex outer adversarial polytope.
This definition of model quality gives a fundamentally different perspective to the tradeoff between accuracy and tractability. 

\textbf{Decision-focused learning.} 
Many decision-focused learning methods modify the training loss of the surrogate model completely~\cite{elmachtoub2022smart,TangKhalil2024}, e.g., minimizing task loss or regret rather than prediction error alone. These modifications aim to improve the quality of decisions produced in downstream applications. 
Most similar to the present work in motivation are regularizers targeting gradients of the learned surrogate model, with the purpose of improving \emph{optimization performance} in downstream (gradient-based) optimization~\cite{amos2017input,rosemberg2025sobolev,tsay2021sobolev}. 
The present paper takes an orthogonal approach within this broader family.
Here, our objective is not to improve the quality of the solution found (which is also an important problem), but to reduce the computational cost of finding it via MILP. 
In other words, decision-focused learning assumes the downstream problem is solvable and asks how to improve solution quality; this paper assumes a useful surrogate and asks how to make it tractable.

\section{Background}
\label{sec:setup}

\subsection{Neural network notation}
We consider feed-forward neural networks (NNs), which are directed acyclic graphs comprising nodes/neurons structured into $L$ hidden layers.
At each layer $l=1,...,L+1$, the NN contains nodes that receive the outputs of nodes in the preceding layer ($l-1$) as inputs. Each node then computes a weighted sum of its inputs (known as the preactivation), and applies a nonlinear activation function to this computed term. While many options for activation function are available, we focus on the ReLU activation function
\[
y= \max (0,w^Tx + b),
\]
which is amenable to mixed-integer linear programming (MILP) formulations, given its piecewise-linear form~\cite{huchette2026deep}. 

Mathematically, we denote a feedforward neural network model as
$f_\theta : \R^{n_0} \to \R$
with $L$ hidden layers, ReLU activations, and a linear output layer.
Denote the weight matrix and bias vector at layer $\ell$ by
$\wl \in \R^{n_\ell \times n_{\ell-1}}$ and $\bl \in \R^{n_\ell}$,
respectively, for $\ell = 1, \dots, L{+}1$.
The collective parameter vector for the model is expressed as 
$\theta = \{(\wl, \bl)\}_{\ell=1}^{L+1}$.
Finding the values for $\theta$, e.g., to best fit a given dataset, is referred to as \emph{training} the NN model. 

For an input $x \in \R^{n_0}$, evaluation of the neural network $f_\theta(x)$ is referred to as a \textit{forward pass}. In particular, the forward pass computes:
\begin{align}
  \zl_j   &= \wl_j \xl[\ell-1] + \bl_j,
  &\ell &= 1,\dots,L{+}1, \quad j = 1,\dots,n_\ell, \label{eq:preact} \\[4pt]
  \xl_j   &= \relu(\zl_j) = \max(0,\, \zl_j),
  &\ell &= 1,\dots,L, \label{eq:relu}
\end{align}
with $\hat{x}^{(0)} = x$ being the input and $f_\theta(x) = z^{(L+1)}$ being the scalar output (note the lack of nonlinear activation at the output layer).
We consider the input domain as a box
$\mathcal{X} = \{x : x^{\mathrm{lb}} \le x \le x^{\mathrm{ub}}\}$, noting that in general, other constraints on $x$ could be added in later steps. 

\subsection{Training the neural network}
\label{sec:objective}

The neural network parameters are trained using training data $\{(x_i, y_i)\}_{i=1}^N$, e.g., samples drawn from a function we wish to approximate $g : \mathcal{X} \to \R$. Generally, we find values of $\theta$ by minimizing empirical loss over the training data. For regression tasks, this is often taken using the mean squared error (MSE):
\[
\mathcal{L}_{\mathrm{MSE}} = \frac{1}{N}\sum_{i=1}^N \bigl(f_\theta(x_i) - y_i\bigr)^2.
\]

To avoid overfitting (or otherwise guide the training process), regularization terms can be appended to the loss function:
\begin{equation}\label{eq:loss}
  \mathcal{L}(\theta)
  \;=\;
\mathcal{L}_{\mathrm{MSE}}
  \;+\;
  \lambda\,\mathcal{R}(\theta),
\end{equation}
where $\mathcal{R}(\theta)$ is a regularization term and $\lambda > 0$ controls the trade-off between accuracy and regularization. 
We refer the reader to \citet{goodfellow2016deep} for a comprehensive overview of this training paradigm. 

Training the neural network is typically performed using gradient-based optimization methods, requiring the computation of $\nabla_\theta \mathcal{L}$, e.g., using back-propagation. 
We observe that gradients are therefore required for both terms in \eqref{eq:loss}, $\nabla_\theta \mathcal{L}_{\mathrm{MSE}}$ and $\nabla_\theta \mathcal{R}(\theta)$. 
This study precisely aims to introduce regularizers $\mathcal{R}(\theta)$ targeting the MILP tractability of the resulting NN surrogate. 
In Section~\ref{sec:regularizers}, we explicitly derive the form and gradient of each regularizer $\mathcal{R}$ we consider.

\subsection{Mixed-integer optimization formulation}
\label{sec:bigm}

\begin{figure}[hb]
\centering
    \includegraphics[width=0.9\textwidth]{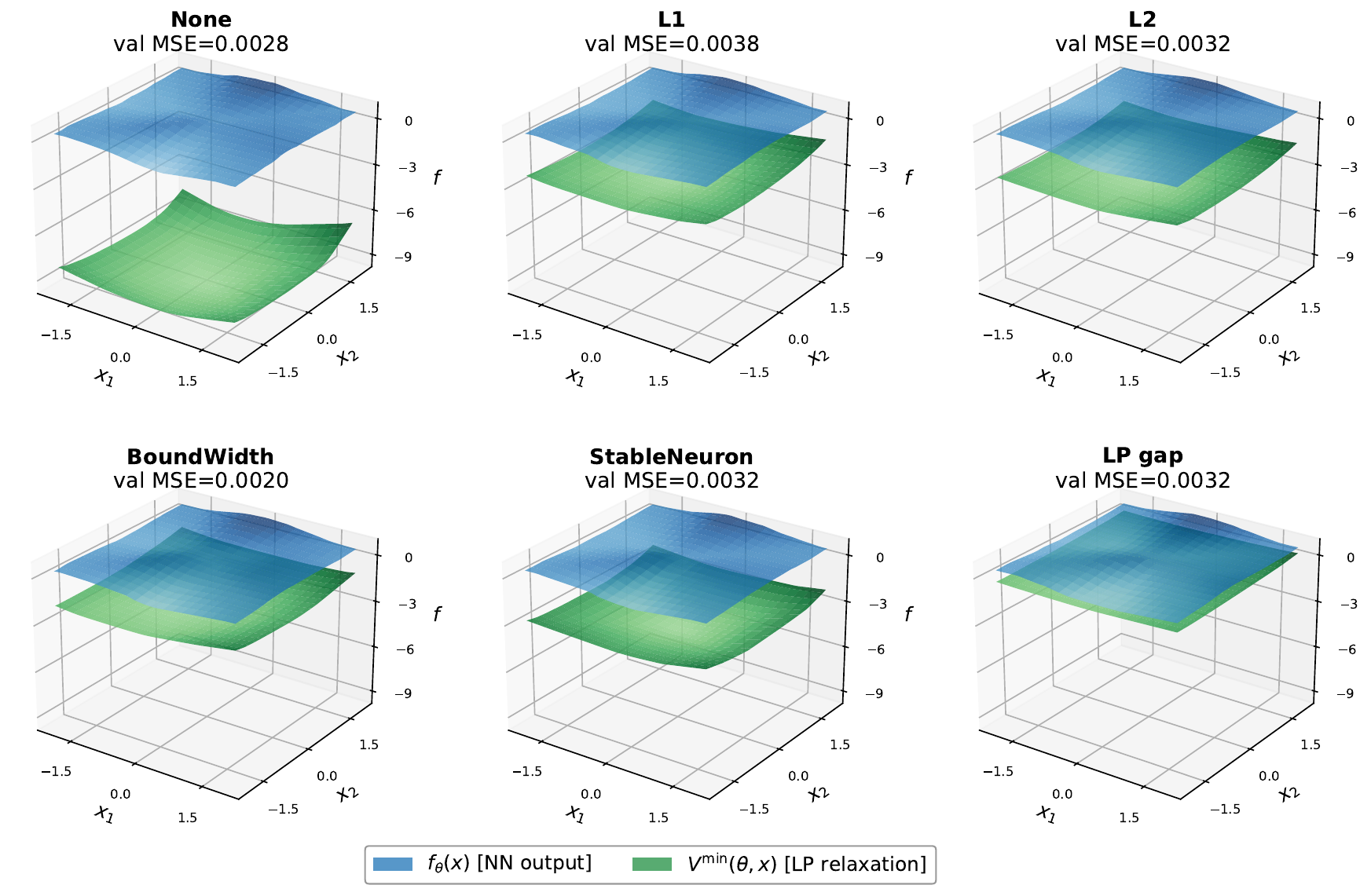}
    \caption{Predictions and LP lower bounds for NN models with \{32,32\} hidden layers trained on the scaled Peaks function with two inputs ($x_1$,$x_2$), with output $f(x)$. Different regularizers are applied during training, with weights chosen to maintain validation MSE of similar scale.}
    \label{fig:relaxations_3d}
\end{figure}

In contrast to the training of NNs (where the parameters $\theta$ are decision variables), optimization over a NN surrogate seeks to compute extreme cases for an \textit{already trained} model.  
In other words, the parameters $\theta$ are fixed, and we optimize over $f_\theta$ as a fixed function (or embed it within constraints in a larger problem). 
This step therefore requires formulating $f_\theta$ over $\mathcal{X}$ as optimization constraints. 

In MILP formulations, each ReLU unit is commonly
modelled with a binary variable $\al_j \in \{0,1\}$ indicating
whether neuron $j$ at layer $\ell$ is active ($\al_j = 1$) or
inactive ($\al_j = 0$).
Let $\Ll_j \le \zl_j \le \Ul_j$ be valid pre-activation bounds.
The big-M formulation of
$\xl_j = \relu(\zl_j)$
is~\cite{fischetti2018deep,lomuscio2017approach,tjeng2017evaluating}:
\begin{subequations}\label{eq:bigm}
\begin{alignat}{2}
  \xl_j &\ge \zl_j, \label{eq:bigm-lb1} \\
  \xl_j &\ge 0, \label{eq:bigm-lb2} \\
  \xl_j &\le \zl_j - \Ll_j\,(1 - \al_j), \label{eq:bigm-ub1} \\
  \xl_j &\le \Ul_j\,\al_j, \label{eq:bigm-ub2}
\end{alignat}
\end{subequations}
for each hidden neuron $(\ell, j)$, $\ell = 1, \dots, L$. The tractability of this formulation hinges on the values used for the bounds $\Ll_j$ and $\Ul_j$, i.e., the big-M coefficients~\cite{sosnin2024scaling}. Notice that smaller values for these coefficients yield tighter constraints~\eqref{eq:bigm-ub1}--\eqref{eq:bigm-ub2}. Ideally, these bounds are taken to be as tight as possible such that they remain valid, with $\zl_j \in [\Ll_j, \Ul_j]$. 
We note that several interesting alternatives to the big-M formulation have been proposed~\cite{anderson2020strong,tsay2021partition}, but it remains popular given its simplicity.  

\begin{remark}[Strength of continuous relaxation]
MILP is often solved using branch-and-bound algorithms, which leverage a cheaper, continuous relaxation to bound the objective at each node of the search tree. 
The solver then explores the domain over decision variables by `branching' until the gap between the best feasible objective value found and the tightest relaxation found falls below a given tolerance. A tighter, or \emph{stronger}, relaxation can reduce this search tree considerably. 
The continuous relaxation of~\eqref{eq:bigm} is obtained by relaxing $\al_j \in \{0,1\}$ to $\al_j \in [0,1]$, yielding a linear program (LP) whose optimal value bounds the MILP optimum.
The \emph{LP relaxation gap} is the difference between this LP bound and the MILP optimum. 
Since the LP relaxation of~\eqref{eq:bigm-ub1}--\eqref{eq:bigm-ub2}
tightens as $|\Ll_j|$ and $\Ul_j$ decrease, the choice of bounds is a primary determinant of relaxation strength, and therefore MILP solve efficiency. 
Figure~\ref{fig:relaxations_3d} illustrates the predictions and continuous relaxations for several NN models. 
\end{remark}

\begin{definition}[Neuron stability]\label{def:stable}
A neuron is \emph{stable active} if $\Ll_j \ge 0$
(the ReLU never turns off), in which case $\xl_j = \zl_j$ and $\al_j = 1$ can be fixed. 
Likewise, a neuron is \emph{stable inactive} if $\Ul_j \le 0$
($\xl_j = 0$, $\al_j = 0$). 
A neuron is said to be \emph{stable} if it is either stable active or inactive; for stable neurons, the value of $\al_j$ is fixed, and no binary variable is required. 
The set of \emph{unstable} neurons requiring a binary
variable to formulate using \eqref{eq:bigm} is therefore defined:
\begin{equation}\label{eq:num_unstable}
  \mathcal{U} \;=\; \bigl\{(\ell, j) : \Ll_j < 0 < \Ul_j\bigr\}.
\end{equation}
The number of (unfixed) binary variables in the MILP resulting from applying \eqref{eq:bigm} to all neurons equals $|\mathcal{U}|$.
\end{definition}

\subsection{Obtaining and tightening bounds}
\label{sec:ibp}

Given the input domain $\mathcal{X} = [x^{\mathrm{lb}},\, x^{\mathrm{ub}}]$,
simple valid pre-activation bounds $\Ll_j, \Ul_j$ can be computed by applying interval arithmetic layer by layer. 
This is also referred to as interval bound propagation, or IBP. 
Let $\hat{l}^{(\ell)}, \hat{u}^{(\ell)}$ denote the post-activation
(post-ReLU) bounds at layer~$\ell$, with the input layer defined by given bounds $\hat{l}^{(0)} = x^{\mathrm{lb}}$, $\hat{u}^{(0)} = x^{\mathrm{ub}}$. 

Valid pre-activation bounds for a layer can be computed using interval arithmetic:
\begin{subequations}\label{eq:ibp}
\begin{align}
  \Ll_j &= [\wl_j]^{+}\, \hat{l}^{(\ell-1)}
          + [\wl_j]^{-}\, \hat{u}^{(\ell-1)}
          + \bl_j,
  \label{eq:ibp-lb} \\[2pt]
  \Ul_j &= [\wl_j]^{+}\, \hat{u}^{(\ell-1)}
          + [\wl_j]^{-}\, \hat{l}^{(\ell-1)}
          + \bl_j,
  \label{eq:ibp-ub}
\end{align}
\end{subequations}
where the operators $[v]^{+} = \max(v, 0)$ and $[v]^{-} = \min(v, 0)$ are applied
element-wise.
The ReLU function output is nonnegative, and the post-ReLU bounds for hidden layers can be further tightened:
\begin{equation}\label{eq:ibp-post}
  \hat{l}^{(\ell)}_j = \max(\Ll_j, 0), \qquad
  \hat{u}^{(\ell)}_j = \max(\Ul_j, 0).
\end{equation}

Interval arithmetic methods do not provide the tightest valid bounds in general, as dependencies between the input nodes are ignored. Propagating the resulting over-approximated bounds through the layers of a neural network leads to increasingly large over-approximations; in other words, propagating weak bounds through layers results in a model with significantly weaker continuous relaxation. 
Tighter bounds could potentially be obtained using optimization-based bound tightening (OBBT), i.e., solving an optimization problem with the objective set to minimize/maximize a particular pre-activation term to provide its bounds~\cite{badilla2023computational,zhao2024bound}. To reduce the computational cost of OBBT problems, OBBT can be performed using relaxations or problem-based decompositions~\cite{sosnin2024scaling}. In contrast to interval arithmetic, bounds obtained using OBBT can incorporate variable dependencies. 
In this work, we focus on IBP bounds given their popularity. 

A key observation is that the operations in the IBP recursion~\eqref{eq:ibp}--\eqref{eq:ibp-post} are compositions of affine maps and element-wise $\max(\cdot, 0)$, similar to the ReLU forward pass. The IBP process is therefore subdifferentiable with respect to the NN parameters~$\theta$.
The subgradients are well-defined almost everywhere and can be computed by automatic differentiation (e.g., in PyTorch or JAX).
This property enables the IBP bounds to be incorporated directly into gradient-based training as differentiable regularization terms, as developed in the following sections.

\section{MILP-informed regularization}
\label{sec:regularizers}
We now introduce regularizers $\mathcal{R}(\theta)$ for use in the training objective~\eqref{eq:loss} that target MILP tractability. 
For instance, Figure~\ref{fig:relaxations_3d} shows the predictions and LP lower bounds for trained NN models on the Peaks function. 
We begin with standard shrinkage penalties, which serve as baselines, and then present two IBP-based regularizers that directly target the mechanisms governing MILP difficulty. 

\subsection{Shrinkage regularization}
Shrinkage regularization is a strategy that aims to improve model generalizability and reduce overfitting by penalizing large parameter values, effectively `shrinking' them towards zero. Shrinking the parameter values manages the bias-variance trade-off by introducing a small amount of bias to (significantly) reduce model variance. Common methods here include Ridge (L2) and Lasso (L1) regression. 
These methods may produce models with tighter bounds, as shrinking the values of $\wl_j$ can directly improve the bounds obtained using \eqref{eq:ibp}. 
\citet{plate2026analysis} find that increasing shrinkage regularization can produce neural networks with a lower number of linear regions, improving performance in downstream MILP.
\citet{manngaard2018structural} study methods using these regularizers to explicitly induce weight sparsity. 

\paragraph{L1 regularization}
\begin{equation}\label{eq:l1}
  \mathcal{R}_{\mathrm{L1}}(\theta)
  = \sum_{\ell=1}^{L+1}
    \bigl(\lVert \wl \rVert_1 + \lVert \bl \rVert_1\bigr),
\end{equation}
where $\lVert \cdot \rVert_1$ denotes the entry-wise $\ell_1$ norm.
This promotes weight sparsity and indirectly reduces IBP bound widths, as $\Ul_j - \Ll_j$ scales with $\lVert \wl_j \rVert_1$ (see~\eqref{eq:bw-width} below).  
However, it does not account for the layered, recursive structure
of bound propagation and treats all parameters uniformly regardless
of their role in the MILP formulation.

\paragraph{L2 regularization}
\begin{equation}\label{eq:l2}
  \mathcal{R}_{\mathrm{L2}}(\theta)
  = \sum_{\ell=1}^{L+1}
    \bigl(\lVert \wl \rVert_F^2 + \lVert \bl \rVert_2^2\bigr),
\end{equation}
where $\lVert \cdot \rVert_F$ denotes the Frobenius norm.
Again, this can indirectly shrink big-M values but does not directly target bound widths or neuron stability.

\subsection{Bound-width regularization}
\label{sec:bw}

Define the \emph{width} of the IBP pre-activation bound at neuron $(\ell, j)$ as
$\Delta^{(\ell)}_j = \Ul_j - \Ll_j$.
We introduce a bound-width regularizer, which simply penalizes the mean bound width obtained across all hidden neurons:
\begin{equation}\label{eq:bw}
  \mathcal{R}_{\mathrm{BW}}(\theta)
  \;=\;
  \frac{1}{\sum n_\ell} \sum_{\ell=1}^{L}\sum_{j=1}^{n_\ell}
    \Delta^{(\ell)}_j
  \;=\;
  \frac{1}{\sum n_\ell}  \sum_{\ell=1}^{L}\sum_{j=1}^{n_\ell}
    \bigl(\Ul_j - \Ll_j\bigr).
\end{equation}

Subtracting~\eqref{eq:ibp-lb} from \eqref{eq:ibp-ub} gives the bound width at layer $\ell$ as:
\begin{equation}\label{eq:bw-width}
  \Delta^{(\ell)}_j
  \;=\;
  \bigl([\wl_j]^{+} - [\wl_j]^{-}\bigr)
  \bigl(\hat{u}^{(\ell-1)} - \hat{l}^{(\ell-1)}\bigr)
  \;=\;
  |\wl_j|\,\Delta^{(\ell-1)}_{\mathrm{post}},
\end{equation}
where $|\wl_j|$ denotes the element-wise absolute value of the $j$-th row, and $\Delta^{(\ell-1)}_{\mathrm{post}}$ is the vector of post-ReLU bound widths at layer $\ell{-}1$.
The post-ReLU bound widths satisfy $\Delta^{(\ell)}_{\mathrm{post},j} \le \Delta^{(\ell)}_j$ by~\eqref{eq:ibp-post}, so the layer-wise recursion~\eqref{eq:bw-width} shows that bound widths compound multiplicatively through the network depth.

\paragraph{Gradient.}
Since the post-ReLU bound widths $\Delta^{(\ell-1)}_{\mathrm{post}}$
depend in turn on earlier layers through the recursion~\eqref{eq:ibp}--\eqref{eq:ibp-post},
the total gradient $\partial \mathcal{R}_{\mathrm{BW}} / \partial \theta$
captures the full chain of IBP bound propagation through the NN.
In our experiments we directly implement \eqref{eq:bw-width} in PyTorch, and its gradient is computed automatically by PyTorch's reverse-mode automatic differentiation. 
Note that this requires implementing the (subdifferentiable) `IBP forward pass,' i.e., propagating bounds through the layers of the neural network using~\eqref{eq:ibp}--\eqref{eq:ibp-post}. 
The computational cost is one IBP forward pass per training step. 
Note that more advanced OBBT propagation schemes may be incorporated as regularizers following the differentiable optimization procedures in Section~\ref{sec:lp}. 

\paragraph{Interpretation.}
Including $\mathcal{R}_{\mathrm{BW}}$ as a regularization term in \eqref{eq:loss} explicitly penalizes the magnitude of big-M constants in the downstream MILP formulation. 
For an unstable neuron, we observe that  $|\Ll_j| \le \Delta^{(\ell)}_j$ and
$\Ul_j \le \Delta^{(\ell)}_j$.
Reducing $\Delta^{(\ell)}_j$ therefore simultaneously shrinks both big-M values in \eqref{eq:bigm-ub1}--\eqref{eq:bigm-ub2}, tightening the LP relaxation. 
When the bounds are both positive or negative, the neuron is stable, and no binary variable is required (Definition~\ref{def:stable}). 

\begin{remark}
An alternative view of $\mathcal{R}_{\mathrm{BW}}$ is the direct penalization of the magnitude of big-M constants, i. e., the product of weight magnitudes and input bound ranges. In this case, this is exactly represented by~\eqref{eq:bw-width}, as IBP bound widths (composed recursively through the layers) are precisely the big-M constants. Nevertheless, in more sophisticated MILP formulations without big-M constants, corresponding regularization terms can still be derived based on the width of the involved bounds. 
\end{remark}

\subsection{Stability regularization}
\label{sec:sn}

While the inclusion of $\mathcal{R}_{\mathrm{BW}}$ can strengthen the continuous LP relaxation by tightening bounds involved, the \emph{combinatorial} difficulty of the MILP is governed by the number of binary variables, and therefore unstable neurons $|\mathcal{U}|$ in \eqref{eq:num_unstable}. 
Both the relaxation tightness and the number of discrete combinations in a search tree impact the efficiency of branch-and-bound algorithms. 
As given in Definition~\ref{def:stable}, neuron $(\ell, j)$ is unstable when its pre-activation bounds straddle zero: $\Ll_j < 0 < \Ul_j$. 
A naive approach could directly penalize the number of unstable neurons $|\mathcal{U}|$.

Nevertheless, knowing bounds $\Ll_j$ and $\Ul_j$ also informs us how `close' a neuron is to being stable, e.g., how close the bounds are to zero. 
Based on this idea, we introduce a regularization term that penalizes the mean ``distance to stability'' to encourage stable nodes during training:
\begin{equation}\label{eq:sn}
  \mathcal{R}_{\mathrm{SN}}(\theta)
  \;=\;
  \frac{1}{\sum n_\ell}  \sum_{\ell=1}^{L}\sum_{j=1}^{n_\ell}
    \min\!\bigl([-\Ll_j]^{+},\; [\Ul_j]^{+}\bigr),
\end{equation}
where $[v]^{+} = \max(v, 0)$.
For a stable neuron ($\Ll_j \ge 0$ or $\Ul_j \le 0$), at least one of
$[-\Ll_j]^{+}$ or $[\Ul_j]^{+}$ is zero, so the contribution to $\mathcal{R}_{\mathrm{SN}}$ is zero.
For an unstable neuron, $[-\Ll_j]^{+} = |\Ll_j|$ and
$[\Ul_j]^{+} = \Ul_j$, and the contribution to $\mathcal{R}_{\mathrm{SN}}$ is
$\min(|\Ll_j|, \Ul_j) > 0$.
In other words, the regularizer pushes either $\Ll_j$ upward toward zero (making the neuron
stably active) or $\Ul_j$ downward toward zero (making it stably inactive), whichever requires the smaller change. 
We note that, even if this does not force the neuron to be stable, pusing one of the bounds closer to zero may still produce a tighter continuous relaxation (Figure~\ref{fig:regularizers}).

\begin{proposition}[Subgradient of~$\mathcal{R}_{\mathrm{SN}}$]\label{prop:sn-grad}
The subgradient of~\eqref{eq:sn} with respect to~$\theta$ is
\begin{equation}\label{eq:sn-grad}
  \frac{\partial \mathcal{R}_{\mathrm{SN}}}{\partial \theta}
  \;=\;
  \sum_{(\ell,j) \in \mathcal{U}}
  \begin{cases}
    \displaystyle
    -\frac{\partial \Ll_j}{\partial \theta},
      & \text{if } |\Ll_j| < \Ul_j, \\[8pt]
    \displaystyle
    \phantom{-}\frac{\partial \Ul_j}{\partial \theta},
      & \text{if } \Ul_j < |\Ll_j|,
  \end{cases}
\end{equation}
where $\partial \Ll_j / \partial \theta$ and
$\partial \Ul_j / \partial \theta$ are obtained from automatic
differentiation through the IBP recursion. 
At the non-differentiable point $|\Ll_j| = \Ul_j$, any convex
combination of the two cases is a valid subgradient.
\end{proposition}

In our implementation, we use the PyTorch \texttt{torch.minimum} function, which handles the subgradient at the tie point $|\Ll_j| = \Ul_j$ automatically. 

A related line of work aims to produce networks that are not only robust, but also \emph{easy to verify exactly} using MILP-based solvers.
\citet{xiao2019training} identify weight sparsity and ReLU stability as two network properties that reduce exact verification time.
They employ $\ell_1$ regularization and small-weight pruning to
promote sparsity, and introduce an alternative regularizer (termed RS loss in \cite{xiao2019training}) targeting ReLU stability. We denote this as an alternative stability regularizer:
\begin{equation}\label{eq:rs-loss}
  \mathcal{R}_{\mathrm{SN2}}
  \;=\;
  \frac{1}{\sum n_\ell}  \sum_{\ell=1}^{L}\sum_{j=1}^{n_\ell}
  -\tanh\!\bigl(1 + \Ul_j \cdot \Ll_j \bigr),
\end{equation}
where $\Ul_j$ and $\Ll_j$ are again upper
and lower bounds on the pre-activation of neuron~$(\ell,j)$.
When both bounds have the same sign (stable neuron), the
product~$\Ul_j \cdot \Ll_j$ is positive and the penalty is small;
when the bounds straddle zero (unstable neuron), the product is
negative and the penalty increases.
The authors found that adding this regularizer to the adversarial training objective reduces unstable neuron counts and yields considerable speedups in MILP-based verification time.

The stability regularizer $\mathcal{R}_{\mathrm{SN}}$~\eqref{eq:sn} is conceptually related to the RS Loss~\eqref{eq:rs-loss} of~\citet{xiao2019training}: both encourage neurons to be stably active or inactive.
Nevertheless, the formulations differ practically. Our proposed regularizer $\mathcal{R}_{\mathrm{SN}}$ in \eqref{eq:sn} uses $\min([-L]^+, [U]^+)$, to directly measures the distance to stability and has a piecewise-linear structure. On the other hand, the RS Loss in \eqref{eq:rs-loss}, denoted here as $\mathcal{R}_{\mathrm{SN2}}$, uses a smooth surrogate for sign agreement, which does not account for distance to stability.
Moreover, in our setting $\mathcal{R}_{\mathrm{SN}}$ and $\mathcal{R}_{\mathrm{SN2}}$ can be combined with other regularizers to target complementary aspects of MILP difficulty.

\paragraph{Interpretation.}
The regularizers $\mathcal{R}_{\mathrm{SN}}$ and $\mathcal{R}_{\mathrm{BW}}$ target
different mechanisms of MILP.
The former $\mathcal{R}_{\mathrm{BW}}$ shrinks all bound widths uniformly, improving the strength of the continuous LP relaxation.
On the other hand, $\mathcal{R}_{\mathrm{SN}}$ concentrates its effect on the boundary
at zero, aiming to eliminate (fix) binary variables from the formulation entirely.
A trained NN could have tight bounds (small $\Delta^{(\ell)}_j$) that still
straddle zero on many neurons, or wide bounds that happen to be one-sided (i.e., stable neurons).
The two regularizers address related and complementary aspects of MILP difficulty and can be combined.

\section{Relaxation-informed regularization}
\label{sec:lp}

The regularization methods introduced in Section~\ref{sec:regularizers} all may help improve the strength of MILP reformulations of a NN surrogate model, albeit indirectly. In other words, they are heuristics aimed at producing tighter bounds or reducing the number of binary variables. Figure~\ref{fig:regularizers} illustrates the relaxation-related properties targeted by each regularizer. 
In this section, we consider a direct measure of relaxation quality: the LP relaxation gap itself.

For a given input $x_i \in \mathcal{X}$, the true network output is
$f_\theta(x_i)$, which is uniquely determined by the forward pass~\eqref{eq:preact}--\eqref{eq:relu}. 
This unique solution is exactly encoded (for fixed input $x_i$) by the MILP constraints~\eqref{eq:bigm} when integrality is enforced. 
The LP relaxation, however, admits different output values because the relaxed binary variables $\al_j \in [0,1]$ allow intermediate neurons to deviate from their true ReLU outputs.

\subsection{Pointwise LP relaxation gap}
As mentioned above, the continuous relaxation of~\eqref{eq:bigm} obtained by relaxing integral constraints yields an LP that effectively provides a bound on the MILP optimum. 
We now derive a regularizer using the LP relaxation gap, i.e., the difference between the LP bound and the MILP solution (which gives the true NN output $f_\theta$). 
For a fixed NN input $x_i$, we denote the LP relaxation value:
\begin{equation}\label{eq:lp-relax}
\begin{aligned}
  V^{\min}(\theta, x_i) \;=\;
  \min_{\zl,\,\xl,\,\al} \quad
    & z^{(L+1)} \\
  \text{s.t.} \quad
    & \eqref{eq:preact},\;\eqref{eq:bigm}, \\
    & \al_j \in [0,1], \\
    & \hat{x}^{(0)} = x_i.
\end{aligned}
\end{equation}

\begin{figure}[t]
\centering
    \includegraphics[width=0.8\textwidth]{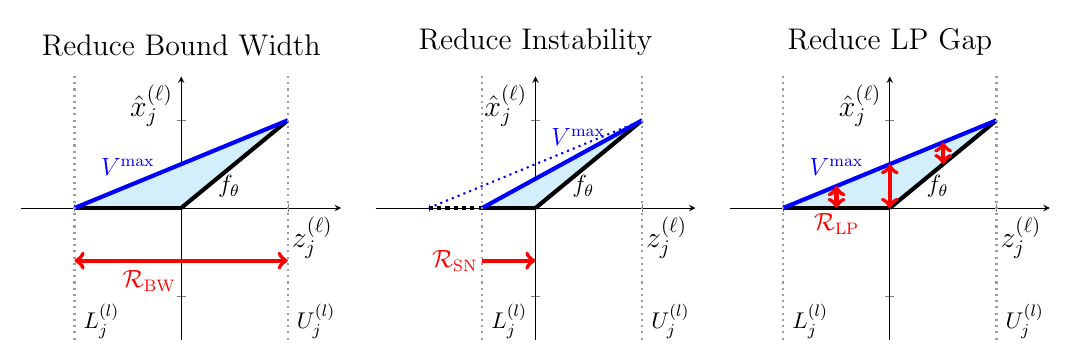}
    \caption{Conceptual depiction of the various goals of MILP-related regularization.}
    \label{fig:regularizers}
\end{figure}

The \emph{pointwise LP gap} in the minimization direction is:
\begin{equation}\label{eq:lp-gap-sample}
  \delta^{\min}(\theta, x_i)
  \;=\;
  f_\theta(x_i) - V^{\min}(\theta, x_i)
  \;\ge\; 0,
\end{equation}
since the LP relaxation can only under-estimate the minimum, i.e.,
$V^{\min} \le f_\theta(x_i)$. 
An analogous quantity $\delta^{\max}(\theta,x_i) = V^{\max}(\theta, x_i) - f_\theta(x_i) \ge 0$
measures the gap in the maximization direction.
The LP gap regularizer penalizes the average gap over a mini-batch~$B$:
\begin{equation}\label{eq:lp-gap-reg}
  \mathcal{R}_{\mathrm{LP}}(\theta)
  \;=\;
  \frac{1}{|B_s|}\sum_{i \in B_s} \delta(\theta,x_i),
\end{equation}
where $B_s \subseteq B$ is an optional subsample of the training mini-batch to limit the number of LP solves per training step. 
In practice, we find that even $|B_s|=1$ can achieve the desired effect. 
We use $\delta_i$ to denote $\delta^{\min}_i$, $\delta^{\max}_i$, or their
sum (total LP gap), depending on the optimization context.
For example, in surrogate-based problems where the NN output must be minimized,  penalizing $\delta^{\min}_i$ is the natural choice, as it targets the gap relevant to the downstream MILP objective.

We note that in some settings surrogate models can have multiple output neurons, e.g., classification models or quantile neural networks~\cite{alcantara2025quantile,ghilardi2025integrated}, complicating the definition of the LP relaxation value~\eqref{eq:lp-relax}. For these models we propose quantifying the LP relaxation gap for a surrogate objective by projecting the vector of outputs $z^{(L+1)}$ onto a random vector, analogous to stochastic Sobolev training~\cite{tsay2021sobolev}. Following this approach, the objective function for~\eqref{eq:lp-relax} is replaced by $\omega^\top z^{(L+1)}$, where $\omega$ is a normalized randomly sampled vector. 
Averaging over many mini-batches would naturally encourage tightening in all possible output directions. 

\paragraph{Interpretation}
$\mathcal{R}_{\mathrm{LP}}$ measures relaxation looseness at individual
  training points $x_i$, while the global LP gap
  (minimizing/maximizing over all $x \in \mathcal{X}$) measures the
  worst-case looseness over the domain $\mathcal{X}$.
  Including pointwise estimates at many training points is expected to generally
  tighten the relaxation over regions of the search space, e.g., sub-domains of a branch-and-bound search. Intuitively, the global relaxation may be tightened as well, though this is not guaranteed.

\subsection{Differentiating through the LP solution}
\label{sec:envelope}

Computing the gradient
$\partial \mathcal{R}_{\mathrm{LP}} / \partial \theta$ requires
differentiating the solution to the LP~\eqref{eq:lp-relax} with
respect to the network parameters~$\theta$.
The LP is a parametric linear program whose constraint data depend
on~$\theta$. 
Writing this LP~\eqref{eq:lp-relax} in standard form:
\begin{equation}\label{eq:lp-standard}
\begin{aligned}
  V^{\min}(\theta, x_i) \;=\; \min_{y} \quad
    & c^\top y \\[4pt]
  \text{s.t.} \quad
    & A_{\mathrm{eq}}(\theta)\, y = b_{\mathrm{eq}}(\theta, x_i), \\
    & G(\theta) y \le h(\theta), \\
    & y^{\mathrm{lb}} \le y \le y^{\mathrm{ub}},
\end{aligned}
\end{equation}
where $y$ collects all primal variables $(\zl, \xl, \al)$ across NN layers, $c$ is the objective vector (in this case selecting the output neuron), the equality constraints encode the pre-activation definitions~\eqref{eq:preact}, and the inequality constraints encode the big-M constraints~\eqref{eq:bigm} with $\al_j \in [0,1]$.

At the LP solution, let $\nu^* \in \R^{m_{\mathrm{eq}}}$ and $\mu^* \in \R^{m_{\mathrm{ineq}}}$ denote the optimal dual variables for the equality and inequality constraints, respectively, with $\mu^*\geq 0$.

\begin{proposition}[Sensitivity for parametric LP]\label{prop:envelope}
Suppose the LP~\eqref{eq:lp-standard} has a unique, non-degenerate
optimal basis.
Then the optimal value $V^{\min}$ is differentiable with respect
to~$\theta$, and
\begin{equation}\label{eq:envelope}
  \frac{\partial V^{\min}}{\partial \theta}
  \;=\;
  (\nu^*)^\top \frac{\partial b_{\mathrm{eq}}}{\partial \theta}
  \;-\;
  (\nu^*)^\top \frac{\partial A_{\mathrm{eq}}}{\partial \theta}\, y^*
  \;+\;
  (\mu^*)^\top \frac{\partial h}{\partial \theta}
  \;-\;
  (\mu^*)^\top \frac{\partial G}{\partial \theta}\, y^*.
\end{equation}
\end{proposition}

\begin{proof}
To obtain these derivatives, we follow the approach of \cite{amos2017optnet} and differentiate the KKT conditions. A similar analysis is also provided in~~\citet[Chapter 3.4]{fiacco1983}. In particular, the Lagrangian of~\eqref{eq:lp-standard} is given by:
\begin{equation}
L = c^\top y + \nu^\top\left( A_{\mathrm{eq}}(\theta)\, y - b_{\mathrm{eq}}(\theta, x_i)\right) + \mu^\top\left(G(\theta) y - h(\theta)\right).
\end{equation}
The KKT conditions for stationarity, primal feasibility, and complementary slackness are:
\begin{equation}\label{eq:kkt}
    \begin{aligned}
        c + A_{\mathrm{eq}}(\theta)^\top \nu^* + G(\theta)^\top \mu^* &= 0,\\
        A_{\mathrm{eq}}(\theta)\, y^* - b_{\mathrm{eq}}(\theta, x_i) &= 0,\\
        D(\mu^*) \left(G(\theta) y^* - h(\theta)\right) &= 0,
    \end{aligned}
\end{equation}
where the $D(\cdot)$ operator forms a diagonal matrix from a vector. 
To obtain a derivative, we assume (or approximate) the active-constraint set is locally constant, i.e., at a non-degenerate optimal basis, so the solution $[y^*(\theta), \nu^*(\theta), \mu^*(\theta)]$ is a smooth function of~$\theta$ by the implicit function theorem applied to~\eqref{eq:kkt}. We refer the reader to~\citet[Chapter 2.4]{fiacco1983} for an overview of relevant implicit function theorem results.  
Since the objective vector $c$ is fixed, we can first differentiate the objective $V^{\min} = c^\top y^*$, giving
\begin{equation}\label{eq:proof-step1}
  \frac{\partial V^{\min}}{\partial \theta}
  \;=\; c^\top \frac{\partial y^*}{\partial \theta}.
\end{equation}
We then substitute the stationarity condition from \eqref{eq:kkt}, giving
\begin{equation}\label{eq:proof-step2}
  \frac{\partial V^{\min}}{\partial \theta}
  \;=\; (\nu^*)^\top A_{\mathrm{eq}}\,\frac{\partial y^*}{\partial \theta}
  \;+\; (\mu^*)^\top G\,\frac{\partial y^*}{\partial \theta}.
\end{equation}
Now, differentiating the primal feasibility conditions,
$A_{\mathrm{eq}}(\theta)\,y^*(\theta) = b_{\mathrm{eq}}(\theta, x_i)$, gives:
\begin{equation}\label{eq:proof-eq}
  A_{\mathrm{eq}}\,\frac{\partial y^*}{\partial \theta} \;+\; \frac{\partial A_{\mathrm{eq}}}{\partial \theta}\,y^*
  \;=\; \frac{\partial b_{\mathrm{eq}}}{\partial \theta}.
\end{equation}
For the inequality constraints, complementary slackness \eqref{eq:kkt} gives $\mu^*_i > 0$ only when $G_i(\theta)\,y^* = h_i(\theta)$, i.e., the $i$-th constraint is active. 
Differentiating the active inequality constraints (noting that $\mu^*_i = 0$ for inactive constraints) therefore gives:
\begin{equation}\label{eq:proof-ineq}
  (\mu^*)^\top G\,\frac{\partial y^*}{\partial \theta} \;+\; (\mu^*)^\top\frac{\partial G}{\partial \theta}\,y^*
  \;=\; (\mu^*)^\top\frac{\partial h}{\partial \theta}.
\end{equation}
Substituting~\eqref{eq:proof-eq} and \eqref{eq:proof-ineq}
into~\eqref{eq:proof-step2} to eliminate $A_{eq}\cdot \partial y^*/\partial\theta$ and $G \cdot\partial y^*/\partial\theta$ yields~\eqref{eq:envelope}.
\end{proof}

The result applies the familiar LP shadow-price interpretation (the rate of change of the optimum with respect to the right-hand side equals the dual variable) to perturbations in both the constraint matrix and the inequality right-hand side. 
In other words, the sensitivity of the optimal value to perturbations in constraint data can be computed from the optimal dual variables, without requiring differentiation through the \texttt{min} operator itself. 
The dual variables $\nu^*$ and $\mu^*$ are a standard output of LP solvers (e.g., as shadow prices from HiGHS or Gurobi). 
We refer the reader to~\cite[Chapter 5]{bertsimas1997introduction} for a more comprehensive treatment of parametric LPs and global LP sensitivity analysis. 

While Proposition~\ref{prop:envelope} gives a simple avenue to obtain sensitivities for simple, LP-based relaxations, more complicated formulations may also be used to produce bounds, e.g., convex NLP relaxations. 
Note that recent works~\cite{schweidtmann2019deterministic,wilhelm2023convex} study relaxations for nonlinear activation functions, another direction for future generalization beyond LP relaxations. 
The proposed regularizer may be generalized to these settings, e.g., following approaches to differentiate through nonlinear programs~\cite{agrawal2019differentiable,amos2017optnet,pineda2022theseus}. 
Moreover, traditional envelope theorems describe conditions for the value of a parameterized (nonlinear) optimization problem to be differentiable in the parameter and provide formulas for their derivatives~\cite{fiacco1983,milgrom2002}.
We note there is also a growing literature on software frameworks for differentiable optimization~\cite{besanccon2024flexible,rosemberg2025general}. 

\begin{remark}[Envelope theorem versus KKT differentiation]
\label{rem:kkt} 
An alternative route to $\partial V^{\min}/\partial\theta$ is to
differentiate the KKT stationarity conditions implicitly.
At the optimal basis, differentiating the stationarity conditions~\eqref{eq:kkt} with respect to~$\theta$ yields a linear system involving the Jacobian $\partial y^*/\partial\theta$ of the
optimal primal solution.
Computing this Jacobian requires solving an $n_y \times |\theta|$ system at every training step, where $n_y$ is the number of primal LP variables (scaling with network width and depth), and $|\theta|$ is the number of network parameters.
The proposed formulation avoids this scaling entirely.
Because $\partial \mathcal{L}/\partial y = 0$ at optimality (primal stationarity), the terms involving $\partial y^*/\partial\theta$
cancel in the total derivative, and the gradient collapses to the
dual-weighted expression~\eqref{eq:envelope}.
In practice, we require only the dual variables~$\nu^*$ and~$\mu^*$, which are already a standard solver output, (with no additional linear system to solve).
\end{remark}

\subsection{Application to the big-M LP}
\label{sec:channels}

We observe that the LP constraints depend on the value of~$\theta$ through two channels, as written in~\eqref{eq:lp-standard}:
\begin{enumerate}
  \item \textbf{Equality constraints (direct):}
    The pre-activation definitions $\zl = \wl \xl[\ell-1] + \bl$ contribute $\bl$ to the
    right-hand side~$b_{\mathrm{eq}}$ and $\wl$ to the constraint
    matrix~$A_{\mathrm{eq}}$.
  \item \textbf{Inequality constraints (indirect):}
    The big-M values $\Ll_j, \Ul_j$ appearing in~\eqref{eq:bigm-ub1}--\eqref{eq:bigm-ub2} contribute to both $G$ and $h$, and their values depend on~$\theta$ indirectly. For example, the bounds may be computed through IBP recursion~\eqref{eq:ibp}.
\end{enumerate}

Noting the similarity of channel (2) to the bound widths discussed in Section~\ref{sec:regularizers}, in our implementation, we treat the big-M values $\Ll_j, \Ul_j$ as constants when differentiating through the LP, retaining only channel~(1). 
This simplification is further motivated in Section~\ref{sec:combined}, where we explicitly show that the omitted big-M sensitivity can be recovered by the bound-width regularizer when the two are combined. 

Following this simplification, $\Ll_j, \Ul_j$ are treated as constants during differentiation, and the inequality constraint data $G$ and $h$ are independent of~$\theta$. The sensitivity~\eqref{eq:envelope} therefore reduces to:
\begin{equation}\label{eq:envelope-eq-only}
  \frac{\partial V^{\min}}{\partial \theta}\bigg|_{L,U}
  \;=\;
  (\nu^*)^\top
  \frac{\partial b_{\mathrm{eq}}}{\partial \theta}
  \;-\;
  (\nu^*)^\top
  \frac{\partial A_{\mathrm{eq}}}{\partial \theta}\, y^*.
\end{equation}

The equality constraints can be grouped by layer for ease of notation.
At layer~$\ell$, the constraint 
$\zl_j - \wl_j \xl[\ell-1] = \bl_j$
has dual variable $\nu^{(\ell)}_j$.
Differentiating and plugging into \eqref{eq:envelope-eq-only} gives:
\begin{align}
  \frac{\partial V^{\min}}{\partial \bl_j}\bigg|_{L,U}
    &= \nu^{(\ell)}_j,
  \label{eq:grad-b} \\[6pt]
  \frac{\partial V^{\min}}{\partial W^{(\ell)}_{j,k}}\bigg|_{L,U}
    &= \nu^{(\ell)}_j \cdot \hat{x}^{(\ell-1)*}_k,
  \label{eq:grad-W}
\end{align}
where $\hat{x}^{(\ell-1)*}_k$ is the LP primal value of the
post-activation variable at layer~$\ell{-}1$
(for $\ell = 1$, these correspond to elements of the fixed input component $x_i$).

The gradient of the per-sample LP gap~\eqref{eq:lp-gap-sample} is
therefore approximated as:
\begin{equation}\label{eq:gap-grad}
  \frac{\partial \delta^{\min}_i}{\partial \theta}
  \;=\;
  \frac{\partial f_\theta(x_i)}{\partial \theta}
  \;-\;
  \frac{\partial V^{\min}}{\partial \theta}\bigg|_{L,U},
\end{equation}
where the first term is the standard backpropagation gradient.

\subsection{Proxy implementation}\label{sec:straight-through}

Rather than implementing a custom backward pass for optimization problems as in~\cite{amos2017optnet}, we construct a differentiable proxy tensor that has a gradient matching~\eqref{eq:grad-b}--\eqref{eq:grad-W}:
\begin{equation}\label{eq:proxy}
  P(\theta)
  \;=\;
  \sum_{\ell=1}^{L+1}
    \bigl(\nu^{(\ell)}\bigr)^\top
    \bigl(\wl \,\hat{x}^{(\ell-1)*} + \bl\bigr),
\end{equation}
where $\nu^{(\ell)}$ and $\hat{x}^{(\ell-1)*}$ are treated as fixed constants (detached from the computation graph) obtained from the LP solution, while the network parameters $\wl$ and $\bl$ remain in the computation graph. 
Observe that, by construction, the derivatives $\partial P / \partial \theta$
reproduce~\eqref{eq:grad-b}--\eqref{eq:grad-W} exactly.
Nevertheless, the forward-pass value of $P(\theta)$ does not match the true LP value $V^{\min}$, which we would like to include in the regularizer \eqref{eq:lp-gap-reg}. 
Therefore, we apply the idea of a `straight-through
estimator,' i.e., a proxy derivative that is used in the backward pass only~\cite{bengio2013estimating,yin2019understanding}:
\begin{equation}\label{eq:st}
  \widetilde{V}^{\min}(\theta)
  \;=\;
  P(\theta) - \operatorname{sg}\bigl[P(\theta)\bigr] + V^{\min},
\end{equation}
where $\operatorname{sg}[\cdot]$ denotes the \emph{stop-gradient
operator}. 
Specifically, $\operatorname{sg}[u]$ returns the same numerical value
as~$u$, but is treated as a constant during differentiation, i.e., 
$\partial\,\operatorname{sg}[u]/\partial\theta \equiv 0$.
In automatic differentiation frameworks this is implemented by
detaching the tensor from the computation graph (e.g.\
\texttt{u.detach()} in PyTorch).

The two passes of~\eqref{eq:st} behave differently by design.
In the \emph{forward pass}, $P(\theta)$ and
$\operatorname{sg}[P(\theta)]$ evaluate to the same scalar~$p$, so
$\widetilde{V}^{\min} = p - p + V^{\min} = V^{\min}$. In other words, the forward pass returns the desired LP optimal value from the solver.
In the \emph{backward pass}, the stop-gradient removes the second term
and the solver output $V^{\min}$ is a constant (solving the relaxation using an LP solver is not included in the computation graph), giving $\partial \widetilde{V}^{\min}/ \partial\theta = \partial P/\partial\theta - 0 + 0 = \partial P/\partial\theta$. 
In other words, the backward pass returns the desired gradient~\eqref{eq:grad-b}--\eqref{eq:grad-W}.
This proxy implementation avoids having to implement a custom backward pass for the proposed regularizer, while preserving both the correct function value and the correct gradient.

Figure~\ref{fig:relaxations_gaps} illustrates the pointwise LP relaxation gap $\delta^{\min}(\theta, x_i)$ in \eqref{eq:lp-gap-sample} for NN models trained with the various regularizers on a simple benchmark function. We observe that the LP gap regularizer can produce surrogate models with much tighter pointwise relaxations over the function domain.

\begin{figure}
\centering
    \includegraphics[width=0.9\textwidth]{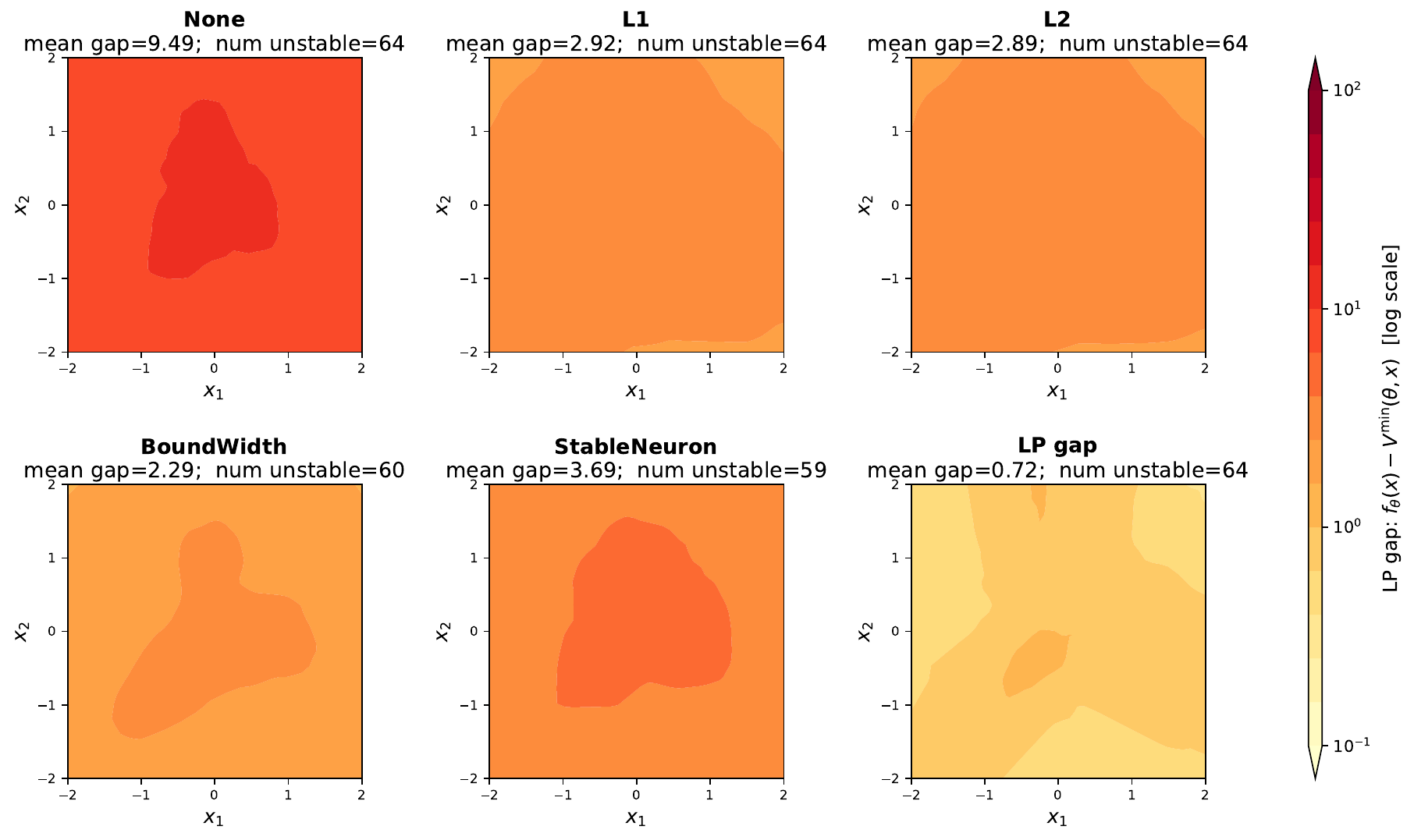}
    \caption{LP gap for NN models with \{32,32\} hidden layers trained on the scaled Peaks function with two inputs ($x_1$,$x_2$), with output $f(x)$. Different regularizers are applied during training, with weights chosen to maintain validation MSE of similar scale.}
    \label{fig:relaxations_gaps}
\end{figure}

\section{Combining regularization strategies}
\label{sec:combined}

Sections~\ref{sec:regularizers} and~\ref{sec:lp} introduce several regularization strategies that target different aspects of MILP difficulty when used downstream as surrogate models (Figure~\ref{fig:regularizers}). 
A summary of the various proposed regularizers is given in Table~\ref{tab:summary}. 
Computational costs for the various regularizers are given in Table~\ref{tab:cost}. 

\begin{table}[h]
\centering
\caption{Summary of MILP- and relaxation-informed regularization strategies.}
\label{tab:summary}
\begin{tabular}{@{}p{2.2cm}p{4.5cm}p{5cm}@{}}
\toprule
Regularizer & What it targets & Gradient w.r.t.\ $\theta$ \\
\midrule
$\mathcal{R}_{\mathrm{BW}}$
  & Bound widths (big-M values) 
  & $dL/d\theta,\; dU/d\theta$ via autodiff \\[4pt]
$\mathcal{R}_{\mathrm{SN}}$
  & Number of binary variables \newline (combinatorial difficulty)
  & $dL/d\theta$ or $dU/d\theta$ via autodiff \newline (for unstable neurons)\\[4pt]
$\mathcal{R}_{\mathrm{LP}}$
  & LP relaxation gap \newline ($L, U$ treated as constants)
  & LP dual variables \\[4pt]
$\mathcal{R}_{\mathrm{BW}} + \mathcal{R}_{\mathrm{LP}}$
  & Both big-M values and LP gap
  & Approximates full $dV^{\max}/d\theta$ \newline (see Proposition~\ref{prop:combined}) \\
\bottomrule
\end{tabular}
\end{table}

\begin{table}[h]
\centering
\caption{Per-step computational cost of each regularizer.}
\label{tab:cost}
\begin{tabular}{@{}lll@{}}
\toprule
regularizer & Extra cost per training step & Gradient source \\
\midrule
$\mathcal{R}_{\mathrm{L1}},\; \mathcal{R}_{\mathrm{L2}}$
  & Negligible
  & Autograd \\
$\mathcal{R}_{\mathrm{BW}}$
  & 1 IBP forward pass
  & Autograd through IBP \\
$\mathcal{R}_{\mathrm{SN}}$, $\mathcal{R}_{\mathrm{SN2}}$
  & 1 IBP forward pass
  & Autograd through IBP \\
$\mathcal{R}_{\mathrm{LP}}$
  & $|B_s|$ LP solves & LP duals + autograd \\
$\mathcal{R}_{\mathrm{BW}} + \mathcal{R}_{\mathrm{LP}}$
  & 1 IBP pass + $|B_s|$ LP solves
  & Both channels of~\eqref{eq:total-deriv} \\
\bottomrule
\end{tabular}
\end{table}

Here $|B_s|$ is the number of LP samples per batch (a tunable parameter to manage training overhead).
The IBP forward pass has the same cost as a standard network forward pass (one matrix--vector product per layer).
The LP solves are the dominant cost for $\mathcal{R}_{\mathrm{LP}}$; they can be parallelized across samples and potentially accelerated by warm-starting from the previous iterate.

\subsection{Approximation of the total derivative}

As observed in Section~\ref{sec:channels}, the LP optimal value $V^{\max}$ depends on $\theta$ through two channels:
\begin{equation}\label{eq:total-deriv}
  \frac{dV^{\max}}{d\theta}
  \;=\;
  \underbrace{
    \frac{\partial V^{\max}}{\partial \theta}\bigg|_{L,U\,\text{const}}
  }_{\text{direct: constraint RHS}}
  \;+\;
  \underbrace{
    \sum_{(\ell,j)}
    \left(
      \frac{\partial V^{\max}}{\partial \Ll_j}
      \frac{d\Ll_j}{d\theta}
      +
      \frac{\partial V^{\max}}{\partial \Ul_j}
      \frac{d\Ul_j}{d\theta}
    \right)
  }_{\text{indirect: big-M values}}.
\end{equation}

The first term is what $\mathcal{R}_{\mathrm{LP}}$ effectively computes via~\eqref{eq:envelope-eq-only}, as the bound widths (big-M values) are assumed constant.
The second (omitted) term captures how changing $\theta$ alters the big-M constants $\Ll_j, \Ul_j$, which in turn affect the LP feasible region and hence the tightness of the LP relaxation.
This second term factors as:
\begin{itemize}
  \item $\partial V^{\max} / \partial \Ll_j$ and
    $\partial V^{\max} / \partial \Ul_j$:
    the sensitivity of the LP value to the big-M constants,
    given by the dual variables $\mu^*$ of the inequality
    constraints~\eqref{eq:bigm-ub1}--\eqref{eq:bigm-ub2};
  \item $d\Ll_j / d\theta$ and $d\Ul_j / d\theta$:
    the gradients of the bound w.r.t. $\theta$, e.g., obtained using IBP.
\end{itemize}

The bound-width regularizer $\mathcal{R}_{\mathrm{BW}}$ penalizes
$\sum_{(\ell,j)} (\Ul_j - \Ll_j)$, whose gradient is:
\begin{equation}\label{eq:bw-grad}
  \frac{\partial \mathcal{R}_{\mathrm{BW}}}{\partial \theta}
  \;=\;
  \sum_{(\ell,j)}
  \left(
    \frac{d\Ul_j}{d\theta} - \frac{d\Ll_j}{d\theta}
  \right).
\end{equation}
Comparing with the second (indirect) term in~\eqref{eq:total-deriv}, we
see that $\mathcal{R}_{\mathrm{BW}}$ effectively provides a
\emph{surrogate} for the indirect big-M sensitivity path, with the LP dual
multipliers
$\partial V^{\max}/\partial \Ll_j$ and
$\partial V^{\max}/\partial \Ul_j$
replaced by the uniform weights $-1$ and $+1$, respectively.

\begin{proposition}[Combining $\mathcal{R}_{\mathrm{BW}}$ and $\mathcal{R}_{\mathrm{LP}}$ regularizers approximates the full gradient]
\label{prop:combined}
The combined regularizer
$\mathcal{R}_{\mathrm{LP}} + \alpha\,\mathcal{R}_{\mathrm{BW}}$
produces the gradient:
\begin{equation}\label{eq:combined-grad}
  \frac{\partial \mathcal{R}_{\mathrm{LP}}}{\partial \theta}
  + \alpha \frac{\partial \mathcal{R}_{\mathrm{BW}}}{\partial \theta}
  \;=\;
  \underbrace{
    \frac{\partial V^{\max}}{\partial \theta}\bigg|_{L,U\,\text{const}}
  }_{\text{from } \mathcal{R}_{\mathrm{LP}}}
  \;+\;
  \alpha
  \underbrace{
    \sum_{(\ell,j)}
    \left(
      \frac{d\Ul_j}{d\theta} - \frac{d\Ll_j}{d\theta}
    \right)
  }_{\text{from } \mathcal{R}_{\mathrm{BW}}},
\end{equation}
which approximates the total derivative~\eqref{eq:total-deriv} with the
LP dual weights $\partial V^{\max}/\partial L_j^{(\ell)}$ and
$\partial V^{\max}/\partial U_j^{(\ell)}$ replaced by $\alpha$ and
$-\alpha$.
\end{proposition}

\begin{remark}[Why not differentiate through the big-M values directly?]
  Computing the exact second term in~\eqref{eq:total-deriv} would require the LP dual variables $\mu^*$ for the inequality constraints as well as the full IBP Jacobian $dL/d\theta, dU/d\theta$. 
  While feasible in principle, this doubles the information needed from each LP solve and couples the LP backward pass to the IBP backward pass.
  The combined $\mathcal{R}_{\mathrm{LP}} + \alpha\,\mathcal{R}_{\mathrm{BW}}$ avoids this coupling while still capturing both sensitivity paths, with $\alpha$ serving as a tunable proxy for the (unknown, sample-dependent) LP dual weights.
  The scalar $\alpha$ can be interpreted as a uniform ``importance weight'' for big-M tightness relative to constraint-RHS sensitivity.
\end{remark}

\section{Computational Results}
\label{sec:results}
To evaluate the regularization techniques proposed in Sections~\ref{sec:regularizers}--\ref{sec:lp}, we first consider the experimental settings of~\citet{plate2026analysis} and train NNs as surrogates for standard non-convex benchmark functions. We furthermore study quantile NNs as surrogates in stochastic programming applications, following~\citet{alcantara2025quantile}. We compare training and MILP performance on downstream optimization problems with different (combinations of) regularizers added during training.  

\subsection{Implementation}
All experiments were run on a server equipped with AMD EPYC 7742 64-Core Processors. Each training and optimization run was allocated 8 CPU cores and 16 GB of memory.  
NN surrogate models and regularizers were implemented using PyTorch~\cite{paszke2019pytorch}, and MILP optimization problems were solved using Gurobi v13.0.1~\cite{gurobi}. 
The LPs for the relaxation-based regularizer are implemented using \texttt{scipy.optimize} and solved using HiGHS~\cite{huangfu2018parallelizing}.
The author acknowledges the use of Anthropic's Claude (v4.6 models) to assist with setting-up the server experimental environments. The content was reviewed by the author, who takes full responsibility for the final manuscript.

Although Gurobi can solve LPs, we use a HiGHS implementation for two reasons: first, it avoids the per-call overhead of constructing Gurobi model objects inside each training batch, which dominates runtime for the relaxed LPs encountered; and second, it keeps the entire training pipeline within open-source Python dependencies following the convention of machine learning software. 
The LP instances encountered during training consist of one LP per regularized sample, with the number of variables and constraints growing linearly in the total number of neurons. 
We found that HiGHS solves each these LP in milliseconds, but the cost accumulates over many mini-batches, which is reflected in the computational costs reported in Table~\ref{tab:train_time_ratios}. 
While our experiments are limited to CPU servers, an interesting direction for future work is to exploit GPU-based LP solvers~\cite{applegate2021practical,applegate2023faster} during training, which could substantially reduce this overhead and integrate more naturally with GPU-based model training pipelines.

\subsection{Direct Optimization over Surrogates}
\begin{figure}
\centering
    \includegraphics[width=0.9\textwidth]{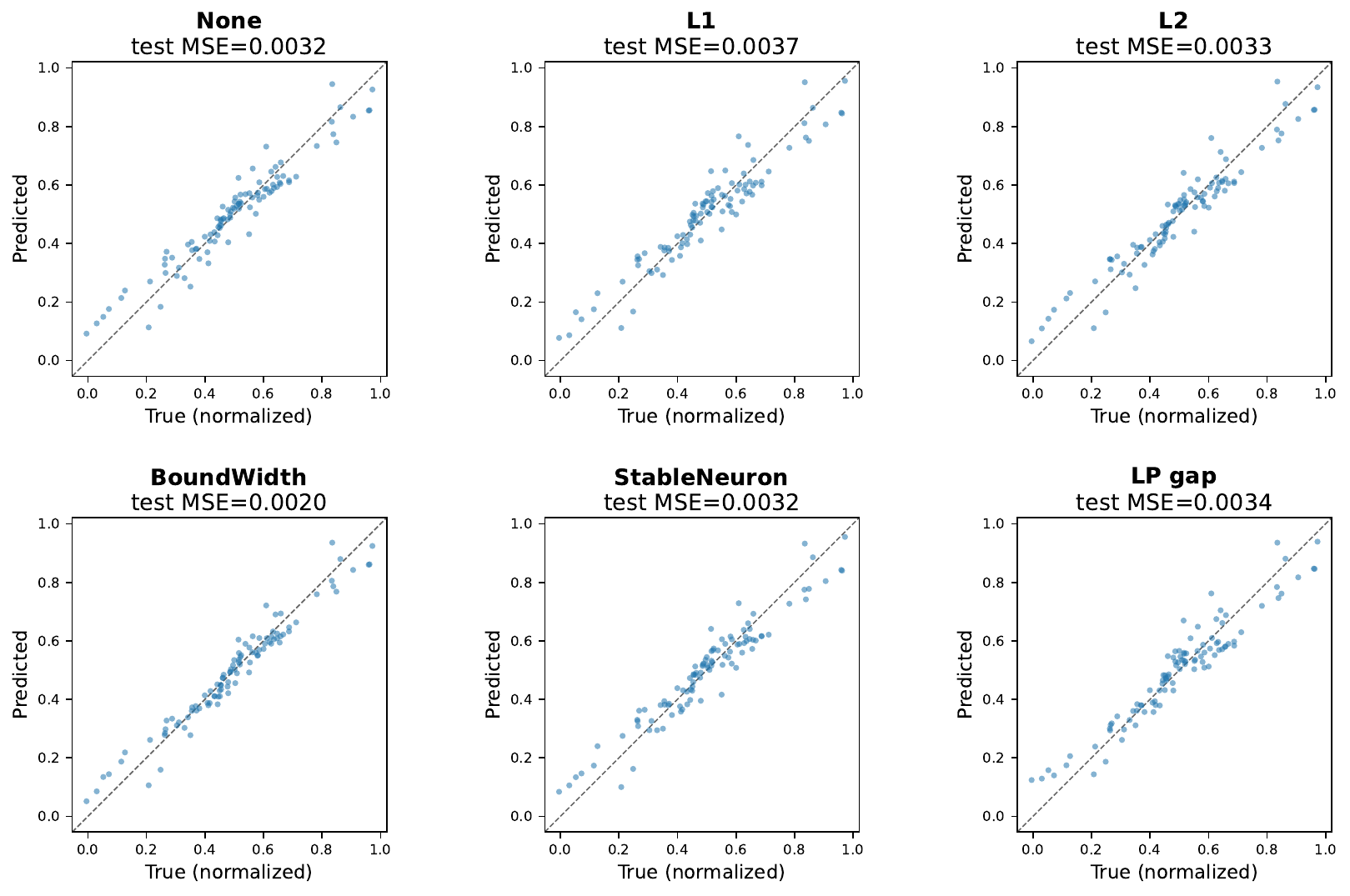}
    \caption{Parity plots for NN models with \{32,32\} hidden layers trained on the scaled Peaks function with two inputs ($x_1$,$x_2$), with output $f(x)$. Different regularizers are applied during training, with weights chosen to maintain validation MSE of similar scale.}
    \label{fig:relaxations_parity}
\end{figure}

\subsubsection{Benchmark Functions}
We first study the direct minimization over a surrogate model output, i.e., solving the straightforward problem to minimize $f_\theta(x)$. 
We consider the benchmark functions and training settings used by~\citet{plate2026analysis} to facilitate comparison:
\begin{enumerate}
    \item The Himmelblau function $f_\mathrm{himmelblau}:[-5,5]^2\rightarrow\R$, denoted as \texttt{himmelblau}, which has a global minimum of $0$ at four points:
    \[
    f_\mathrm{himmelblau}(x) = (x_1^2+x_2-11)^2 + (x_1+x_2^2-7)^2
    \]
    \item The Peaks function $f_\mathrm{peaks}:[-2,2]^2\rightarrow \R$, denoted as \texttt{peaks}, which is multimodal with a unique global minimum of $-6.551$~\citep{plate2026analysis}:
    \begin{multline*}
    f_\mathrm{peaks}(x) = -3(1-x_1)^2 \exp\left(-x_1^2-(x_2-1)^2\right) - \\10(\frac{x_1}{5}-x_1^3 - x_2^5)\exp(-x_1^2 - x_2^2) - \frac{1}{3}\exp\left(-(x_1+1)^2-x_2^2\right)
    \end{multline*}
    \item The $d$-dimensional Ackley function $f_\mathrm{ackley}, d : [-3.5,3.5]^d\rightarrow \R$, denoted as \texttt{ackley-d}. The function is multimodal with a unique global minimum of $0$:
    \begin{multline*}
    f_\mathrm{ackley}(x) = -20\, \exp\left(\frac{1}{5}\sqrt{\frac{1}{d} \sum_{i=1}^d x_i^2 }\right) - \exp\left( \frac{1}{d} \sum_{i=1}^d \mathrm{cos}\left(2\pi x_i \right) \right) + \exp(1) + 20
    \end{multline*}
\end{enumerate}

As an illustrative example, we first train feedforward neural networks with two hidden layers of 32 neurons each with the various proposed regularizers, tuning the regularization penalties to ensure a similar validation MSE. The continuous relaxations are shown in Figure~\ref{fig:relaxations_3d}, the pointwise relaxation gaps in Figure~\ref{fig:relaxations_gaps}, and prediction parity plots in Figure~\ref{fig:relaxations_parity}. This simple example allows us to visually verify that the proposed regularizers may be tuned to improve relaxation tightness without significantly worsening prediction accuracy and/or generalization ability. 

For the optimization studies below, we now partially follow the training setting of~\cite{plate2026analysis} and consider feedforward neural networks of with $\{2,3,5\}$ hidden layers of $25$ neurons each. The models are trained on 100,000 samples for \texttt{peaks} and \texttt{himmelblau} and 150,000 samples for \texttt{ackley-2}, with all samples generated using Latin Hypercube sampling. 
We test some larger models including hidden layers of $50$ neurons on the 5-dimensional Ackley function, \texttt{ackley-5}, where the number of samples is doubled to 300,000. 
Data are normalized, and 30\% of the data are used as a test set to measure  generalization ability. Networks are trained for 200 epochs using the Adam optimizer.

\subsubsection{Training Results}
\paragraph{Computational overhead.}
Table~\ref{tab:train_time_ratios} shows training-time ratios relative to the unregularized baseline.
Shrinkage regularizers ($\mathcal{R}_{\mathrm{L1}}$, $\mathcal{R}_{\mathrm{L2}}$) add modest overhead (generally 1--2$\times$), as they require only element-wise weight penalties.
The bound-width and stability regularizers $\mathcal{R}_{\mathrm{BW}}$, $\mathcal{R}_{\mathrm{SN}}$ and $\mathcal{R}_{\mathrm{SN2}}$ incur similar overheads to each other and slightly more than the shrinkage regularizers. We also observe their computational costs scale with network depth due to the layer-by-layer IBP propagation.
The LP-based regularizer $\mathcal{R}_{\mathrm{LP}}$ is the most expensive, at approximately 5--10$\times$, reflecting the cost of solving one full LP relaxation of the network per regularization sample; the combined regularizer $\mathcal{R}_{\mathrm{BW}}+\mathcal{R}_{\mathrm{LP}}$ similarly reflects computational costs dominated by the LP component.
Note these overheads are incurred only once during model training time and can potentially be amortized over usage in many downstream optimization instances. 

\begin{table}[h]
\caption{Training-time overhead of each regularizer relative to the unregularized baseline, averaged across \texttt{ackley-2}, \texttt{himmelblau}, \texttt{peaks}. Each cell shows $\bar{r} = \frac{1}{|\mathcal{B}|}\sum_{b \in \mathcal{B}} \frac{\bar{t}_{b,\mathrm{reg}}}{\bar{t}_{b,\mathrm{none}}}$, where $\bar{t}$ is the mean training time over seeds. }
\label{tab:train_time_ratios}
\centering
\begin{tabular}{ll  c  c  c}
\toprule
regularizer & $\lambda$ & \multicolumn{1}{c}{\texttt{2-25-25-1}} & \multicolumn{1}{c}{\texttt{2-25-25-25-1}} & \multicolumn{1}{c}{\texttt{2-25-25-25-25-25-1}} \\
\midrule
Baseline ($t_0$) & --- & $70.0 \pm 12.3$\,s & $83.9 \pm 13.6$\,s & $107.6 \pm 22.2$\,s \\
\midrule
None & --- & $1.00$ & $1.00$ & $1.00$ \\
\midrule
\multirow{3}{*}{$\mathcal{R}_{\mathrm{L1}}$} & $10^{-4}$ & $1.58 \pm 0.31$ & $1.62 \pm 0.31$ & $1.74 \pm 0.30$ \\
 & $10^{-3}$ & $1.19 \pm 0.02$ & $1.23 \pm 0.02$ & $1.25 \pm 0.04$ \\
 & $10^{-2}$ & $1.21 \pm 0.04$ & $1.23 \pm 0.04$ & $1.24 \pm 0.04$ \\
\cmidrule(lr){1-5}
\multirow{3}{*}{$\mathcal{R}_{\mathrm{L2}}$} & $10^{-4}$ & $2.06 \pm 0.73$ & $1.55 \pm 0.30$ & $1.86 \pm 0.65$ \\
 & $10^{-3}$ & $1.28 \pm 0.06$ & $1.25 \pm 0.06$ & $1.30 \pm 0.07$ \\
 & $10^{-2}$ & $1.29 \pm 0.12$ & $1.23 \pm 0.08$ & $1.35 \pm 0.07$ \\
\cmidrule(lr){1-5}
\multirow{3}{*}{$\mathcal{R}_{\mathrm{BW}}$} & $10^{-4}$ & $1.65 \pm 0.26$ & $1.77 \pm 0.25$ & $2.45 \pm 0.40$ \\
 & $10^{-3}$ & $1.42 \pm 0.03$ & $1.49 \pm 0.06$ & $1.70 \pm 0.09$ \\
 & $10^{-2}$ & $1.41 \pm 0.04$ & $1.52 \pm 0.04$ & $1.76 \pm 0.05$ \\
\cmidrule(lr){1-5}
\multirow{3}{*}{$\mathcal{R}_{\mathrm{SN}}$} & $10^{-4}$ & $1.75 \pm 0.20$ & $2.08 \pm 0.22$ & $3.01 \pm 0.77$ \\
 & $10^{-3}$ & $1.59 \pm 0.04$ & $1.67 \pm 0.07$ & $1.91 \pm 0.08$ \\
 & $10^{-2}$ & $1.58 \pm 0.01$ & $1.64 \pm 0.09$ & $1.92 \pm 0.10$ \\
\cmidrule(lr){1-5}
\multirow{3}{*}{$\mathcal{R}_{\mathrm{SN2}}$} & $10^{-4}$ & $2.38 \pm 0.57$ & $2.02 \pm 0.41$ & $2.19 \pm 0.40$ \\
 & $10^{-3}$ & $1.53 \pm 0.03$ & $1.57 \pm 0.05$ & $1.79 \pm 0.08$ \\
 & $10^{-2}$ & $1.51 \pm 0.05$ & $1.59 \pm 0.03$ & $1.83 \pm 0.11$ \\
\cmidrule(lr){1-5}
\multirow{3}{*}{$\mathcal{R}_{\mathrm{LP}}$} & $10^{-4}$ & $6.06 \pm 1.21$ & $6.64 \pm 0.64$ & $10.92 \pm 0.34$ \\
 & $10^{-3}$ & $4.55 \pm 0.09$ & $6.38 \pm 0.21$ & $9.74 \pm 0.43$ \\
 & $10^{-2}$ & $4.59 \pm 0.07$ & $6.07 \pm 0.23$ & $9.09 \pm 0.55$ \\
\cmidrule(lr){1-5}
\multirow{3}{*}{$\mathcal{R}_{\mathrm{BW}}+\mathcal{R}_{\mathrm{LP}}$} & $10^{-4}$ & $5.92 \pm 0.57$ & $7.11 \pm 1.26$ & $10.28 \pm 0.45$ \\
 & $10^{-3}$ & $4.37 \pm 0.03$ & $5.14 \pm 0.24$ & $7.28 \pm 0.38$ \\
 & $10^{-2}$ & $4.01 \pm 0.08$ & $4.41 \pm 0.14$ & $5.43 \pm 0.50$ \\
\botrule
\end{tabular}
\end{table}

\paragraph{Accuracy vs tractability tradeoff.}
Table~\ref{tab:test_mse_ratios} shows normalized test MSE ratios relative to the unregularized baseline. 
Across all regularization methods, increasing the regularization penalty induces a reduced test accuracy as expected. 
Shrinkage regularizers ($\mathcal{R}_{\mathrm{L1}}$, $\mathcal{R}_{\mathrm{L2}}$) impose a steep accuracy penalty even at the moderate regularization strengths considered, with ratios exceeding 100$\times$ at $\lambda = 10^{-3}$ for the deepest architectures. Note that, following convention, these shrinkage regularizers are not normalized by the number of model parameters, while the per-sample regularizers are averaged over number of samples. 

On the other hand, the proposed bound-width ($\mathcal{R}_{\mathrm{BW}}$) and stability ($\mathcal{R}_{\mathrm{SN}}$) regularizers achieve ratios near, or even below unity at $\lambda = 10^{-4}$, indicating that mild regularization of the IBP bound widths may even provide a beneficial implicit regularization that simultaneously improves generalization and MILP tractability (though we do not claim this in general). 
The combined $\mathcal{R}_{\mathrm{BW}}+\mathcal{R}_{\mathrm{LP}}$ regularizer shows the same property at $\lambda=10^{-4}$. Nevertheless, for these regularizers, accuracy again degrades as $\lambda$ increases, most notably for the LP and combined regularizers. 
These results suggest that, in general regularizers provide a handle for tuning the tradeoff between surrogate model quality and the tractability of downstream optimization applications.

\begin{table}
\caption{Test MSE of each regularizer relative to the unregularized baseline, averaged across \texttt{ackley-2}, \texttt{himmelblau}, \texttt{peaks}. Cells show $\bar{r} = \frac{1}{|\mathcal{B}|}\sum_{b \in \mathcal{B}} \frac{\bar{m}_{b,\mathrm{reg}}}{\bar{m}_{b,\mathrm{none}}}$ $\pm$ std, where $\bar{m}$ is the mean test MSE over seeds. Raw MSE values differ across benchmarks, so only the normalised ratio is shown. Values $<1$ indicate lower test error than the baseline.}
\label{tab:test_mse_ratios}
\centering
\begin{tabular}{ll  c  c  c}
\toprule
regularizer & $\lambda$ & \multicolumn{1}{c}{\texttt{2-25-25-1}} & \multicolumn{1}{c}{\texttt{2-25-25-25-1}} & \multicolumn{1}{c}{\texttt{2-25-25-25-25-25-1}} \\
\midrule
None & --- & $1.00$ & $1.00$ & $1.00$ \\
\midrule
\multirow{3}{*}{$\mathcal{R}_{\mathrm{L1}}$} & $10^{-4}$ & $4.51 \pm 3.36$ & $10.55 \pm 7.11$ & $32.86 \pm 13.38$ \\
 & $10^{-3}$ & $45.48 \pm 43.78$ & $106.86 \pm 94.31$ & $352.35 \pm 235.47$ \\
 & $10^{-2}$ & $148.46 \pm 147.78$ & $415.74 \pm 416.71$ & $1493.12 \pm 1290.10$ \\
\cmidrule(lr){1-5}
\multirow{3}{*}{$\mathcal{R}_{\mathrm{L2}}$} & $10^{-4}$ & $1.92 \pm 1.16$ & $3.71 \pm 1.52$ & $12.09 \pm 4.73$ \\
 & $10^{-3}$ & $18.72 \pm 19.17$ & $35.29 \pm 30.76$ & $96.81 \pm 58.84$ \\
 & $10^{-2}$ & $84.11 \pm 71.10$ & $212.63 \pm 170.74$ & $760.82 \pm 525.09$ \\
\cmidrule(lr){1-5}
\multirow{3}{*}{$\mathcal{R}_{\mathrm{BW}}$} & $10^{-4}$ & $0.54 \pm 0.22$ & $0.40 \pm 0.13$ & $0.78 \pm 0.06$ \\
 & $10^{-3}$ & $0.89 \pm 0.38$ & $0.99 \pm 0.31$ & $1.75 \pm 0.23$ \\
 & $10^{-2}$ & $3.15 \pm 2.06$ & $4.77 \pm 2.06$ & $11.04 \pm 4.53$ \\
\cmidrule(lr){1-5}
\multirow{3}{*}{$\mathcal{R}_{\mathrm{SN}}$} & $10^{-4}$ & $0.74 \pm 0.17$ & $0.61 \pm 0.21$ & $0.80 \pm 0.20$ \\
 & $10^{-3}$ & $0.63 \pm 0.26$ & $0.68 \pm 0.27$ & $1.34 \pm 0.10$ \\
 & $10^{-2}$ & $2.28 \pm 1.16$ & $3.11 \pm 1.04$ & $6.47 \pm 2.30$ \\
\cmidrule(lr){1-5}
\multirow{3}{*}{$\mathcal{R}_{\mathrm{SN2}}$} & $10^{-4}$ & $0.97 \pm 0.02$ & $0.97 \pm 0.03$ & $1.01 \pm 0.02$ \\
 & $10^{-3}$ & $0.93 \pm 0.04$ & $0.95 \pm 0.03$ & $1.00 \pm 0.05$ \\
 & $10^{-2}$ & $1.10 \pm 0.12$ & $0.99 \pm 0.07$ & $0.95 \pm 0.02$ \\
\cmidrule(lr){1-5}
\multirow{3}{*}{$\mathcal{R}_{\mathrm{LP}}$} & $10^{-4}$ & $0.77 \pm 0.35$ & $0.65 \pm 0.27$ & $0.90 \pm 0.14$ \\
 & $10^{-3}$ & $0.98 \pm 0.58$ & $1.02 \pm 0.47$ & $2.17 \pm 0.97$ \\
 & $10^{-2}$ & $3.69 \pm 3.18$ & $7.20 \pm 5.45$ & $33.13 \pm 28.23$ \\
\cmidrule(lr){1-5}
\multirow{3}{*}{$\mathcal{R}_{\mathrm{BW}}+\mathcal{R}_{\mathrm{LP}}$} & $10^{-4}$ & $0.51 \pm 0.25$ & $0.47 \pm 0.21$ & $0.89 \pm 0.16$ \\
 & $10^{-3}$ & $1.45 \pm 0.89$ & $1.87 \pm 0.83$ & $3.75 \pm 1.54$ \\
 & $10^{-2}$ & $9.51 \pm 9.58$ & $12.47 \pm 10.34$ & $42.91 \pm 28.30$ \\
\botrule
\end{tabular}
\end{table}

\subsubsection{Optimization Results}
Tables~\ref{tab:himmelblau}--\ref{tab:ackley2d} report results for surrogate models trained with the various regularization strategies on two-dimensional benchmark functions. 
Four MILP tractability metrics are reported: number of unstable neurons $|\mathcal{U}|$, LP relaxation gap, MILP node count, and wall-clock MILP solve time. 
The unregularized baseline is computationally intractable on the deepest architectures for all benchmarks, consistently exceeding the 1800\,s MILP time limit. 

Overall, models trained with the proposed regularizers exhibit reduced MILP solve times, by up to four orders of magnitude relative to the baseline. Recall that, stronger regularization can further reduce MILP solve time at the expense of accuracy (Table~\ref{tab:test_mse_ratios}). 
To show this performance tradeoff, rows are shaded in grey when the mean objective found in the downstream problem is at least 5\% higher than the objective found using the unregularized surrogate model. 
On the simple \texttt{himmelblau} function (Table~\ref{tab:himmelblau}), surrogate model training appears especially sensitive to the shrinkage regularizers, where regularization degrades downstream performance in all cases except the weakest L2 regularizer. 
Models trained with the bound-width regularizer weakly included considerably accelerate the downstream MILP solution time without affecting solution quality. For the slightly more complicated \texttt{peaks} function (Table~\ref{tab:peaks}), the shrinkage regularizers again generally degrade decision performance for smaller surrogate models. This trend is mitigated for the deepest model (five hidden layers), which is more over-parameterized. 
In this setting, we found that regularization with the combined bound-width and LP regularizer to greatly accelerate downstream MILP solution, including many settings where the MILP can be solved in a single node. 

Overall, on these simple functions, $\mathcal{R}_{\mathrm{BW}}$ at $\lambda = 10^{-3}$ reduces unstable neurons by roughly 50\% on \texttt{himmelblau} and \texttt{peaks} (e.g., from 75 to 32--43 for the 3-layer architecture) and drives MILP times below 1\,s across most architectures on the simpler benchmarks. 
The stability regularizer $\mathcal{R}_{\mathrm{SN}}$ achieves similar reductions in $|\mathcal{U}|$ and MILP times, with slightly less aggressive compression of the LP gap as expected. 
The second stability regularizer $\mathcal{R}_{\mathrm{SN2}}$ consistently produces worse models compared to $\mathcal{R}_{\mathrm{SN}}$, suggesting the weaker gradient signal can make it less effective in practice. 
The combined $\mathcal{R}_{\mathrm{BW}}+\mathcal{R}_{\mathrm{LP}}$ regularizer is particularly effective at the lowest regularization strength ($\lambda=10^{-4}$), where it simultaneously achieves the smallest $|\mathcal{U}|$ and near-zero LP gap, resulting in the best MILP solve times in most settings where solution quality is not degraded. 
For the challenging \texttt{ackley-2} function (Table~\ref{tab:ackley2d}), the combined regularizer again generally results in shortest downstream MILP solve times. Note that the shrinkage regularizers can also reduce MILP solve times in this setting without affecting solution quality, albeit to a lesser extent.

Interestingly, inclusion of the LP regularizer $\mathcal{R}_{\mathrm{LP}}$ alone nearly eliminates the LP relaxation gap (often to $< 0.01$), but does not reduce the number of unstable neurons, since it only tightens the continuous relaxation without encouraging neuron stability. 
As a result, Gurobi must still branch on a large number of binary ReLU variables, and solve times remain significant on the harder instances.
This complementarity motivates the combined regularizer: $\mathcal{R}_{\mathrm{BW}}$ drives neurons toward stability (reducing $|\mathcal{U}|$), while $\mathcal{R}_{\mathrm{LP}}$ tightens the LP gap, and together they compound to produce substantially smaller B\&B trees (and MILP solve times). 
In summary, results on these two-dimensional benchmark functions show that, for simpler functions (where models are more overparameterized), downstream performance is more sensitive to regularization weights, especially shrinkage regularizers, and weak relaxation-informed regularization can greatly accelerate downstream performance. For the challenging Ackley function, (where models are less overparametrized), most regularization techniques accelerate downstream solution, with the powerful combined regularizer producing models with the best performance across most architectures. 

\paragraph{Effect of function complexity and architecture depth.}
The relative benefit of the proposed regularizers grows with both function complexity and network depth.
On \texttt{himmelblau} (the simplest benchmark), $\mathcal{R}_{\mathrm{BW}}$ at $\lambda=10^{-3}$ is sufficient to reduce the 5-layer MILP from infeasible within the time limit to 0.23\,s on average.
On \texttt{ackley-2}, higher LP gaps in the baseline (e.g., 17.84 vs.\ 13.39 for \texttt{peaks} with two hidden layers) mean that reducing $|\mathcal{U}|$ alone is not sufficient at low $\lambda$, and the combined regularizer is needed for consistently fast solves.

Table~\ref{tab:ackley5d} shows results for some larger NNs trained on the five-dimensional Ackley function. The L2 regularizer was omitted from these experiments, as we found it to be dominated by the L1 regularizer in previous experiments, similar to literature observations for this setting~\cite{plate2026analysis}. 
On \texttt{ackley-5} function, even with more neurons (up to 250 in the 50-wide architecture), the combined $\mathcal{R}_{\mathrm{BW}}+\mathcal{R}_{\mathrm{LP}}$ regularizer at $\lambda=10^{-4}$ reduces the 5-layer solve time from $>1800$\,s to under 1\,s while maintaining competitive surrogate accuracy. 
The standard L1 shrinkage regularizers offers less noticeable tractability improvements on these harder instances (larger NN models), though they are effective on the smaller models, e.g., three hidden layers (middle column of Table~\ref{tab:ackley5d}). 
This observation suggests that directly targeting the structure of the MILP embedding is essential for large gains in more challenging settings, where surrogate models are larger and less overparametrized.

\begin{sidewaystable}[h]
\caption{Results on the \texttt{himmelblau} benchmark across architectures and regularizers. Each regularizer family shows three rows for $\lambda \in \{10^{-4}, 10^{-3}, 10^{-2}\}$. $|\mathcal{U}|$: mean number of unstable neurons; LP gap: mean LP relaxation gap; MILP nodes/time: mean branch-and-bound nodes and wall-clock time (s).
\textbf{Bold} values indicate the best (lowest) result across all regularizers at each $\lambda$ level per architecture; ties are all bolded. 
\colorbox{gray!20}{Shaded} entries mark when the mean objective value found is worse (higher) than the unregularized baseline; such rows are excluded from best-value consideration.}
\label{tab:himmelblau}
\centering
\begin{tabular}{ll  cccc  cccc  cccc}
\toprule
\multirow{3}{*}{regularizer} & \multirow{3}{*}{$\lambda$} & \multicolumn{4}{c}{\texttt{2-25-25-1}} & \multicolumn{4}{c}{\texttt{2-25-25-25-1}} & \multicolumn{4}{c}{\texttt{2-25-25-25-25-25-1}} \\
\cmidrule(lr){3-6} \cmidrule(lr){7-10} \cmidrule(lr){11-14}
 & & $|\mathcal{U}|$ & LP gap & \multicolumn{2}{c}{MILP} & $|\mathcal{U}|$ & LP gap & \multicolumn{2}{c}{MILP} & $|\mathcal{U}|$ & LP gap & \multicolumn{2}{c}{MILP} \\
\cmidrule(lr){5-6} \cmidrule(lr){9-10} \cmidrule(lr){13-14}
 &  &  &  & nodes & time &  &  & nodes & time &  &  & nodes & time \\
\midrule
None & --- & $50.0$ & $23.93$ & $18{,}530$ & $4.08$ & $75.0$ & $53.91$ & $841{,}316$ & $243.87$ & $125.0$ & $350.81$ & $>2{,}321{,}554$ & $>1800$ \\
\midrule
\multirow{3}{*}{$\mathcal{R}_{\mathrm{L1}}$} & $10^{-4}$ & \cellcolor{gray!20}$49.5$ & \cellcolor{gray!20}$1.12$ & \cellcolor{gray!20}$49$ & \cellcolor{gray!20}$0.14$ & \cellcolor{gray!20}$71.2$ & \cellcolor{gray!20}$0.37$ & \cellcolor{gray!20}$67$ & \cellcolor{gray!20}$0.25$ & \cellcolor{gray!20}$105.0$ & \cellcolor{gray!20}$0.28$ & \cellcolor{gray!20}$778$ & \cellcolor{gray!20}$2.02$ \\
 & $10^{-3}$ & \cellcolor{gray!20}$46.5$ & \cellcolor{gray!20}$0.51$ & \cellcolor{gray!20}$19$ & \cellcolor{gray!20}$0.07$ & \cellcolor{gray!20}$61.5$ & \cellcolor{gray!20}$0.04$ & \cellcolor{gray!20}$1$ & \cellcolor{gray!20}$0.06$ & \cellcolor{gray!20}$85.2$ & \cellcolor{gray!20}$0.51$ & \cellcolor{gray!20}$1{,}412$ & \cellcolor{gray!20}$2.9$ \\
 & $10^{-2}$ & \cellcolor{gray!20}$49.9$ & \cellcolor{gray!20}$6.85$ & \cellcolor{gray!20}$3{,}400$ & \cellcolor{gray!20}$1.44$ & \cellcolor{gray!20}$75.0$ & \cellcolor{gray!20}$15.15$ & \cellcolor{gray!20}$742{,}272$ & \cellcolor{gray!20}$295.94$ & $125.0$ & $65.26$ & $>2{,}088{,}555$ & $>1800$ \\
\cmidrule(lr){1-14}
\multirow{3}{*}{$\mathcal{R}_{\mathrm{L2}}$} & $10^{-4}$ & $39.0$ & $0.29$ & $32$ & $0.18$ & $53.8$ & $0.07$ & $148$ & $0.39$ & $\mathbf{70.5}$ & $0.01$ & $205$ & $\mathbf{0.50}$ \\
 & $10^{-3}$ & \cellcolor{gray!20}$43.5$ & \cellcolor{gray!20}$0.21$ & \cellcolor{gray!20}$1$ & \cellcolor{gray!20}$0.04$ & \cellcolor{gray!20}$65.0$ & \cellcolor{gray!20}$0.13$ & \cellcolor{gray!20}$416$ & \cellcolor{gray!20}$0.48$ & \cellcolor{gray!20}$94.6$ & \cellcolor{gray!20}$1.14$ & \cellcolor{gray!20}$112{,}459$ & \cellcolor{gray!20}$73.40$ \\
 & $10^{-2}$ & \cellcolor{gray!20}$49.4$ & \cellcolor{gray!20}$6.12$ & \cellcolor{gray!20}$5{,}610$ & \cellcolor{gray!20}$1.64$ & \cellcolor{gray!20}$73.5$ & \cellcolor{gray!20}$10.84$ & \cellcolor{gray!20}$207{,}382$ & \cellcolor{gray!20}$70.27$ & \cellcolor{gray!20}$123.8$ & \cellcolor{gray!20}$72.41$ & \cellcolor{gray!20}$2{,}443{,}711$ & \cellcolor{gray!20}$1649.91$ \\
\cmidrule(lr){1-14}
\multirow{3}{*}{$\mathcal{R}_{\mathrm{BW}}$} & $10^{-4}$ & $39.3$ & $1.64$ & $231$ & $0.38$ & $52.1$ & $0.31$ & $275$ & $0.46$ & $79.2$ & $0.24$ & $1{,}280$ & $1.73$ \\
 & $10^{-3}$ & $\mathbf{23.6}$ & $0.26$ & $\mathbf{1}$ & $\mathbf{0.03}$ & $\mathbf{32.7}$ & $\mathbf{0.13}$ & $\mathbf{14}$ & $\mathbf{0.07}$ & $\mathbf{51.2}$ & $\mathbf{0.02}$ & $\mathbf{41}$ & $\mathbf{0.23}$ \\
 & $10^{-2}$ & \cellcolor{gray!20}$13.2$ & \cellcolor{gray!20}$0.13$ & \cellcolor{gray!20}$1$ & \cellcolor{gray!20}$0.01$ & \cellcolor{gray!20}$18.4$ & \cellcolor{gray!20}$0.04$ & \cellcolor{gray!20}$1$ & \cellcolor{gray!20}$0.02$ & $\mathbf{28.8}$ & $\mathbf{0.01}$ & $\mathbf{1}$ & $\mathbf{0.03}$ \\
\cmidrule(lr){1-14}
\multirow{3}{*}{$\mathcal{R}_{\mathrm{SN}}$} & $10^{-4}$ & $45.3$ & $5.61$ & $1{,}148$ & $0.65$ & $61.8$ & $1.55$ & $2{,}327$ & $1.21$ & $88.4$ & $0.54$ & $9{,}531$ & $8.06$ \\
 & $10^{-3}$ & $25.5$ & $0.64$ & $\mathbf{1}$ & $0.04$ & \cellcolor{gray!20}$34.9$ & \cellcolor{gray!20}$0.40$ & \cellcolor{gray!20}$23$ & \cellcolor{gray!20}$0.15$ & $53.7$ & $0.11$ & $312$ & $0.74$ \\
 & $10^{-2}$ & $\mathbf{13.4}$ & $\mathbf{0.35}$ & $\mathbf{1}$ & $\mathbf{0.01}$ & \cellcolor{gray!20}$17.0$ & \cellcolor{gray!20}$0.30$ & \cellcolor{gray!20}$1$ & \cellcolor{gray!20}$0.02$ & \cellcolor{gray!20}$28.0$ & \cellcolor{gray!20}$0.06$ & \cellcolor{gray!20}$1$ & \cellcolor{gray!20}$0.06$ \\
\cmidrule(lr){1-14}
\multirow{3}{*}{$\mathcal{R}_{\mathrm{SN2}}$} & $10^{-4}$ & $50.0$ & $24.15$ & $15{,}001$ & $3.32$ & \cellcolor{gray!20}$74.8$ & \cellcolor{gray!20}$52.72$ & \cellcolor{gray!20}$1{,}028{,}957$ & \cellcolor{gray!20}$232.30$ & $124.0$ & $334.55$ & $>3{,}090{,}349$ & $>1800$ \\
 & $10^{-3}$ & \cellcolor{gray!20}$49.6$ & \cellcolor{gray!20}$24.09$ & \cellcolor{gray!20}$18{,}414$ & \cellcolor{gray!20}$4.36$ & $73.7$ & $52.28$ & $978{,}297$ & $276.49$ & $123.0$ & $347.01$ & $>2{,}574{,}826$ & $>1800$ \\
 & $10^{-2}$ & $48.4$ & $23.70$ & $18{,}444$ & $5.11$ & \cellcolor{gray!20}$72.5$ & \cellcolor{gray!20}$51.46$ & \cellcolor{gray!20}$822{,}372$ & \cellcolor{gray!20}$284.34$ & $122.5$ & $329.29$ & $>2{,}483{,}279$ & $>1800$ \\
\cmidrule(lr){1-14}
\multirow{3}{*}{$\mathcal{R}_{\mathrm{LP}}$} & $10^{-4}$ & $50.0$ & $0.51$ & $159$ & $0.46$ & $75.0$ & $0.11$ & $1{,}906$ & $1.07$ & $125.0$ & $0.01$ & $693{,}059$ & $582.14$ \\
 & $10^{-3}$ & $49.9$ & $\mathbf{0.02}$ & $60$ & $0.22$ & \cellcolor{gray!20}$75.0$ & \cellcolor{gray!20}$0.01$ & \cellcolor{gray!20}$1{,}316$ & \cellcolor{gray!20}$0.98$ & \cellcolor{gray!20}$125.0$ & \cellcolor{gray!20}$0.00$ & \cellcolor{gray!20}$35{,}418$ &\cellcolor{gray!20} $33.46$ \\
 & $10^{-2}$ & \cellcolor{gray!20}$49.5$ & \cellcolor{gray!20}$0.00$ & \cellcolor{gray!20}$1$ & \cellcolor{gray!20}$0.06$ & \cellcolor{gray!20}$74.9$ & \cellcolor{gray!20}$0.00$ & \cellcolor{gray!20}$163$ & \cellcolor{gray!20}$0.27$ & \cellcolor{gray!20}$124.8$ & \cellcolor{gray!20}$0.00$ & \cellcolor{gray!20}$7{,}184$ & \cellcolor{gray!20}$7.38$ \\
\cmidrule(lr){1-14}
\multirow{3}{*}{$\mathcal{R}_{\mathrm{BW}}+\mathcal{R}_{\mathrm{LP}}$} & $10^{-4}$ & $\mathbf{32.2}$ & $\mathbf{0.08}$ & $\mathbf{1}$ & $\mathbf{0.03}$ & $\mathbf{46.6}$ & $\mathbf{0.00}$ & $\mathbf{1}$ & $\mathbf{0.07}$ & $77.2$ & $\mathbf{0.00}$ & $\mathbf{204}$ & $0.60$ \\
 & $10^{-3}$ & \cellcolor{gray!20}$18.9$ & \cellcolor{gray!20}$0.00$ & \cellcolor{gray!20}$1$ & \cellcolor{gray!20}$0.01$ & \cellcolor{gray!20}$28.6$ & \cellcolor{gray!20}$0.00$ & \cellcolor{gray!20}$1$ & \cellcolor{gray!20}$0.02$ & \cellcolor{gray!20}$46.4$ & \cellcolor{gray!20}$0.00$ & \cellcolor{gray!20}$1$ & \cellcolor{gray!20}$0.05$ \\
 & $10^{-2}$ & \cellcolor{gray!20}$11.8$ & \cellcolor{gray!20}$0.00$ & \cellcolor{gray!20}$1$ & \cellcolor{gray!20}$0.01$ & \cellcolor{gray!20}$14.1$ & \cellcolor{gray!20}$0.00$ & \cellcolor{gray!20}$1$ & \cellcolor{gray!20}$0.01$ & \cellcolor{gray!20}$20.8$ & \cellcolor{gray!20}$0.00$ & \cellcolor{gray!20}$1$ & \cellcolor{gray!20}$0.01$ \\
\botrule
\end{tabular}
\end{sidewaystable}

\begin{sidewaystable}[h]
\caption{Results on the \texttt{peaks} benchmark across architectures and regularizers. Each regularizer family shows three rows for $\lambda \in \{10^{-4}, 10^{-3}, 10^{-2}\}$. $|\mathcal{U}|$: mean number of unstable neurons; LP gap: mean LP relaxation gap; MILP nodes/time: mean branch-and-bound nodes and wall-clock time (s).
\textbf{Bold} values indicate the best (lowest) result across all regularizers at each $\lambda$ level per architecture; ties are all bolded. 
\colorbox{gray!20}{Shaded} entries mark when the mean objective value found is worse (higher) than the unregularized baseline; such rows are excluded from best-value consideration.}
\label{tab:peaks}
\centering
\begin{tabular}{ll  cccc  cccc  cccc}
\toprule
\multirow{3}{*}{regularizer} & \multirow{3}{*}{$\lambda$} & \multicolumn{4}{c}{\texttt{2-25-25-1}} & \multicolumn{4}{c}{\texttt{2-25-25-25-1}} & \multicolumn{4}{c}{\texttt{2-25-25-25-25-25-1}} \\
\cmidrule(lr){3-6} \cmidrule(lr){7-10} \cmidrule(lr){11-14}
 & & $|\mathcal{U}|$ & LP gap & \multicolumn{2}{c}{MILP} & $|\mathcal{U}|$ & LP gap & \multicolumn{2}{c}{MILP} & $|\mathcal{U}|$ & LP gap & \multicolumn{2}{c}{MILP} \\
\cmidrule(lr){5-6} \cmidrule(lr){9-10} \cmidrule(lr){13-14}
 &  &  &  & nodes & time &  &  & nodes & time &  &  & nodes & time \\
\midrule
None & --- & $50.0$ & $13.39$ & $10{,}547$ & $2.37$ & $74.9$ & $32.07$ & $462{,}620$ & $117.38$ & $124.9$ & $232.06$ & $>2{,}337{,}016$ & $>1800$ \\
\midrule
\multirow{3}{*}{$\mathcal{R}_{\mathrm{L1}}$} & $10^{-4}$ & \cellcolor{gray!20}$47.8$ & \cellcolor{gray!20}$1.22$ & \cellcolor{gray!20}$34$ & \cellcolor{gray!20}$0.08$ & \cellcolor{gray!20}$69.2$ & \cellcolor{gray!20}$0.44$ & \cellcolor{gray!20}$107$ & \cellcolor{gray!20}$0.23$ & $100.8$ & $1.58$ & $133{,}770$ & $90.61$ \\
 & $10^{-3}$ & \cellcolor{gray!20}$47.8$ & \cellcolor{gray!20}$2.31$ & \cellcolor{gray!20}$456$ & \cellcolor{gray!20}$0.36$ & \cellcolor{gray!20}$64.5$ & \cellcolor{gray!20}$3.79$ & \cellcolor{gray!20}$6{,}404$ & \cellcolor{gray!20}$3.22$ & $75.8$ & $4.07$ & $118{,}823$ & $125.17$ \\
 & $10^{-2}$ & \cellcolor{gray!20}$49.5$ & \cellcolor{gray!20}$2.61$ & \cellcolor{gray!20}$566$ & \cellcolor{gray!20}$0.43$ & \cellcolor{gray!20}$74.0$ & \cellcolor{gray!20}$2.95$ & \cellcolor{gray!20}$5{,}058$ & \cellcolor{gray!20}$2.48$ & $124.7$ & $10.64$ & $858{,}981$ & $613.41$ \\
\cmidrule(lr){1-14}
\multirow{3}{*}{$\mathcal{R}_{\mathrm{L2}}$} & $10^{-4}$ & $42.1$ & $1.00$ & $86$ & $0.23$ & $60.5$ & $0.29$ & $622$ & $0.51$ & $\mathbf{79.0}$ & $0.02$ & $649$ & $0.89$ \\
 & $10^{-3}$ & $\cellcolor{gray!20}44.0$ & \cellcolor{gray!20}$0.29$ & \cellcolor{gray!20}$42$ & \cellcolor{gray!20}$0.07$ & \cellcolor{gray!20}$59.8$ & \cellcolor{gray!20}$0.81$ & \cellcolor{gray!20}$2{,}680$ & \cellcolor{gray!20}$1.73$ & $72.2$ & $0.24$ & $1{,}606$ & $5.25$ \\
 & $10^{-2}$ & \cellcolor{gray!20}$50.0$ & \cellcolor{gray!20}$4.01$ & \cellcolor{gray!20}$1{,}726$ & \cellcolor{gray!20}$0.80$ & \cellcolor{gray!20}$75.0$ & \cellcolor{gray!20}$7.74$ & \cellcolor{gray!20}$36{,}476$ & \cellcolor{gray!20}$17.60$ & $125.0$ & $32.88$ & $2{,}308{,}125$ & $1341.62$ \\
\cmidrule(lr){1-14}
\multirow{3}{*}{$\mathcal{R}_{\mathrm{BW}}$} & $10^{-4}$ & $44.0$ & $2.07$ & $396$ & $0.49$ & $58.5$ & $0.83$ & $671$ & $0.72$ & $91.6$ & $0.39$ & $990$ & $2.11$ \\
 & $10^{-3}$ & $29.8$ & $0.82$ & $\mathbf{1}$ & $0.04$ & $43.1$ & $0.41$ & $\mathbf{1}$ & $0.08$ & $64.7$ & $0.10$ & $30$ & $0.17$ \\
 & $10^{-2}$ & $\mathbf{17.1}$ & $0.50$ & $\mathbf{1}$ & $\mathbf{0.02}$ & $23.1$ & $0.27$ & $\mathbf{1}$ & $0.02$ & $37.8$ & $0.02$ & $\mathbf{1}$ & $\mathbf{0.03}$ \\
\cmidrule(lr){1-14}
\multirow{3}{*}{$\mathcal{R}_{\mathrm{SN}}$} & $10^{-4}$ & $45.4$ & $4.03$ & $799$ & $0.60$ & $62.4$ & $1.86$ & $1{,}999$ & $1.46$ & $92.7$ & $0.88$ & $2{,}406$ & $4.11$ \\
 & $10^{-3}$ & $34.1$ & $1.00$ & $\mathbf{1}$ & $0.06$ & $44.0$ & $0.61$ & $\mathbf{1}$ & $0.10$ & $67.4$ & $0.27$ & $450$ & $0.57$ \\
 & $10^{-2}$ & $19.6$ & $0.83$ & $\mathbf{1}$ & $0.02$ & $24.3$ & $0.42$ & $\mathbf{1}$ & $0.03$ & $37.8$ & $0.16$ & $\mathbf{1}$ & $0.06$ \\
\cmidrule(lr){1-14}
\multirow{3}{*}{$\mathcal{R}_{\mathrm{SN2}}$} & $10^{-4}$ & $50.0$ & $12.26$ & $8{,}454$ & $2.15$ & $74.6$ & $27.77$ & $297{,}886$ & $87.00$ & $123.3$ & $200.84$ & $>3{,}083{,}439$ & $>1800$ \\
 & $10^{-3}$ & $46.7$ & $10.59$ & $5{,}319$ & $1.43$ & $68.7$ & $25.24$ & $171{,}206$ & $52.98$ & $115.0$ & $179.48$ & $3{,}168{,}132$ & $1773.63$ \\
 & $10^{-2}$ & $40.2$ & $9.38$ & $3{,}967$ & $1.18$ & $62.5$ & $21.51$ & $94{,}206$ & $32.22$ & $111.2$ & $158.83$ & $2{,}735{,}933$ & $1653.35$ \\
\cmidrule(lr){1-14}
\multirow{3}{*}{$\mathcal{R}_{\mathrm{LP}}$} & $10^{-4}$ & $50.0$ & $0.24$ & $\mathbf{1}$ & $0.05$ & $74.9$ & $0.01$ & $568$ & $0.39$ & $124.9$ & $\mathbf{0.00}$ & $51{,}325$ & $70.36$ \\
 & $10^{-3}$ & $50.0$ & $0.00$ & $\mathbf{1}$ & $0.05$ & $75.0$ & $0.00$ & $11$ & $0.13$ & $124.8$ & $\mathbf{0.00}$ & $107{,}905$ & $123.91$ \\
 & $10^{-2}$ & $50.0$ & $\mathbf{0.00}$ & $\mathbf{1}$ & $0.04$ & $74.9$ & $0.00$ & $\mathbf{1}$ & $0.09$ & $125.0$ & $0.00$ & $40{,}729$ & $47.21$ \\
\cmidrule(lr){1-14}
\multirow{3}{*}{$\mathcal{R}_{\mathrm{BW}}+\mathcal{R}_{\mathrm{LP}}$} & $10^{-4}$ & $\mathbf{38.9}$ & $\mathbf{0.12}$ & $\mathbf{1}$ & $\mathbf{0.05}$ & $\mathbf{55.6}$ & $\mathbf{0.00}$ & $\mathbf{1}$ & $\mathbf{0.10}$ & $90.3$ & $0.00$ & $\mathbf{189}$ & $\mathbf{0.60}$ \\
 & $10^{-3}$ & $\mathbf{26.3}$ & $\mathbf{0.00}$ & $\mathbf{1}$ & $\mathbf{0.01}$ & $\mathbf{41.1}$ & $\mathbf{0.00}$ & $\mathbf{1}$ & $\mathbf{0.03}$ & $\mathbf{59.4}$ & $0.00$ & $\mathbf{1}$ & $\mathbf{0.12}$ \\
 & $10^{-2}$ & \cellcolor{gray!20}$14.4$ & \cellcolor{gray!20}$0.00$ & \cellcolor{gray!20}$1$ & \cellcolor{gray!20}$0.01$ & $\mathbf{21.6}$ & $\mathbf{0.00}$ & $\mathbf{1}$ & $\mathbf{0.01}$ & $\mathbf{29.3}$ & $\mathbf{0.00}$ & $\mathbf{1}$ & $0.03$ \\
\botrule
\end{tabular}
\end{sidewaystable}

\begin{sidewaystable}[h]
\caption{Results on the \texttt{ackley-2} benchmark across architectures and regularizers. Each regularizer family shows three rows for $\lambda \in \{10^{-4}, 10^{-3}, 10^{-2}\}$. $|\mathcal{U}|$: mean number of unstable neurons; LP gap: mean LP relaxation gap; MILP nodes/time: mean branch-and-bound nodes and wall-clock time (s).
\textbf{Bold} values indicate the best (lowest) result across all regularizers at each $\lambda$ level per architecture; ties are all bolded. 
\colorbox{gray!20}{Shaded} entries mark when the mean objective value found is worse (higher) than the unregularized baseline; such rows are excluded from best-value consideration.}
\label{tab:ackley2d}
\centering
\begin{tabular}{ll  cccc  cccc  cccc}
\toprule
\multirow{3}{*}{regularizer} & \multirow{3}{*}{$\lambda$} & \multicolumn{4}{c}{\texttt{2-25-25-1}} & \multicolumn{4}{c}{\texttt{2-25-25-25-1}} & \multicolumn{4}{c}{\texttt{2-25-25-25-25-25-1}} \\
\cmidrule(lr){3-6} \cmidrule(lr){7-10} \cmidrule(lr){11-14}
 & & $|\mathcal{U}|$ & LP gap & \multicolumn{2}{c}{MILP} & $|\mathcal{U}|$ & LP gap & \multicolumn{2}{c}{MILP} & $|\mathcal{U}|$ & LP gap & \multicolumn{2}{c}{MILP} \\
\cmidrule(lr){5-6} \cmidrule(lr){9-10} \cmidrule(lr){13-14}
 &  &  &  & nodes & time &  &  & nodes & time &  &  & nodes & time \\
\midrule
None & --- & $50.0$ & $17.84$ & $14{,}683$ & $3.82$ & $75.0$ & $94.92$ & $1{,}642{,}697$ & $462.94$ & $125.0$ & $777.85$ & $>2{,}558{,}663$ & $>1800$ \\
\midrule
\multirow{3}{*}{$\mathcal{R}_{\mathrm{L1}}$} & $10^{-4}$ & $50.0$ & $6.15$ & $2{,}963$ & $1.11$ & $74.8$ & $17.13$ & $474{,}665$ & $115.66$ & $123.3$ & $88.96$ & $1{,}752{,}605$ & $1304.37$ \\
 & $10^{-3}$ & $47.4$ & $2.34$ & $1{,}012$ & $0.52$ & $69.9$ & $6.14$ & $45{,}800$ & $15.65$ & $101.8$ & $32.29$ & $1{,}250{,}287$ & $900.52$ \\
 & $10^{-2}$ & $49.6$ & $2.13$ & $803$ & $0.51$ & $69.1$ & $1.62$ & $7{,}001$ & $3.45$ & $113.2$ & $10.66$ & $694{,}491$ & $648.10$ \\
\cmidrule(lr){1-14}
\multirow{3}{*}{$\mathcal{R}_{\mathrm{L2}}$} & $10^{-4}$ & $48.0$ & $6.56$ & $4{,}978$ & $1.34$ & $71.2$ & $19.20$ & $498{,}437$ & $98.04$ & $118.5$ & $107.00$ & $3{,}377{,}901$ & $1721.21$ \\
 & $10^{-3}$ & $46.5$ & $3.03$ & $1{,}223$ & $0.61$ & $71.8$ & $8.88$ & $78{,}488$ & $21.36$ & $117.0$ & $48.85$ & $1{,}934{,}706$ & $967.87$ \\
 & $10^{-2}$ & $48.7$ & $2.54$ & $1{,}072$ & $0.82$ & $73.3$ & $5.50$ & $37{,}752$ & $15.01$ & $124.7$ & $40.25$ & $2{,}450{,}682$ & $1443.08$ \\
\cmidrule(lr){1-14}
\multirow{3}{*}{$\mathcal{R}_{\mathrm{BW}}$} & $10^{-4}$ & $41.2$ & $4.94$ & $612$ & $0.69$ & $62.5$ & $9.36$ & $15{,}001$ & $5.01$ & $\mathbf{99.3}$ & $8.54$ & $486{,}829$ & $470.16$ \\
 & $10^{-3}$ & $22.3$ & $0.90$ & $\mathbf{1}$ & $0.03$ & $\mathbf{38.8}$ & $\mathbf{2.03}$ & $\mathbf{102}$ & $\mathbf{0.25}$ & $66.5$ & $1.78$ & $3{,}712$ & $2.46$ \\
 & $10^{-2}$ & $9.9$ & $0.12$ & $6$ & $0.02$ & \cellcolor{gray!20}$12.2$ & \cellcolor{gray!20}$0.05$ & \cellcolor{gray!20}$1$ & \cellcolor{gray!20}$0.01$ & $21.7$ & $0.08$ & $1$ & $0.02$ \\
\cmidrule(lr){1-14}
\multirow{3}{*}{$\mathcal{R}_{\mathrm{SN}}$} & $10^{-4}$ & $45.5$ & $6.73$ & $1{,}220$ & $0.79$ & $64.5$ & $11.36$ & $19{,}685$ & $9.29$ & $101.2$ & $12.68$ & $357{,}570$ & $492.44$ \\
 & $10^{-3}$ & \cellcolor{gray!20}$34.5$ & \cellcolor{gray!20}$2.33$ & \cellcolor{gray!20}$32$ & \cellcolor{gray!20}$0.15$ & \cellcolor{gray!20}$51.8$ & \cellcolor{gray!20}$3.65$ & \cellcolor{gray!20}$1{,}261$ & \cellcolor{gray!20}$0.87$ & $83.3$ & $3.54$ & $33{,}285$ & $35.72$ \\
 & $10^{-2}$ & \cellcolor{gray!20}$10.2$ & \cellcolor{gray!20}$0.16$ & \cellcolor{gray!20}$1$ & \cellcolor{gray!20}$0.01$ & \cellcolor{gray!20}$17.8$ & \cellcolor{gray!20}$0.34$ & \cellcolor{gray!20}$1$ & \cellcolor{gray!20}$0.02$ & $35.2$ & $0.55$ & $14$ & $0.11$ \\
\cmidrule(lr){1-14}
\multirow{3}{*}{$\mathcal{R}_{\mathrm{SN2}}$} & $10^{-4}$ & $49.7$ & $16.86$ & $16{,}744$ & $4.14$ & $74.9$ & $87.16$ & $1{,}170{,}921$ & $271.89$ & $124.5$ & $740.97$ & $>2{,}145{,}526$ & $>1800$ \\
 & $10^{-3}$ & $48.8$ & $16.66$ & $11{,}933$ & $2.83$ & $73.2$ & $84.83$ & $1{,}157{,}158$ & $273.73$ & $122.1$ & $759.84$ & $>2{,}670{,}470$ & $>1800$ \\
 & $10^{-2}$ & $45.9$ & $16.72$ & $12{,}411$ & $3.04$ & $70.5$ & $84.66$ & $1{,}078{,}148$ & $218.33$ & $120.3$ & $707.98$ & $>2{,}667{,}471$ & $>1800$ \\
\cmidrule(lr){1-14}
\multirow{3}{*}{$\mathcal{R}_{\mathrm{LP}}$} & $10^{-4}$ & $50.0$ & $\mathbf{0.81}$ & $49$ & $0.21$ & $75.0$ & $\mathbf{0.11}$ & $3{,}438$ & $1.71$ & $125.0$ & $0.02$ & $601{,}896$ & $592.19$ \\
 & $10^{-3}$ & \cellcolor{gray!20}$50.0$ & \cellcolor{gray!20}$0.00$ & \cellcolor{gray!20}$1$ & \cellcolor{gray!20}$0.07$ & \cellcolor{gray!20}$75.0$ & \cellcolor{gray!20}$0.01$ & \cellcolor{gray!20}$308$ & \cellcolor{gray!20}$0.35$ & $125.0$ & $0.00$ & $59{,}792$ & $74.28$ \\
 & $10^{-2}$ & $49.0$ & $\mathbf{0.00}$ & $6$ & $0.06$ & $74.5$ & $0.00$ & $4$ & $0.12$ & $124.8$ & $0.00$ & $1{,}834$ & $1.73$ \\
\cmidrule(lr){1-14}
\multirow{3}{*}{$\mathcal{R}_{\mathrm{BW}}+\mathcal{R}_{\mathrm{LP}}$} & $10^{-4}$ & $\mathbf{37.2}$ & $0.88$ & $\mathbf{20}$ & $\mathbf{0.10}$ & $\mathbf{58.0}$ & $0.21$ & $\mathbf{269}$ & $\mathbf{0.51}$ & $99.8$ & $\mathbf{0.00}$ & $\mathbf{2{,}772}$ & $\mathbf{2.53}$ \\
 & $10^{-3}$ & $\mathbf{19.6}$ & $\mathbf{0.07}$ & $5$ & $\mathbf{0.02}$ & \cellcolor{gray!20}$33.2$ & \cellcolor{gray!20}$0.00$ & \cellcolor{gray!20}$1$ & \cellcolor{gray!20}$0.02$ & $\mathbf{64.2}$ & $\mathbf{0.00}$ & $\mathbf{25}$ & $\mathbf{0.20}$ \\
 & $10^{-2}$ & $\mathbf{7.7}$ & $0.01$ & $\mathbf{1}$ & $\mathbf{0.00}$ & $\mathbf{8.3}$ & $\mathbf{0.00}$ & $\mathbf{1}$ & $\mathbf{0.00}$ & $\mathbf{10.8}$ & $\mathbf{0.00}$ & $\mathbf{1}$ & $\mathbf{0.01}$ \\
\botrule
\end{tabular}
\end{sidewaystable}

\begin{sidewaystable}[h]
\caption{Results on the \texttt{ackley-5} benchmark across architectures and regularizers. Each regularizer family shows three rows for $\lambda \in \{10^{-4}, 10^{-3}, 10^{-2}\}$. $|\mathcal{U}|$: mean number of unstable neurons; LP gap: mean LP relaxation gap; MILP nodes/time: mean branch-and-bound nodes and wall-clock time (s).
\textbf{Bold} values indicate the best (lowest) result across all regularizers at each $\lambda$ level per architecture; ties are all bolded. 
\colorbox{gray!20}{Shaded} entries mark when the mean objective value found is worse (higher) than the unregularized baseline; such rows are excluded from best-value consideration.}
\label{tab:ackley5d}
\centering
\begin{tabular}{ll  cccc  cccc  cccc}
\toprule
\multirow{3}{*}{regularizer} & \multirow{3}{*}{$\lambda$} & \multicolumn{4}{c}{\texttt{5-25-25-25-25-25-1}} & \multicolumn{4}{c}{\texttt{5-50-50-50-1}} & \multicolumn{4}{c}{\texttt{5-50-50-50-50-50-1}} \\
\cmidrule(lr){3-6} \cmidrule(lr){7-10} \cmidrule(lr){11-14}
 & & $|\mathcal{U}|$ & LP gap & \multicolumn{2}{c}{MILP} & $|\mathcal{U}|$ & LP gap & \multicolumn{2}{c}{MILP} & $|\mathcal{U}|$ & LP gap & \multicolumn{2}{c}{MILP} \\
\cmidrule(lr){5-6} \cmidrule(lr){9-10} \cmidrule(lr){13-14}
 &  &  &  & nodes & time &  &  & nodes & time &  &  & nodes & time \\
\midrule
None & --- & $125.0$ & $376.81$ & $2{,}851{,}019$ & $1541.25$ & $149.6$ & $32.70$ & $2{,}044{,}480$ & $1363.88$ & $249.9$ & $1955.61$ & $>468{,}135$ & $>1800$ \\
\midrule
\multirow{3}{*}{$\mathcal{R}_{\mathrm{L1}}$} & $10^{-4}$ & $84.8$ & $2.06$ & $166{,}609$ & $180.14$ & $122.8$ & $0.97$ & $152{,}175$ & $204.79$ & $204.2$ & $32.71$ & $241{,}319$ & $1042.25$ \\
 & $10^{-3}$ & $81.8$ & $0.08$ & $339$ & $0.47$ & $130.4$ & $\mathbf{1.55}$ & $\mathbf{474{,}476}$ & $\mathbf{611.46}$ & $181.0$ & $7.20$ & $207{,}418$ & $698.49$ \\
 & $10^{-2}$ & $123.0$ & $9.77$ & $872{,}461$ & $554.50$ & $\mathbf{148.3}$ & $\mathbf{5.13}$ & $\mathbf{566{,}103}$ & $\mathbf{788.14}$ & $229.2$ & $12.83$ & $277{,}162$ & $1153.38$ \\
\cmidrule(lr){1-14}
\multirow{3}{*}{$\mathcal{R}_{\mathrm{BW}}$} & $10^{-4}$ & $67.8$ & $0.56$ & $523$ & $0.95$ & $\mathbf{79.8}$ & $\mathbf{1.84}$ & $\mathbf{1{,}506}$ & $\mathbf{1.76}$ & $94.2$ & $1.10$ & $2{,}089$ & $3.81$ \\
 & $10^{-3}$ & \cellcolor{gray!20}$34.6$ & \cellcolor{gray!20}$0.08$ & \cellcolor{gray!20}$1$ & \cellcolor{gray!20}$0.02$ & \cellcolor{gray!20}$43.2$ & \cellcolor{gray!20}$0.32$ & \cellcolor{gray!20}$1$ & \cellcolor{gray!20}$0.05$ & $49.6$ & $0.13$ & $1$ & $\mathbf{0.05}$ \\
 & $10^{-2}$ & \cellcolor{gray!20}$16.3$ & \cellcolor{gray!20}$0.00$ & \cellcolor{gray!20}$0$ & \cellcolor{gray!20}$0.00$ & \cellcolor{gray!20}$16.8$ & \cellcolor{gray!20}$0.00$ & \cellcolor{gray!20}$1$ & \cellcolor{gray!20}$0.00$ & \cellcolor{gray!20}$22.6$ & \cellcolor{gray!20}$0.00$ & \cellcolor{gray!20}$0$ & \cellcolor{gray!20}$0.01$ \\
\cmidrule(lr){1-14}
\multirow{3}{*}{$\mathcal{R}_{\mathrm{SN}}$} & $10^{-4}$ & $73.9$ & $1.03$ & $1{,}379$ & $1.62$ & $86.8$ & $1.99$ & $1{,}950$ & $2.30$ & $107.7$ & $1.93$ & $34{,}844$ & $25.64$ \\
 & $10^{-3}$ & $\mathbf{43.5}$ & $0.14$ & $\mathbf{2}$ & $\mathbf{0.06}$ & \cellcolor{gray!20}$54.4$ & \cellcolor{gray!20}$0.60$ & \cellcolor{gray!20}$46$ & \cellcolor{gray!20}$0.20$ & $64.8$ & $0.18$ & $103$ & $0.35$ \\
 & $10^{-2}$ & \cellcolor{gray!20}$18.6$ & \cellcolor{gray!20}$0.02$ & \cellcolor{gray!20}$1$ & \cellcolor{gray!20}$0.01$ & \cellcolor{gray!20}$21.1$ & \cellcolor{gray!20}$0.01$ & \cellcolor{gray!20}$1$ & \cellcolor{gray!20}$0.01$ & \cellcolor{gray!20}$29.1$ & \cellcolor{gray!20}$0.01$ & \cellcolor{gray!20}$1$ & \cellcolor{gray!20}$0.02$ \\
\cmidrule(lr){1-14}
\multirow{3}{*}{$\mathcal{R}_{\mathrm{SN2}}$} & $10^{-4}$ & $124.3$ & $358.06$ & $2{,}667{,}570$ & $1439.55$ & $134.4$ & $24.52$ & $1{,}201{,}306$ & $728.03$ & $246.8$ & $1732.69$ & $>591{,}537$ & $>1800$ \\
 & $10^{-3}$ & $123.2$ & $361.40$ & $2{,}899{,}113$ & $1401.15$ & $\mathbf{117.6}$ & $19.23$ & $744{,}149$ & $377.74$ & $243.7$ & $1809.90$ & $>503{,}628$ & $>1800$ \\
 & $10^{-2}$ & $\mathbf{122.1}$ & $349.86$ & $2{,}418{,}068$ & $1253.44$ & \cellcolor{gray!20}$119.8$ & \cellcolor{gray!20}$21.80$ & \cellcolor{gray!20}$356{,}002$ & \cellcolor{gray!20}$221.30$ & $240.5$ & $1475.18$ & $>637{,}046$ & $>1800$ \\
\cmidrule(lr){1-14}
\multirow{3}{*}{$\mathcal{R}_{\mathrm{LP}}$} & $10^{-4}$ & $125.0$ & $0.00$ & $13{,}136$ & $8.33$ & \cellcolor{gray!20}$149.9$ & \cellcolor{gray!20}$0.00$ & \cellcolor{gray!20}$13{,}686$ & \cellcolor{gray!20}$9.17$ & $250.0$ & $\mathbf{0.00}$ & $10{,}110$ & $20.54$ \\
 & $10^{-3}$ & $124.9$ & $\mathbf{0.00}$ & $96{,}528$ & $41.03$ & \cellcolor{gray!20}$149.9$ & \cellcolor{gray!20}$0.00$ & \cellcolor{gray!20}$13{,}889$ & \cellcolor{gray!20}$6.00$ & $249.8$ & $0.00$ & $200{,}495$ & $231.59$ \\
 & $10^{-2}$ & $124.9$ & $\mathbf{0.00}$ & $\mathbf{614{,}044}$ & $\mathbf{224.40}$ & \cellcolor{gray!20}$149.8$ & \cellcolor{gray!20}$0.00$ & \cellcolor{gray!20}$48{,}391$ & \cellcolor{gray!20}$30.20$ & $249.8$ & $\mathbf{0.00}$ & $\mathbf{5{,}465}$ & $\mathbf{12.43}$ \\
\cmidrule(lr){1-14}
\multirow{3}{*}{$\mathcal{R}_{\mathrm{BW}}+\mathcal{R}_{\mathrm{LP}}$} & $10^{-4}$ & $\mathbf{66.2}$ & $\mathbf{0.01}$ & $\mathbf{5}$ & $\mathbf{0.14}$ & \cellcolor{gray!20}$71.9$ & \cellcolor{gray!20}$0.04$ & \cellcolor{gray!20}$3$ & \cellcolor{gray!20}$0.31$ & $\mathbf{87.6}$ & $0.01$ & $\mathbf{155}$ & $\mathbf{0.52}$ \\
 & $10^{-3}$ & \cellcolor{gray!20}$30.9$ & \cellcolor{gray!20}$0.00$ & \cellcolor{gray!20}$1$ & \cellcolor{gray!20}$0.01$ & \cellcolor{gray!20}$34.9$ & \cellcolor{gray!20}$0.00$ & \cellcolor{gray!20}$1$ & \cellcolor{gray!20}$0.04$ & $\mathbf{45.0}$ & $\mathbf{0.00}$ & $\mathbf{1}$ & $0.06$ \\
 & $10^{-2}$ & \cellcolor{gray!20}$13.0$ & \cellcolor{gray!20}$0.00$ & \cellcolor{gray!20}$1$ & \cellcolor{gray!20}$0.00$ & \cellcolor{gray!20}$14.4$ & \cellcolor{gray!20}$0.00$ & \cellcolor{gray!20}$1$ & \cellcolor{gray!20}$0.01$ & $\mathbf{31.0}$ & $0.60$ & $18{,}290$ & $90.03$ \\
\botrule
\end{tabular}
\end{sidewaystable}

\clearpage

\subsection{Optimization over Quantile Neural Networks}

The previous section studied the proposed regularizers on single-output surrogate models, where the downstream MILP simply minimizes the predicted output.
However, many practical applications involve more complicated settings, e.g., surrogate models can have multiple outputs and the downstream optimization problem can involve a more complex objective. 
To encompass some of these elements, we consider quantile neural networks (QNNs) applied to two-stage stochastic programming (2SP), an important setting in which QNN surrogates have been used to approximate the distribution of the second-stage `recourse' function~\cite{alcantara2025quantile,ghilardi2025integrated}. 
QNNs are multi-output models that predict specific quantiles of a target distribution, enabling them to estimate uncertainty and produce prediction intervals instead of only point predictions.

We consider the capacitated facility location problem (CFLP), a classical 2SP in which binary first-stage decisions $y \in \{0,1\}^{n_f}$ determine which of $n_f$ facilities to open, and the second-stage recourse allocates facility capacity to $n_c$ customer demands that are revealed as random scenarios $\xi$. 
\citet{patel2022neur2sp} study the problem using standard NN surrogates for the second-stage objective, and \citet{alcantara2025quantile} use QNN surrogates to enable risk-aware, distributional modeling. 
Interestingly, \citet{liu2025icnn} train input-convex NN surrogates to accelerate the downstream 2SP problem, which can then be formulated as an LP. 
Following the setup of~\citet{alcantara2025quantile}, we train a QNN surrogate $f_\theta : \{0,1\}^{n_f} \to \R^K$ that, given a first-stage decision $y$, predicts $K = 50$ quantiles of the second-stage cost distribution at equally spaced intervals. 
The QNN is trained by minimizing the pinball (quantile regression) loss:
\begin{equation}\label{eq:pinball}
  \mathcal{L}_{\mathrm{pinball}}(\theta)
  = \frac{1}{N K} \sum_{i=1}^{N} \sum_{k=1}^{K}
    \max\bigl\{\tau_k (v_i - f_{\theta,k}(y_i)),\;
               (\tau_k - 1)(v_i - f_{\theta,k}(y_i))\bigr\},
\end{equation}
where $v_i$ is the realized second-stage cost for the $i$-th training sample, generated by solving the recourse problem for a random demand scenario.

Once trained, the QNN surrogate replaces the expensive recourse computation in the first-stage optimization and enables distributional predictions such as Conditional Value-at-Risk (CVaR).
The resulting 2SP is formulated as a MILP that minimizes a mean and/or CVaR-based objective over the predicted quantile outputs:
\begin{equation}\label{eq:2sp-obj}
  \min_{y \in \{0,1\}^{n_f}} \;
  c_f^\top y
  \;+\;
  (1-\lambda_\mathrm{2SP})\,\frac{1}{K}\sum_{k=1}^{K} f_{\theta,k}(y)
  \;+\;
  \frac{\lambda_\mathrm{2SP}}{|\mathcal{T}|}\sum_{k \in \mathcal{T}} f_{\theta,k}(y),
\end{equation}
where $c_f$ is the vector of first-stage costs,
$\lambda_\mathrm{2SP}$ is the risk-aversion parameter,
$\mathcal{T} = \{k : \tau_k \ge \alpha\}$ is the set of tail quantile indices,
and $\alpha$ is the CVaR confidence level.
This formulation includes both binary decision variables and the MILP encoding of the trained QNN, making tractability of the overall problem highly dependent on the structural properties of the surrogate model. 
We refer the reader to~\citet{alcantara2025quantile} for further details on the QNN-based 2SP framework.

\subsubsection{Training Results}

We study the \texttt{CFLP\_50\_50} instance ($n_f = n_c = 50$), following the instance generation procedure of~\citet{patel2022neur2sp}, which is based in turn on~\citet{cornuejols1991comparison}. 
We also consider an extended \texttt{CFLP\_75\_75} instance ($n_f = n_c = 75$), generated using the same procedure. 
The training dataset for each consists of 20{,}000 samples, with samples generated by drawing a random first-stage decision $y$ and solving the second-stage recourse for a random demand scenario. 
For the \texttt{CFLP\_50\_50} and \texttt{CFLP\_75\_75} problems, 250 and 570 samples respectively reached the 600\,s time limit, but feasible solutions were found for all of them. 
Data are split 80/20 into training and validation sets, normalized, and models are trained for 200 epochs with Adam.

We consider three architectures with $\{2, 3, 5\}$ hidden layers of 25 neurons each, denoted respectively as \texttt{X-25-25-50}, \texttt{X-25-25-25-50}, \texttt{X-25-25-25-25-25-50}. The input dimension equals $n_f$, and output dimensions is 50, corresponding to 50 quantiles. 
Each configuration is trained over 20 random seeds and tested in downstream MILPs as in~\eqref{eq:2sp-obj}. 
We omit the $\mathcal{R}_\mathrm{SN2}$ regularizer for brevity, as it failed to produce tractable models in nearly all of the settings above, but reintroduce the L2 regularizer, as its performance in multi-output settings has not been studied. 
The multi-output LP relaxation gap regularizer $\mathcal{R}_{\mathrm{LP}}$ uses random projections with non-negative directions, consistent with the non-negative weights in the downstream objective. 
The computational overheads of the proposed regularizers are independent of the specific training data; we observe costs largely similar to what was reported in Table~\ref{tab:train_time_ratios} for the benchmark functions. 

Tables~\ref{tab:cflp_50_50_pinball}--\ref{tab:cflp_75_75_pinball} report the test pinball loss (scaled $\times 10^2$) for each regularizer and architecture combination.
The unregularized baseline achieves a pinball loss of approximately $3.0$ and $2.5$ for the two problems respectively (scaled) across all architectures.
The proposed regularizers $\mathcal{R}_{\mathrm{BW}}$ and $\mathcal{R}_{\mathrm{SN}}$ preserve or even slightly improve prediction quality at all regularization strengths considered; on both problems $\mathcal{R}_{\mathrm{BW}}$ at $\lambda = 10^{-2}$ achieves the lowest pinball loss across architectures. 
In contrast, the shrinkage regularizers again suffer severe prediction degradation at higher regularization strengths, supporting our observation that these regularizers must be tuned and deployed more cautiously in practice. 
For example, $\mathcal{R}_{\mathrm{L1}}$ at $\lambda = 10^{-2}$ inflates the pinball loss by an order of magnitude (to $30$--$35$ on both problems), indicating that the network has collapsed to a near-constant function. 
$\mathcal{R}_{\mathrm{L2}}$ exhibits similar degradation at $\lambda = 10^{-2}$.
Any downstream MILP tractability ``improvements'' at these settings are therefore not attributable to the regularizer, but rather to the trivialization of the surrogate model. 

In this setting a trained QNN surrogate is usable across downstream instances, e.g., \citet{alcantara2025quantile} perform sensitivity studies across various levels of $\lambda_\mathrm{2SP}$ and $\alpha$, while \citet{ghilardi2025integrated} consider varying first-stage costs $c_f$. 
Given the lack of an obvious single objective, we consider the pinball loss~\eqref{eq:pinball} as a general metric of model decision quality across downstream optimization tasks: rows where the test pinball loss exceeds the baseline by more than 10\% are shaded in grey in Tables~\ref{tab:cflp_50_50_pinball}--\ref{tab:cflp_75_75_pinball}. Moreover, they are excluded from best-value comparisons in Tables~\ref{tab:cflp_50_50}--\ref{tab:cflp_75_75}. 

\begin{table}
\caption{Pinball loss (mean $\pm$ std over seeds, scaled $\times 10^2$) on \texttt{cflp\_50\_50} for each regularizer and architecture. Values are computed on the held-out test set. \textbf{Bold} indicates the lowest pinball loss within each $\lambda$ level per architecture, excluding shaded rows. \colorbox{gray!20}{Shaded} entries exceed the unregularized baseline by more than 10\%, indicating that the regularizer has degraded predictive quality.}
\label{tab:cflp_50_50_pinball}
\centering
\begin{tabular}{ll  c  c  c}
\toprule
regularizer & $\lambda$ & \multicolumn{1}{c}{\texttt{50-25-25-50}} & \multicolumn{1}{c}{\texttt{50-25-25-25-50}} & \multicolumn{1}{c}{\texttt{50-25-25-25-25-25-50}} \\
 & & ($\times 10^{2}$) & ($\times 10^{2}$) & ($\times 10^{2}$) \\
\midrule
None & --- & $3.03 \pm 0.02$ & $3.02 \pm 0.03$ & $3.06 \pm 0.00$ \\
\midrule
\multirow{3}{*}{$\mathcal{R}_{\mathrm{L1}}$} & $10^{-4}$ & $\mathbf{2.94 \pm 0.01}$ & $\mathbf{2.89 \pm 0.01}$ & $\mathbf{2.87 \pm 0.01}$ \\
 & $10^{-3}$ & \cellcolor{gray!20}$3.43 \pm 0.14$ & $3.26 \pm 0.15$ & $3.20 \pm 0.15$ \\
 & $10^{-2}$ & \cellcolor{gray!20}$30.77 \pm 6.80$ & \cellcolor{gray!20}$32.54 \pm 7.27$ & \cellcolor{gray!20}$34.60 \pm 7.07$ \\
\cmidrule(lr){1-5}
\multirow{3}{*}{$\mathcal{R}_{\mathrm{L2}}$} & $10^{-4}$ & $2.99 \pm 0.02$ & $3.00 \pm 0.02$ & $2.97 \pm 0.02$ \\
 & $10^{-3}$ & $3.15 \pm 0.02$ & $3.06 \pm 0.01$ & $3.02 \pm 0.01$ \\
 & $10^{-2}$ & \cellcolor{gray!20}$22.82 \pm 3.54$ & \cellcolor{gray!20}$14.38 \pm 9.06$ & \cellcolor{gray!20}$4.77 \pm 0.02$ \\
\cmidrule(lr){1-5}
\multirow{3}{*}{$\mathcal{R}_{\mathrm{BW}}$} & $10^{-4}$ & $3.01 \pm 0.02$ & $3.00 \pm 0.03$ & $2.98 \pm 0.02$ \\
 & $10^{-3}$ & $\mathbf{2.97 \pm 0.02}$ & $\mathbf{2.94 \pm 0.02}$ & $2.93 \pm 0.02$ \\
 & $10^{-2}$ & $\mathbf{2.92 \pm 0.01}$ & $\mathbf{2.88 \pm 0.01}$ & $2.86 \pm 0.01$ \\
\cmidrule(lr){1-5}
\multirow{3}{*}{$\mathcal{R}_{\mathrm{SN}}$} & $10^{-4}$ & $3.02 \pm 0.02$ & $3.01 \pm 0.03$ & $2.99 \pm 0.03$ \\
 & $10^{-3}$ & $3.01 \pm 0.02$ & $2.99 \pm 0.02$ & $2.95 \pm 0.02$ \\
 & $10^{-2}$ & $2.93 \pm 0.02$ & $2.90 \pm 0.02$ & $\mathbf{2.86 \pm 0.01}$ \\
\cmidrule(lr){1-5}
\multirow{3}{*}{$\mathcal{R}_{\mathrm{LP}}$} & $10^{-4}$ & $3.04 \pm 0.02$ & $3.09 \pm 0.03$ & $3.07 \pm 0.04$ \\
 & $10^{-3}$ & $3.08 \pm 0.03$ & $3.11 \pm 0.04$ & $3.33 \pm 0.11$ \\
 & $10^{-2}$ & $3.08 \pm 0.03$ & $3.12 \pm 0.05$ & \cellcolor{gray!20}$8.76 \pm 5.59$ \\
\cmidrule(lr){1-5}
\multirow{3}{*}{$\mathcal{R}_{\mathrm{BW}}+\mathcal{R}_{\mathrm{LP}}$} & $10^{-4}$ & $3.03 \pm 0.02$ & $3.04 \pm 0.02$ & $3.00 \pm 0.02$ \\
 & $10^{-3}$ & $2.99 \pm 0.02$ & $2.98 \pm 0.02$ & $\mathbf{2.91 \pm 0.03}$ \\
 & $10^{-2}$ & $2.92 \pm 0.01$ & \cellcolor{gray!20}$23.05 \pm 8.71$ & \cellcolor{gray!20}$23.18 \pm 10.39$ \\
\botrule
\end{tabular}
\end{table}

\begin{table}
\caption{Pinball loss (mean $\pm$ std over seeds, scaled $\times 10^2$) on \texttt{cflp\_75\_75} for each regularizer and architecture. Values are computed on the held-out test set. \textbf{Bold} indicates the lowest pinball loss within each $\lambda$ level per architecture, excluding shaded rows. \colorbox{gray!20}{Shaded} entries exceed the unregularized baseline by more than 10\%, indicating that the regularizer has degraded predictive quality.}
\label{tab:cflp_75_75_pinball}
\centering
\begin{tabular}{ll  c  c  c}
\toprule
regularizer & $\lambda$ & \multicolumn{1}{c}{\texttt{75-25-25-50}} & \multicolumn{1}{c}{\texttt{75-25-25-25-50}} & \multicolumn{1}{c}{\texttt{75-25-25-25-25-25-50}} \\
 & & ($\times 10^{2}$) & ($\times 10^{2}$) & ($\times 10^{2}$) \\
\midrule
None & --- & $2.53 \pm 0.04$ & $2.49 \pm 0.03$ & $2.47 \pm 0.02$ \\
\midrule
\multirow{3}{*}{$\mathcal{R}_{\mathrm{L1}}$} & $10^{-4}$ & $\mathbf{2.40 \pm 0.01}$ & $\mathbf{2.35 \pm 0.00}$ & $\mathbf{2.33 \pm 0.01}$ \\
 & $10^{-3}$ & \cellcolor{gray!20}$2.88 \pm 0.12$ & \cellcolor{gray!20}$2.75 \pm 0.18$ & $2.58 \pm 0.05$ \\
 & $10^{-2}$ & \cellcolor{gray!20}$30.51 \pm 6.23$ & \cellcolor{gray!20}$30.81 \pm 6.63$ & \cellcolor{gray!20}$34.69 \pm 6.43$ \\
\cmidrule(lr){1-5}
\multirow{3}{*}{$\mathcal{R}_{\mathrm{L2}}$} & $10^{-4}$ & $2.49 \pm 0.02$ & $2.48 \pm 0.02$ & $2.44 \pm 0.02$ \\
 & $10^{-3}$ & $2.60 \pm 0.02$ & $2.50 \pm 0.01$ & $2.46 \pm 0.01$ \\
 & $10^{-2}$ & \cellcolor{gray!20}$24.11 \pm 0.11$ & \cellcolor{gray!20}$22.10 \pm 5.75$ & \cellcolor{gray!20}$15.82 \pm 9.60$ \\
\cmidrule(lr){1-5}
\multirow{3}{*}{$\mathcal{R}_{\mathrm{BW}}$} & $10^{-4}$ & $2.51 \pm 0.03$ & $2.46 \pm 0.02$ & $2.43 \pm 0.02$ \\
 & $10^{-3}$ & $\mathbf{2.44 \pm 0.02}$ & $\mathbf{2.40 \pm 0.01}$ & $2.37 \pm 0.02$ \\
 & $10^{-2}$ & $\mathbf{2.37 \pm 0.01}$ & $\mathbf{2.34 \pm 0.01}$ & $2.33 \pm 0.01$ \\
\cmidrule(lr){1-5}
\multirow{3}{*}{$\mathcal{R}_{\mathrm{SN}}$} & $10^{-4}$ & $2.52 \pm 0.03$ & $2.48 \pm 0.03$ & $2.45 \pm 0.03$ \\
 & $10^{-3}$ & $2.50 \pm 0.03$ & $2.44 \pm 0.03$ & $2.40 \pm 0.02$ \\
 & $10^{-2}$ & $2.39 \pm 0.02$ & $2.34 \pm 0.02$ & $\mathbf{2.31 \pm 0.01}$ \\
\cmidrule(lr){1-5}
\multirow{3}{*}{$\mathcal{R}_{\mathrm{LP}}$} & $10^{-4}$ & $2.55 \pm 0.03$ & $2.56 \pm 0.02$ & $2.50 \pm 0.03$ \\
 & $10^{-3}$ & $2.61 \pm 0.04$ & $2.60 \pm 0.04$ & \cellcolor{gray!20}$2.79 \pm 0.04$ \\
 & $10^{-2}$ & $2.60 \pm 0.06$ & $2.65 \pm 0.05$ & \cellcolor{gray!20}$7.50 \pm 4.08$ \\
\cmidrule(lr){1-5}
\multirow{3}{*}{$\mathcal{R}_{\mathrm{BW}}+\mathcal{R}_{\mathrm{LP}}$} & $10^{-4}$ & $2.53 \pm 0.02$ & $2.50 \pm 0.03$ & $2.47 \pm 0.03$ \\
 & $10^{-3}$ & $2.47 \pm 0.02$ & $2.44 \pm 0.02$ & $\mathbf{2.36 \pm 0.05}$ \\
 & $10^{-2}$ & \cellcolor{gray!20}$3.55 \pm 5.23$ & \cellcolor{gray!20}$26.95 \pm 0.76$ & \cellcolor{gray!20}$28.69 \pm 0.75$ \\
\botrule
\end{tabular}
\end{table}

\subsubsection{Optimization Results}

Tables~\ref{tab:cflp_50_50}--\ref{tab:cflp_75_75} report the same four MILP tractability metrics for the 2SP formulations~\eqref{eq:2sp-obj}: unstable neuron count $|\mathcal{U}|$, LP relaxation gap, branch-and-bound nodes, and wall-clock MILP solve time. 
We arbitrarily select $\lambda_\mathrm{2SP} = 0.1$ for the risk-aversion parameter and $\alpha = 0.9$ is the CVaR confidence level.
The unregularized baseline again becomes increasingly intractable with network depth, with mean solve times in the \texttt{cflp\_50\_50} case study ranging from 2.3\,s (2 hidden layers) to 702\,s (5 hidden layers) and node counts growing from $5{,}235$ to $270{,}161$. 
These are magnified in the \texttt{cflp\_75\_75} case study to 4.6\,s (2 hidden layers) and 1395\,s (5 hidden layers), with node counts growing from $9{,}061$ to $1{,}799{,}004$, reflecting the more challenging response surface learned by the surrogate.

In these two settings, we again observe that the L1 regularizer can be effective at low weights ($\lambda=10^{-4}$), but must be deployed cautiously as it can quickly degrade model prediction performance. 
In contrast, the bound-width regularizer $\mathcal{R}_{\mathrm{BW}}$ improves downstream performance with less sensitivity to tuning. For example, at $\lambda = 10^{-2}$ $\mathcal{R}_{\mathrm{BW}}$ reduces MILP solve times by 1--3 orders of magnitude across all architectures while maintaining the lowest pinball loss in both problem settings (Tables~\ref{tab:cflp_50_50_pinball}--\ref{tab:cflp_75_75_pinball}).
On the deepest architectures, the MILP solution time drops several orders of magnitude, from 702\,s and 1395\,s, on the two problems respectively, to less than 1\,s, with unstable neurons reduced from 125 to approximately 20. 
The stability regularizer $\mathcal{R}_{\mathrm{SN}}$ at $\lambda = 10^{-2}$ achieves similar improvements without sensitivity to regularization weight tuning, reducing the 5-layer solve time to $\approx$0.35\,s with 25--30 unstable neurons, though we again observe its LP gap reduction is slightly less aggressive.

We again observe the LP-based regularizer $\mathcal{R}_{\mathrm{LP}}$ effectively reduces the LP gap but, as in the benchmark experiments, does not reduce the number of unstable neurons by itself. 
On the 2- and 3-layer architectures, this still yields meaningful acceleration of the downstream MILP (e.g., 0.56\,s and 1.47\,s vs.\ baselines of 2.3\,s and 18\,s in the \texttt{cflp\_50\_50} setting).
However, on the deepest architectures $\mathcal{R}_{\mathrm{LP}}$ alone is insufficient, matching our observations from the Ackley function above. Here, despite reducing the LP gap from $2.5 \times 10^6$ to $1.4 \times 10^4$ at $\lambda = 10^{-2}$ in the \texttt{cflp\_50\_50} setting), the network retains all 125 unstable neurons and still requires $>$100\,s to solve. 
Notably, $\mathcal{R}_{\mathrm{LP}}$ at $\lambda = 10^{-2}$ also degrades prediction quality on the 5-layer architecture (pinball loss $8.76$ vs.\ baseline $3.06$), perhaps reflecting the difficulty of the multi-output LP regularization at strong regularization strengths. 
The same trends can be seen in the larger problem. 
Alternatively, this could simply be an effect of the regularizer strength, as the shrinkage regularizers also degrade prediction quality at this weight. 

The combined regularizer $\mathcal{R}_{\mathrm{BW}} + \mathcal{R}_{\mathrm{LP}}$ at $\lambda = 10^{-3}$ achieves a particularly effective balance of the above effects in both problem settings. For the 5-layer surrogate model architecture, it reduces $|\mathcal{U}|$ from 125 to $>$50 at a weight of $\lambda=10^{-3}$, reduces the LP gap from $\mathcal{O}(10^6)$ to approximately $400$ on both problems, and produces a MILP solvable in under 1\,s without degrading (even improving) prediction quality.
Across almost all architectures the combined regularizer consistently matches or improves upon the individual components, confirming the complementarity strengths of targeting both neuron stability (via $\mathcal{R}_{\mathrm{BW}}$) and relaxation tightness (via $\mathcal{R}_{\mathrm{LP}}$).

\begin{figure}[ht]
\centering
    \includegraphics[width=0.9\textwidth]{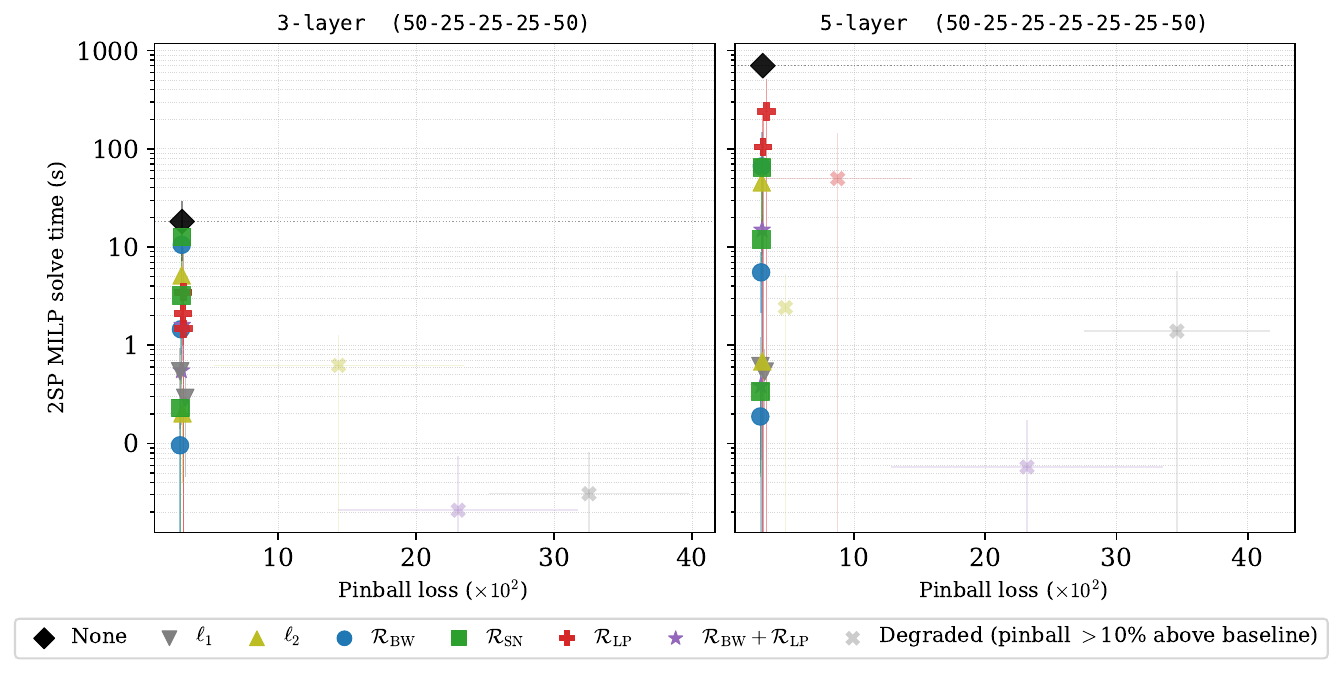}
    \caption{Tradeoff between solution time of downstream MILP performance and pinball loss for QNNs in stochastic programming setting. Each plotted marker shows mean performance for a different regularizer and weight.}
    \label{fig:pareto}
\end{figure}

\paragraph{Comparison at equal prediction quality}

The shading convention in Tables~\ref{tab:cflp_50_50_pinball}--\ref{tab:cflp_75_75_pinball} and Tables~\ref{tab:cflp_50_50}--\ref{tab:cflp_75_75} exposes a critical confound in na\"ive comparisons. 
As trends are largely consistent between the two problem settings, we focus our discussion on the \texttt{cflp\_50\_50} setting here. 
We observe that $\mathcal{R}_{\mathrm{L1}}$ at $\lambda = 10^{-2}$ achieves the fastest MILP solve times in raw numbers (0.03\,s), but this is only because the network has collapsed. In fact, Table~\ref{tab:cflp_50_50_pinball} shows the pinball loss is $10\times$ worse than the baseline.
At comparable prediction quality, e.g., comparing $\mathcal{R}_{\mathrm{L1}}$ at $\lambda = 10^{-4}$ (pinball $2.89$, MILP time 0.54\,s on the 3-layer model) against $\mathcal{R}_{\mathrm{BW}}$ at $\lambda = 10^{-2}$ (pinball $2.88$, MILP time 0.10\,s), the proposed relaxation-informed regularizer achieves a $5\times$ faster solve with $3\times$ fewer unstable neurons for a QNN surrogate model the same predictive quality.
This highlights the importance of accounting for prediction quality when evaluating MILP tractability improvements. 

Figure~\ref{fig:pareto} compares the pinball loss achieved by various regularizer configurations against the downstream MILP solve times in this setting, allowing us to visualize this tradeoff. 
Many points are clustered on the lefthand side of both plots (pinball loss close to the unnormalized baseline of $\approx$3). In this region of regularizers that maintain prediction quality, training with relaxation-aware regularizers $\mathcal{R}_{\mathrm{BW}}$, $\mathcal{R}_{\mathrm{SN}}$, and the combined regularizer sit at the bottom (lowest MILP solve times), forming the Pareto front. Improvements are most dramatic for the deeper QNN surrogate models. 
The shrinkage regularizers (and sometimes the combined regularizer) can produce model-collapse configurations where the regularizer dominates, resulting in a higher pinball loss values. These configurations are indicated using faded markers.

\begin{sidewaystable}[h]
\caption{Results on the \texttt{cflp\_50\_50} QNN surrogate 2SP across architectures and regularizers. Each regularizer family shows three rows for $\lambda \in \{10^{-4}, 10^{-3}, 10^{-2}\}$. $|\mathcal{U}|$: mean number of unstable neurons; LP gap: mean 2SP LP relaxation gap; MILP nodes/time: mean branch-and-bound nodes and wall-clock time (s) for the 2SP mean+CVaR MILP. \textbf{Bold} values indicate the best (lowest) result across all regularizers at each $\lambda$ level per architecture; ties are all bolded. \colorbox{gray!20}{Shaded} entries indicate regularizers whose test pinball loss exceeds the unregularized baseline by more than 10\%, reflecting degraded predictive quality; such configurations are excluded from best-value consideration.}
\label{tab:cflp_50_50}
\centering
\begin{tabular}{ll  cccc  cccc  cccc}
\toprule
\multirow{3}{*}{regularizer} & \multirow{3}{*}{$\lambda$} & \multicolumn{4}{c}{\texttt{50-25-25-50}} & \multicolumn{4}{c}{\texttt{50-25-25-25-50}} & \multicolumn{4}{c}{\texttt{50-25-25-25-25-25-50}} \\
\cmidrule(lr){3-6} \cmidrule(lr){7-10} \cmidrule(lr){11-14}
 & & $|\mathcal{U}|$ & LP gap & \multicolumn{2}{c}{MILP} & $|\mathcal{U}|$ & LP gap & \multicolumn{2}{c}{MILP} & $|\mathcal{U}|$ & LP gap & \multicolumn{2}{c}{MILP} \\
\cmidrule(lr){5-6} \cmidrule(lr){9-10} \cmidrule(lr){13-14}
 &  &  &  & nodes & time &  &  & nodes & time &  &  & nodes & time \\
\midrule
None & --- & $50.0$ & $44005.06$ & $5{,}235$ & $2.31$ & $75.0$ & $196857.98$ & $11{,}982$ & $18.16$ & $125.0$ & $2469369.64$ & $270{,}161$ & $701.74$ \\
\midrule
\multirow{3}{*}{$\mathcal{R}_{\mathrm{L1}}$} & $10^{-4}$ & $\mathbf{32.2}$ & $\mathbf{319.71}$ & $\mathbf{1}$ & $\mathbf{0.39}$ & $\mathbf{55.0}$ & $\mathbf{593.19}$ & $\mathbf{4}$ & $\mathbf{0.54}$ & $103.7$ & $\mathbf{7634.47}$ & $\mathbf{44}$ & $\mathbf{0.61}$ \\
 & $10^{-3}$ & \cellcolor{gray!20}$32.0$ & \cellcolor{gray!20}$110.63$ & \cellcolor{gray!20}$12$ & \cellcolor{gray!20}$0.29$ & $\mathbf{46.8}$ & $\mathbf{404.72}$ & $145$ & $0.29$ & $75.7$ & $829.68$ & $52$ & $0.54$ \\
 & $10^{-2}$ & \cellcolor{gray!20}$32.4$ & \cellcolor{gray!20}$37.42$ & \cellcolor{gray!20}$1$ & \cellcolor{gray!20}$0.03$ & \cellcolor{gray!20}$37.0$ & \cellcolor{gray!20}$102.94$ & \cellcolor{gray!20}$1$ & \cellcolor{gray!20}$0.03$ & \cellcolor{gray!20}$50.0$ & \cellcolor{gray!20}$222.49$ & \cellcolor{gray!20}$384$ & \cellcolor{gray!20}$1.39$ \\
\cmidrule(lr){1-14}
\multirow{3}{*}{$\mathcal{R}_{\mathrm{L2}}$} & $10^{-4}$ & $46.1$ & $11756.43$ & $456$ & $0.63$ & $70.9$ & $41684.35$ & $7{,}079$ & $5.14$ & $117.5$ & $237213.23$ & $36{,}561$ & $45.49$ \\
 & $10^{-3}$ & $40.2$ & $2326.55$ & $\mathbf{1}$ & $\mathbf{0.25}$ & $54.3$ & $6566.90$ & $\mathbf{1}$ & $\mathbf{0.20}$ & $79.8$ & $11557.11$ & $56$ & $0.68$ \\
 & $10^{-2}$ & \cellcolor{gray!20}$44.5$ & \cellcolor{gray!20}$744.48$ & \cellcolor{gray!20}$1$ & \cellcolor{gray!20}$0.06$ & \cellcolor{gray!20}$60.9$ & \cellcolor{gray!20}$1639.93$ & \cellcolor{gray!20}$313$ & \cellcolor{gray!20}$0.62$ & \cellcolor{gray!20}$97.1$ & \cellcolor{gray!20}$3470.08$ & \cellcolor{gray!20}$4{,}049$ & \cellcolor{gray!20}$2.41$ \\
\cmidrule(lr){1-14}
\multirow{3}{*}{$\mathcal{R}_{\mathrm{BW}}$} & $10^{-4}$ & $48.1$ & $26486.03$ & $3{,}621$ & $1.71$ & $69.8$ & $75976.64$ & $9{,}747$ & $10.53$ & $108.2$ & $217288.14$ & $30{,}437$ & $66.05$ \\
 & $10^{-3}$ & $38.9$ & $4776.56$ & $193$ & $0.49$ & $57.4$ & $17001.69$ & $1{,}645$ & $1.45$ & $77.2$ & $41766.83$ & $5{,}948$ & $5.52$ \\
 & $10^{-2}$ & $16.4$ & $207.46$ & $\mathbf{1}$ & $\mathbf{0.08}$ & $\mathbf{17.8}$ & $\mathbf{425.11}$ & $\mathbf{1}$ & $\mathbf{0.10}$ & $\mathbf{18.9}$ & $\mathbf{756.10}$ & $4$ & $\mathbf{0.19}$ \\
\cmidrule(lr){1-14}
\multirow{3}{*}{$\mathcal{R}_{\mathrm{SN}}$} & $10^{-4}$ & $46.9$ & $27429.07$ & $3{,}504$ & $1.43$ & $69.5$ & $97827.09$ & $9{,}981$ & $12.62$ & $107.6$ & $411523.37$ & $41{,}510$ & $64.14$ \\
 & $10^{-3}$ & $40.7$ & $9861.86$ & $1{,}458$ & $1.14$ & $59.0$ & $30793.00$ & $4{,}244$ & $3.22$ & $85.0$ & $84640.76$ & $7{,}524$ & $11.94$ \\
 & $10^{-2}$ & $21.3$ & $935.90$ & $\mathbf{1}$ & $0.17$ & $25.4$ & $1349.28$ & $13$ & $0.23$ & $29.9$ & $1201.37$ & $\mathbf{1}$ & $0.34$ \\
\cmidrule(lr){1-14}
\multirow{3}{*}{$\mathcal{R}_{\mathrm{LP}}$} & $10^{-4}$ & $49.8$ & $4410.45$ & $124$ & $0.65$ & $74.5$ & $18490.36$ & $2{,}325$ & $2.10$ & $125.0$ & $22936.09$ & $84{,}133$ & $104.38$ \\
 & $10^{-3}$ & $48.8$ & $673.95$ & $746$ & $0.60$ & $74.7$ & $1171.38$ & $5{,}972$ & $3.47$ & $125.0$ & $49816.20$ & $365{,}632$ & $239.22$ \\
 & $10^{-2}$ & $49.1$ & $347.71$ & $1{,}057$ & $0.56$ & $74.8$ & $522.99$ & $3{,}189$ & $1.47$ & \cellcolor{gray!20}$125.0$ & \cellcolor{gray!20}$14303.72$ & \cellcolor{gray!20}$64{,}183$ & \cellcolor{gray!20}$49.79$ \\
\cmidrule(lr){1-14}
\multirow{3}{*}{$\mathcal{R}_{\mathrm{BW}}+\mathcal{R}_{\mathrm{LP}}$} & $10^{-4}$ & $47.7$ & $2390.24$ & $26$ & $0.58$ & $68.6$ & $10001.84$ & $1{,}525$ & $1.60$ & $\mathbf{101.6}$ & $27172.28$ & $8{,}130$ & $14.94$ \\
 & $10^{-3}$ & $\mathbf{35.4}$ & $\mathbf{355.61}$ & $311$ & $0.52$ & $48.8$ & $429.38$ & $146$ & $0.55$ & $\mathbf{48.4}$ & $\mathbf{397.85}$ & $\mathbf{1}$ & $\mathbf{0.39}$ \\
 & $10^{-2}$ & $\mathbf{11.8}$ & $\mathbf{78.19}$ & $\mathbf{1}$ & $0.15$ & \cellcolor{gray!20}$8.2$ & \cellcolor{gray!20}$83.49$ & \cellcolor{gray!20}$1$ & \cellcolor{gray!20}$0.02$ & \cellcolor{gray!20}$22.9$ & \cellcolor{gray!20}$182.73$ & \cellcolor{gray!20}$1$ & \cellcolor{gray!20}$0.06$ \\
\botrule
\end{tabular}
\end{sidewaystable}

\begin{sidewaystable}[h]
\caption{Results on the \texttt{cflp\_75\_75} QNN surrogate 2SP across architectures and regularizers. Each regularizer family shows three rows for $\lambda \in \{10^{-4}, 10^{-3}, 10^{-2}\}$. $|\mathcal{U}|$: mean number of unstable neurons; LP gap: mean 2SP LP relaxation gap; MILP nodes/time: mean branch-and-bound nodes and wall-clock time (s) for the 2SP mean+CVaR MILP. \textbf{Bold} values indicate the best (lowest) result across all regularizers at each $\lambda$ level per architecture; ties are all bolded. \colorbox{gray!20}{Shaded} entries indicate regularizers whose test pinball loss exceeds the unregularized baseline by more than 10\%, reflecting degraded predictive quality; such configurations are excluded from best-value consideration.}
\label{tab:cflp_75_75}
\centering
\begin{tabular}{ll  cccc  cccc  cccc}
\toprule
\multirow{3}{*}{regularizer} & \multirow{3}{*}{$\lambda$} & \multicolumn{4}{c}{\texttt{75-25-25-50}} & \multicolumn{4}{c}{\texttt{75-25-25-25-50}} & \multicolumn{4}{c}{\texttt{75-25-25-25-25-25-50}} \\
\cmidrule(lr){3-6} \cmidrule(lr){7-10} \cmidrule(lr){11-14}
 & & $|\mathcal{U}|$ & LP gap & \multicolumn{2}{c}{MILP} & $|\mathcal{U}|$ & LP gap & \multicolumn{2}{c}{MILP} & $|\mathcal{U}|$ & LP gap & \multicolumn{2}{c}{MILP} \\
\cmidrule(lr){5-6} \cmidrule(lr){9-10} \cmidrule(lr){13-14}
 &  &  &  & nodes & time &  &  & nodes & time &  &  & nodes & time \\
\midrule
None & --- & $50.0$ & $76961.06$ & $9{,}061$ & $4.64$ & $75.0$ & $332696.67$ & $13{,}755$ & $20.16$ & $125.0$ & $3562051.76$ & $1{,}799{,}004$ & $1395.20$ \\
\midrule
\multirow{3}{*}{$\mathcal{R}_{\mathrm{L1}}$} & $10^{-4}$ & $\mathbf{32.2}$ & $\mathbf{1131.16}$ & $\mathbf{84}$ & $\mathbf{0.27}$ & $\mathbf{53.9}$ & $\mathbf{2322.02}$ & $\mathbf{121}$ & $\mathbf{0.48}$ & $105.0$ & $\mathbf{12883.05}$ & $\mathbf{500}$ & $\mathbf{1.05}$ \\
 & $10^{-3}$ & \cellcolor{gray!20}$34.2$ & \cellcolor{gray!20}$193.47$ & \cellcolor{gray!20}$1$ & \cellcolor{gray!20}$0.34$ & \cellcolor{gray!20}$47.9$ & \cellcolor{gray!20}$1211.57$ & \cellcolor{gray!20}$316$ & \cellcolor{gray!20}$0.50$ & $76.8$ & $2946.37$ & $396$ & $0.81$ \\
 & $10^{-2}$ & \cellcolor{gray!20}$31.2$ & \cellcolor{gray!20}$0.00$ & \cellcolor{gray!20}$1$ & \cellcolor{gray!20}$0.03$ & \cellcolor{gray!20}$38.2$ & \cellcolor{gray!20}$41.94$ & \cellcolor{gray!20}$1$ & \cellcolor{gray!20}$0.03$ & \cellcolor{gray!20}$27.2$ & \cellcolor{gray!20}$0.00$ & \cellcolor{gray!20}$1$ & \cellcolor{gray!20}$0.02$ \\
\cmidrule(lr){1-14}
\multirow{3}{*}{$\mathcal{R}_{\mathrm{L2}}$} & $10^{-4}$ & $48.4$ & $26889.45$ & $1{,}904$ & $1.27$ & $72.9$ & $85918.82$ & $9{,}295$ & $8.19$ & $120.0$ & $495568.32$ & $53{,}254$ & $62.58$ \\
 & $10^{-3}$ & $39.7$ & $4500.89$ & $\mathbf{1}$ & $\mathbf{0.28}$ & $55.4$ & $10438.69$ & $\mathbf{1}$ & $\mathbf{0.24}$ & $83.5$ & $25823.26$ & $967$ & $1.34$ \\
 & $10^{-2}$ & \cellcolor{gray!20}$44.9$ & \cellcolor{gray!20}$1245.53$ & \cellcolor{gray!20}$1$ & \cellcolor{gray!20}$0.09$ & \cellcolor{gray!20}$63.9$ & \cellcolor{gray!20}$1895.88$ & \cellcolor{gray!20}$147$ & \cellcolor{gray!20}$0.24$ & \cellcolor{gray!20}$92.8$ & \cellcolor{gray!20}$2950.60$ & \cellcolor{gray!20}$862$ & \cellcolor{gray!20}$0.92$ \\
\cmidrule(lr){1-14}
\multirow{3}{*}{$\mathcal{R}_{\mathrm{BW}}$} & $10^{-4}$ & $48.5$ & $46063.34$ & $8{,}071$ & $4.28$ & $71.0$ & $140223.45$ & $9{,}534$ & $11.02$ & $107.6$ & $358786.38$ & $71{,}025$ & $91.14$ \\
 & $10^{-3}$ & $40.8$ & $10063.15$ & $869$ & $0.84$ & $57.0$ & $29386.24$ & $3{,}820$ & $2.06$ & $69.5$ & $59873.45$ & $5{,}001$ & $4.95$ \\
 & $10^{-2}$ & $\mathbf{15.6}$ & $1342.98$ & $\mathbf{7}$ & $\mathbf{0.17}$ & $\mathbf{16.2}$ & $1549.56$ & $\mathbf{1}$ & $\mathbf{0.20}$ & $\mathbf{20.9}$ & $\mathbf{1689.45}$ & $\mathbf{1}$ & $\mathbf{0.21}$ \\
\cmidrule(lr){1-14}
\multirow{3}{*}{$\mathcal{R}_{\mathrm{SN}}$} & $10^{-4}$ & $48.1$ & $54656.72$ & $8{,}502$ & $3.67$ & $68.5$ & $177112.92$ & $9{,}151$ & $11.40$ & $107.0$ & $649873.14$ & $82{,}863$ & $135.28$ \\
 & $10^{-3}$ & $41.6$ & $24314.73$ & $3{,}984$ & $2.18$ & $57.7$ & $55394.37$ & $7{,}255$ & $4.54$ & $81.2$ & $118008.83$ & $6{,}244$ & $16.25$ \\
 & $10^{-2}$ & $20.1$ & $1697.08$ & $395$ & $0.36$ & $23.2$ & $2063.29$ & $45$ & $0.38$ & $26.9$ & $2162.87$ & $1$ & $0.35$ \\
\cmidrule(lr){1-14}
\multirow{3}{*}{$\mathcal{R}_{\mathrm{LP}}$} & $10^{-4}$ & $49.8$ & $7710.15$ & $1{,}626$ & $1.20$ & $74.9$ & $34157.12$ & $6{,}657$ & $5.73$ & $125.0$ & $30153.51$ & $146{,}887$ & $134.11$ \\
 & $10^{-3}$ & $48.6$ & $1132.86$ & $4{,}988$ & $2.15$ & $74.7$ & $2150.41$ & $50{,}828$ & $16.63$ & \cellcolor{gray!20}$125.0$ & \cellcolor{gray!20}$27486.72$ & \cellcolor{gray!20}$285{,}447$ & \cellcolor{gray!20}$260.60$ \\
 & $10^{-2}$ & $49.5$ & $\mathbf{363.43}$ & $11{,}305$ & $1.91$ & $74.9$ & $\mathbf{1112.30}$ & $29{,}829$ & $8.05$ & \cellcolor{gray!20}$125.0$ & \cellcolor{gray!20}$23664.74$ & \cellcolor{gray!20}$162{,}029$ & \cellcolor{gray!20}$109.02$ \\
\cmidrule(lr){1-14}
\multirow{3}{*}{$\mathcal{R}_{\mathrm{BW}}+\mathcal{R}_{\mathrm{LP}}$} & $10^{-4}$ & $47.6$ & $4314.56$ & $1{,}324$ & $1.13$ & $68.5$ & $20778.39$ & $3{,}851$ & $2.63$ & $\mathbf{100.8}$ & $37881.30$ & $7{,}448$ & $13.69$ \\
 & $10^{-3}$ & $\mathbf{35.9}$ & $\mathbf{556.51}$ & $1{,}262$ & $0.90$ & $\mathbf{50.1}$ & $\mathbf{720.99}$ & $491$ & $0.86$ & $\mathbf{43.0}$ & $\mathbf{432.37}$ & $\mathbf{8}$ & $\mathbf{0.33}$ \\
 & $10^{-2}$ & \cellcolor{gray!20}$10.1$ & \cellcolor{gray!20}$67.38$ & \cellcolor{gray!20}$43$ & \cellcolor{gray!20}$0.23$ & \cellcolor{gray!20}$8.2$ & \cellcolor{gray!20}$0.00$ & \cellcolor{gray!20}$0$ & \cellcolor{gray!20}$0.00$ & \cellcolor{gray!20}$16.0$ & \cellcolor{gray!20}$0.00$ & \cellcolor{gray!20}$0$ & \cellcolor{gray!20}$0.01$ \\
\botrule
\end{tabular}
\end{sidewaystable}

\paragraph{Limitations.}
This work focuses on the big-M MILP formulation of ReLU networks; tailored regularizers for other formulations (e.g., ideal formulations~\cite{anderson2020strong} or partition-based formulations~\cite{tsay2021partition}) and other activation functions remain to be established.
The bound-based regularizers $\mathcal{R}_{\mathrm{BW}}$ and $\mathcal{R}_{\mathrm{SN}}$ rely on IBP, which can produce increasingly loose bounds in deeper networks due to the recursive over-approximation discussed in Section~\ref{sec:ibp}. 
The LP-based regularizer $\mathcal{R}_{\mathrm{LP}}$ incurs a non-negligible training overhead (5--20$\times$, Table~\ref{tab:train_time_ratios}), which scales with network size and limits its practicality for very large models, though this is less relevant for the surrogate model setting and cost is incurred only once during training.
Furthermore, $\mathcal{R}_{\mathrm{LP}}$ targets pointwise relaxation gaps, and global tightening is not guaranteed.

\section{Conclusions}
\label{sec:conclusions}
This paper introduced a family of relaxation-informed regularization strategies that target the downstream tractability of neural network surrogate models during training. 
Two bound propagation-based regularizers, $\mathcal{R}_{\mathrm{BW}}$ (bound-width) and $\mathcal{R}_{\mathrm{SN}}$ (stable-neuron), penalize the big-M constants and the number of unstable neurons, respectively, through automatic differentiation of the bounds computation.
An LP relaxation gap regularizer $\mathcal{R}_{\mathrm{LP}}$ directly targets the continuous relaxation tightness, with gradients derived from dual variables via the envelope theorem for parametric linear programs.
Proposition~\ref{prop:combined} shows that combining $\mathcal{R}_{\mathrm{BW}}$ and $\mathcal{R}_{\mathrm{LP}}$ approximates the full total derivative of the LP gap with respect to the network parameters, capturing both the direct constraint sensitivity and the indirect big-M sensitivity.

Computational experiments on benchmark surrogate functions demonstrated that the proposed regularizers can reduce MILP solve times by up to four orders of magnitude while maintaining competitive modeling performance.
On the two-stage stochastic programming case study with quantile neural networks, training with the bound-width regularizer $\mathcal{R}_{\mathrm{BW}}$ could reduce the MILP solve time on the deepest architecture from over 700\,s to under 1\,s without degrading prediction quality, and the combined regularizer achieved similar acceleration, with complementary reductions in both unstable neuron count and LP gap.
These results highlight an important practical consideration: classical shrinkage regularizers ($\mathcal{R}_{\mathrm{L1}}$, $\mathcal{R}_{\mathrm{L2}}$) can improve tractability at strong regularization weights, but this often requires tuning to avoid a collapse in predictive quality (rather than genuine improvement of the MILP formulation).

Several directions for future work are worth noting.
The LP-based regularizer incurs a training overhead of approximately 5--20$\times$ due to repeated LP solves; exploiting GPU-based LP solvers~\cite{applegate2021practical,applegate2023faster} could substantially reduce this cost and integrate with neural network training pipelines.
Moreover, incorporating tighter bound propagation schemes (e.g., optimization-based bound tightening) as differentiable regularizers could yield further improvements, particularly for deeper architectures.
Finally, extending the framework to other activation functions and to more complex downstream formulations beyond the big-M MILP remains an open direction.

\section*{Acknowledgements}
Funding from the BASF/Royal Academy of Engineering Senior Research Fellowship is gratefully acknowledged.

\section*{Data Availability Statement}
No new data were created for synthetic benchmarks. Data for stochastic programming applications were generated following \texttt{https://github.com/khalil-research/Neur2SP}.

\clearpage
\bibliography{sn-bibliography}


\begin{thebibliography}{74}
\ifx \bisbn   \undefined \def \bisbn  #1{ISBN #1}\fi
\ifx \binits  \undefined \def \binits#1{#1}\fi
\ifx \bauthor  \undefined \def \bauthor#1{#1}\fi
\ifx \batitle  \undefined \def \batitle#1{#1}\fi
\ifx \bjtitle  \undefined \def \bjtitle#1{#1}\fi
\ifx \bvolume  \undefined \def \bvolume#1{\textbf{#1}}\fi
\ifx \byear  \undefined \def \byear#1{#1}\fi
\ifx \bissue  \undefined \def \bissue#1{#1}\fi
\ifx \bfpage  \undefined \def \bfpage#1{#1}\fi
\ifx \blpage  \undefined \def \blpage #1{#1}\fi
\ifx \burl  \undefined \def \burl#1{\textsf{#1}}\fi
\ifx \doiurl  \undefined \def \doiurl#1{\url{https://doi.org/#1}}\fi
\ifx \betal  \undefined \def \betal{\textit{et al.}}\fi
\ifx \binstitute  \undefined \def \binstitute#1{#1}\fi
\ifx \binstitutionaled  \undefined \def \binstitutionaled#1{#1}\fi
\ifx \bctitle  \undefined \def \bctitle#1{#1}\fi
\ifx \beditor  \undefined \def \beditor#1{#1}\fi
\ifx \bpublisher  \undefined \def \bpublisher#1{#1}\fi
\ifx \bbtitle  \undefined \def \bbtitle#1{#1}\fi
\ifx \bedition  \undefined \def \bedition#1{#1}\fi
\ifx \bseriesno  \undefined \def \bseriesno#1{#1}\fi
\ifx \blocation  \undefined \def \blocation#1{#1}\fi
\ifx \bsertitle  \undefined \def \bsertitle#1{#1}\fi
\ifx \bsnm \undefined \def \bsnm#1{#1}\fi
\ifx \bsuffix \undefined \def \bsuffix#1{#1}\fi
\ifx \bparticle \undefined \def \bparticle#1{#1}\fi
\ifx \barticle \undefined \def \barticle#1{#1}\fi
\bibcommenthead
\ifx \bconfdate \undefined \def \bconfdate #1{#1}\fi
\ifx \botherref \undefined \def \botherref #1{#1}\fi
\ifx \url \undefined \def \url#1{\textsf{#1}}\fi
\ifx \bchapter \undefined \def \bchapter#1{#1}\fi
\ifx \bbook \undefined \def \bbook#1{#1}\fi
\ifx \bcomment \undefined \def \bcomment#1{#1}\fi
\ifx \oauthor \undefined \def \oauthor#1{#1}\fi
\ifx \citeauthoryear \undefined \def \citeauthoryear#1{#1}\fi
\ifx \endbibitem  \undefined \def \endbibitem {}\fi
\ifx \bconflocation  \undefined \def \bconflocation#1{#1}\fi
\ifx \arxivurl  \undefined \def \arxivurl#1{\textsf{#1}}\fi
\csname PreBibitemsHook\endcsname

\bibitem[\protect\citeauthoryear{Bertsimas and
  Margaritis}{2025}]{bertsimas2025global}
\begin{barticle}
\bauthor{\bsnm{Bertsimas}, \binits{D.}},
\bauthor{\bsnm{Margaritis}, \binits{G.}}:
\batitle{Global optimization: a machine learning approach}.
\bjtitle{Journal of Global Optimization}
\bvolume{91}(\bissue{1}),
\bfpage{1}--\blpage{37}
(\byear{2025})
\end{barticle}
\endbibitem

\bibitem[\protect\citeauthoryear{Bradley
  et~al.}{2022}]{bradley2022perspectives}
\begin{barticle}
\bauthor{\bsnm{Bradley}, \binits{W.}},
\bauthor{\bsnm{Kim}, \binits{J.}},
\bauthor{\bsnm{Kilwein}, \binits{Z.}},
\bauthor{\bsnm{Blakely}, \binits{L.}},
\bauthor{\bsnm{Eydenberg}, \binits{M.}},
\bauthor{\bsnm{Jalvin}, \binits{J.}},
\bauthor{\bsnm{Laird}, \binits{C.}},
\bauthor{\bsnm{Boukouvala}, \binits{F.}}:
\batitle{Perspectives on the integration between first-principles and
  data-driven modeling}.
\bjtitle{Computers \& Chemical Engineering}
\bvolume{166},
\bfpage{107898}
(\byear{2022})
\end{barticle}
\endbibitem

\bibitem[\protect\citeauthoryear{Misener and
  Biegler}{2023}]{misener2023formulating}
\begin{barticle}
\bauthor{\bsnm{Misener}, \binits{R.}},
\bauthor{\bsnm{Biegler}, \binits{L.}}:
\batitle{Formulating data-driven surrogate models for process optimization}.
\bjtitle{Computers \& Chemical Engineering}
\bvolume{179},
\bfpage{108411}
(\byear{2023})
\end{barticle}
\endbibitem

\bibitem[\protect\citeauthoryear{Grimstad and
  Andersson}{2019}]{grimstad2019relu}
\begin{barticle}
\bauthor{\bsnm{Grimstad}, \binits{B.}},
\bauthor{\bsnm{Andersson}, \binits{H.}}:
\batitle{{ReLU} networks as surrogate models in mixed-integer linear programs}.
\bjtitle{Computers \& Chemical Engineering}
\bvolume{131},
\bfpage{106580}
(\byear{2019})
\end{barticle}
\endbibitem

\bibitem[\protect\citeauthoryear{Huchette et~al.}{2026}]{huchette2026deep}
\begin{botherref}
\oauthor{\bsnm{Huchette}, \binits{J.}},
\oauthor{\bsnm{Mu{\~n}oz}, \binits{G.}},
\oauthor{\bsnm{Serra}, \binits{T.}},
\oauthor{\bsnm{Tsay}, \binits{C.}}:
When deep learning meets polyhedral theory: A survey.
INFORMS Journal on Computing
(2026)
\end{botherref}
\endbibitem

\bibitem[\protect\citeauthoryear{Plate et~al.}{2026}]{plate2026analysis}
\begin{botherref}
\oauthor{\bsnm{Plate}, \binits{C.}},
\oauthor{\bsnm{Hahn}, \binits{M.}},
\oauthor{\bsnm{Klimek}, \binits{A.}},
\oauthor{\bsnm{Ganzer}, \binits{C.}},
\oauthor{\bsnm{Sundmacher}, \binits{K.}},
\oauthor{\bsnm{Sager}, \binits{S.}}:
An analysis of optimization problems involving relu neural networks.
Optimization and Engineering,
1--33
(2026)
\end{botherref}
\endbibitem

\bibitem[\protect\citeauthoryear{Botoeva et~al.}{2020}]{botoeva2020efficient}
\begin{bchapter}
\bauthor{\bsnm{Botoeva}, \binits{E.}},
\bauthor{\bsnm{Kouvaros}, \binits{P.}},
\bauthor{\bsnm{Kronqvist}, \binits{J.}},
\bauthor{\bsnm{Lomuscio}, \binits{A.}},
\bauthor{\bsnm{Misener}, \binits{R.}}:
\bctitle{Efficient verification of relu-based neural networks via dependency
  analysis}.
In: \bbtitle{Proceedings of the AAAI Conference on Artificial Intelligence},
vol. \bseriesno{34},
pp. \bfpage{3291}--\blpage{3299}
(\byear{2020})
\end{bchapter}
\endbibitem

\bibitem[\protect\citeauthoryear{R{\"o}ssig and
  Petkovic}{2021}]{rossig2021advances}
\begin{barticle}
\bauthor{\bsnm{R{\"o}ssig}, \binits{A.}},
\bauthor{\bsnm{Petkovic}, \binits{M.}}:
\batitle{Advances in verification of {ReLU} neural networks}.
\bjtitle{Journal of Global Optimization}
\bvolume{81}(\bissue{1}),
\bfpage{109}--\blpage{152}
(\byear{2021})
\end{barticle}
\endbibitem

\bibitem[\protect\citeauthoryear{Sosnin et~al.}{2024}]{sosnin2024certified}
\begin{botherref}
\oauthor{\bsnm{Sosnin}, \binits{P.}},
\oauthor{\bsnm{M{\"u}ller}, \binits{M.N.}},
\oauthor{\bsnm{Baader}, \binits{M.}},
\oauthor{\bsnm{Tsay}, \binits{C.}},
\oauthor{\bsnm{Wicker}, \binits{M.}}:
Certified robustness to data poisoning in gradient-based training.
arXiv preprint arXiv:2406.05670
(2024)
\end{botherref}
\endbibitem

\bibitem[\protect\citeauthoryear{Sosnin et~al.}{2026}]{sosnin2026exact}
\begin{botherref}
\oauthor{\bsnm{Sosnin}, \binits{P.}},
\oauthor{\bsnm{Knapp}, \binits{J.}},
\oauthor{\bsnm{Kennedy}, \binits{F.}},
\oauthor{\bsnm{Collyer}, \binits{J.}},
\oauthor{\bsnm{Tsay}, \binits{C.}}:
Exact certification of data-poisoning attacks using mixed-integer programming.
arXiv preprint arXiv:2602.16944
(2026)
\end{botherref}
\endbibitem

\bibitem[\protect\citeauthoryear{Kanamori et~al.}{2021}]{kanamori2021ordered}
\begin{bchapter}
\bauthor{\bsnm{Kanamori}, \binits{K.}},
\bauthor{\bsnm{Takagi}, \binits{T.}},
\bauthor{\bsnm{Kobayashi}, \binits{K.}},
\bauthor{\bsnm{Ike}, \binits{Y.}},
\bauthor{\bsnm{Uemura}, \binits{K.}},
\bauthor{\bsnm{Arimura}, \binits{H.}}:
\bctitle{Ordered counterfactual explanation by mixed-integer linear
  optimization}.
In: \bbtitle{Proceedings of the AAAI Conference on Artificial Intelligence},
vol. \bseriesno{35},
pp. \bfpage{11564}--\blpage{11574}
(\byear{2021})
\end{bchapter}
\endbibitem

\bibitem[\protect\citeauthoryear{Tsiourvas
  et~al.}{2024}]{tsiourvas2024manifold}
\begin{bchapter}
\bauthor{\bsnm{Tsiourvas}, \binits{A.}},
\bauthor{\bsnm{Sun}, \binits{W.}},
\bauthor{\bsnm{Perakis}, \binits{G.}}:
\bctitle{Manifold-aligned counterfactual explanations for neural networks}.
In: \bbtitle{International Conference on Artificial Intelligence and
  Statistics},
pp. \bfpage{3763}--\blpage{3771}
(\byear{2024}).
\bcomment{PMLR}
\end{bchapter}
\endbibitem

\bibitem[\protect\citeauthoryear{Burtea and Tsay}{2024}]{burtea2024constrained}
\begin{barticle}
\bauthor{\bsnm{Burtea}, \binits{R.}},
\bauthor{\bsnm{Tsay}, \binits{C.}}:
\batitle{Constrained continuous-action reinforcement learning for supply chain
  inventory management}.
\bjtitle{Computers \& Chemical Engineering}
\bvolume{181},
\bfpage{108518}
(\byear{2024})
\end{barticle}
\endbibitem

\bibitem[\protect\citeauthoryear{Ryu et~al.}{2019}]{ryu2019caql}
\begin{botherref}
\oauthor{\bsnm{Ryu}, \binits{M.}},
\oauthor{\bsnm{Chow}, \binits{Y.}},
\oauthor{\bsnm{Anderson}, \binits{R.}},
\oauthor{\bsnm{Tjandraatmadja}, \binits{C.}},
\oauthor{\bsnm{Boutilier}, \binits{C.}}:
{CAQL: Continuous action Q-learning}.
arXiv preprint arXiv:1909.12397
(2019)
\end{botherref}
\endbibitem

\bibitem[\protect\citeauthoryear{Benbaki et~al.}{2023}]{benbaki2023fast}
\begin{bchapter}
\bauthor{\bsnm{Benbaki}, \binits{R.}},
\bauthor{\bsnm{Chen}, \binits{W.}},
\bauthor{\bsnm{Meng}, \binits{X.}},
\bauthor{\bsnm{Hazimeh}, \binits{H.}},
\bauthor{\bsnm{Ponomareva}, \binits{N.}},
\bauthor{\bsnm{Zhao}, \binits{Z.}},
\bauthor{\bsnm{Mazumder}, \binits{R.}}:
\bctitle{Fast as chita: Neural network pruning with combinatorial
  optimization}.
In: \bbtitle{International Conference on Machine Learning},
pp. \bfpage{2031}--\blpage{2049}
(\byear{2023}).
\bcomment{PMLR}
\end{bchapter}
\endbibitem

\bibitem[\protect\citeauthoryear{Serra et~al.}{2021}]{serra2021scaling}
\begin{barticle}
\bauthor{\bsnm{Serra}, \binits{T.}},
\bauthor{\bsnm{Yu}, \binits{X.}},
\bauthor{\bsnm{Kumar}, \binits{A.}},
\bauthor{\bsnm{Ramalingam}, \binits{S.}}:
\batitle{Scaling up exact neural network compression by {ReLU} stability}.
\bjtitle{Advances in Neural Information Processing Systems}
\bvolume{34},
\bfpage{27081}--\blpage{27093}
(\byear{2021})
\end{barticle}
\endbibitem

\bibitem[\protect\citeauthoryear{Perakis and
  Tsiourvas}{2022}]{perakis2022optimizing}
\begin{botherref}
\oauthor{\bsnm{Perakis}, \binits{G.}},
\oauthor{\bsnm{Tsiourvas}, \binits{A.}}:
Optimizing objective functions from trained {ReLU} neural networks via
  sampling.
arXiv preprint arXiv:2205.14189
(2022)
\end{botherref}
\endbibitem

\bibitem[\protect\citeauthoryear{Tong et~al.}{2024}]{tong2024optimization}
\begin{bchapter}
\bauthor{\bsnm{Tong}, \binits{J.}},
\bauthor{\bsnm{Cai}, \binits{J.}},
\bauthor{\bsnm{Serra}, \binits{T.}}:
\bctitle{Optimization over trained neural networks: Taking a relaxing walk}.
In: \bbtitle{International Conference on the Integration of Constraint
  Programming, Artificial Intelligence, and Operations Research},
pp. \bfpage{221}--\blpage{233}
(\byear{2024}).
\bcomment{Springer}
\end{bchapter}
\endbibitem

\bibitem[\protect\citeauthoryear{Tong et~al.}{2025}]{tong2025optimization}
\begin{botherref}
\oauthor{\bsnm{Tong}, \binits{J.}},
\oauthor{\bsnm{Zhu}, \binits{Y.}},
\oauthor{\bsnm{Serra}, \binits{T.}},
\oauthor{\bsnm{Burer}, \binits{S.}}:
Optimization over trained neural networks: Going large with gradient-based
  algorithms.
arXiv preprint arXiv:2512.24295
(2025)
\end{botherref}
\endbibitem

\bibitem[\protect\citeauthoryear{Fajemisin
  et~al.}{2024}]{fajemisin2024optimization}
\begin{barticle}
\bauthor{\bsnm{Fajemisin}, \binits{A.O.}},
\bauthor{\bsnm{Maragno}, \binits{D.}},
\bauthor{\bsnm{Hertog}, \binits{D.}}:
\batitle{Optimization with constraint learning: A framework and survey}.
\bjtitle{European Journal of Operational Research}
\bvolume{314}(\bissue{1}),
\bfpage{1}--\blpage{14}
(\byear{2024})
\end{barticle}
\endbibitem

\bibitem[\protect\citeauthoryear{Maragno et~al.}{2025}]{maragno2025mixed}
\begin{barticle}
\bauthor{\bsnm{Maragno}, \binits{D.}},
\bauthor{\bsnm{Wiberg}, \binits{H.}},
\bauthor{\bsnm{Bertsimas}, \binits{D.}},
\bauthor{\bsnm{Birbil}, \binits{{\c{S}}.{\. I}.}},
\bauthor{\bsnm{Hertog}, \binits{D.}},
\bauthor{\bsnm{Fajemisin}, \binits{A.O.}}:
\batitle{Mixed-integer optimization with constraint learning}.
\bjtitle{Operations Research}
\bvolume{73}(\bissue{2}),
\bfpage{1011}--\blpage{1028}
(\byear{2025})
\end{barticle}
\endbibitem

\bibitem[\protect\citeauthoryear{Dumouchelle
  et~al.}{2023}]{dumouchelle2023neur2ro}
\begin{bchapter}
\bauthor{\bsnm{Dumouchelle}, \binits{J.}},
\bauthor{\bsnm{Julien}, \binits{E.}},
\bauthor{\bsnm{Kurtz}, \binits{J.}},
\bauthor{\bsnm{Khalil}, \binits{E.B.}}:
\bctitle{{Neur2RO}: Neural two-stage robust optimization}.
In: \bbtitle{International Conference on Learning Representations}
(\byear{2023})
\end{bchapter}
\endbibitem

\bibitem[\protect\citeauthoryear{Kronqvist
  et~al.}{2023}]{kronqvist2023alternating}
\begin{bchapter}
\bauthor{\bsnm{Kronqvist}, \binits{J.}},
\bauthor{\bsnm{Li}, \binits{B.}},
\bauthor{\bsnm{Rolfes}, \binits{J.}},
\bauthor{\bsnm{Zhao}, \binits{S.}}:
\bctitle{Alternating mixed-integer programming and neural network training for
  approximating stochastic two-stage problems}.
In: \bbtitle{International Conference on Machine Learning, Optimization, and
  Data Science},
pp. \bfpage{124}--\blpage{139}
(\byear{2023}).
\bcomment{Springer}
\end{bchapter}
\endbibitem

\bibitem[\protect\citeauthoryear{Patel et~al.}{2022}]{patel2022neur2sp}
\begin{barticle}
\bauthor{\bsnm{Patel}, \binits{R.M.}},
\bauthor{\bsnm{Dumouchelle}, \binits{J.}},
\bauthor{\bsnm{Khalil}, \binits{E.}},
\bauthor{\bsnm{Bodur}, \binits{M.}}:
\batitle{{Neur2SP}: Neural two-stage stochastic programming}.
\bjtitle{Advances in neural information processing systems}
\bvolume{35},
\bfpage{23992}--\blpage{24005}
(\byear{2022})
\end{barticle}
\endbibitem

\bibitem[\protect\citeauthoryear{Bergman et~al.}{2022}]{bergman2022janos}
\begin{barticle}
\bauthor{\bsnm{Bergman}, \binits{D.}},
\bauthor{\bsnm{Huang}, \binits{T.}},
\bauthor{\bsnm{Brooks}, \binits{P.}},
\bauthor{\bsnm{Lodi}, \binits{A.}},
\bauthor{\bsnm{Raghunathan}, \binits{A.U.}}:
\batitle{{JANOS}: an integrated predictive and prescriptive modeling
  framework}.
\bjtitle{INFORMS Journal on Computing}
\bvolume{34}(\bissue{2}),
\bfpage{807}--\blpage{816}
(\byear{2022})
\end{barticle}
\endbibitem

\bibitem[\protect\citeauthoryear{Ceccon et~al.}{2022}]{ceccon2022omlt}
\begin{barticle}
\bauthor{\bsnm{Ceccon}, \binits{F.}},
\bauthor{\bsnm{Jalving}, \binits{J.}},
\bauthor{\bsnm{Haddad}, \binits{J.}},
\bauthor{\bsnm{Thebelt}, \binits{A.}},
\bauthor{\bsnm{Tsay}, \binits{C.}},
\bauthor{\bsnm{Laird}, \binits{C.D.}},
\bauthor{\bsnm{Misener}, \binits{R.}}:
\batitle{{OMLT}: Optimization \& machine learning toolkit}.
\bjtitle{Journal of Machine Learning Research}
\bvolume{23}(\bissue{349}),
\bfpage{1}--\blpage{8}
(\byear{2022})
\end{barticle}
\endbibitem

\bibitem[\protect\citeauthoryear{Turner et~al.}{2025}]{turner2025pyscipopt}
\begin{bchapter}
\bauthor{\bsnm{Turner}, \binits{M.}},
\bauthor{\bsnm{Chmiela}, \binits{A.}},
\bauthor{\bsnm{Koch}, \binits{T.}},
\bauthor{\bsnm{Winkler}, \binits{M.}}:
\bctitle{{PySCIPOpt-ML}: Embedding trained machine learning models into
  mixed-integer programs}.
In: \bbtitle{International Conference on the Integration of Constraint
  Programming, Artificial Intelligence, and Operations Research},
pp. \bfpage{218}--\blpage{234}
(\byear{2025}).
\bcomment{Springer}
\end{bchapter}
\endbibitem

\bibitem[\protect\citeauthoryear{Jalving et~al.}{2023}]{jalving2023beyond}
\begin{barticle}
\bauthor{\bsnm{Jalving}, \binits{J.}},
\bauthor{\bsnm{Ghouse}, \binits{J.}},
\bauthor{\bsnm{Cortes}, \binits{N.}},
\bauthor{\bsnm{Gao}, \binits{X.}},
\bauthor{\bsnm{Knueven}, \binits{B.}},
\bauthor{\bsnm{Agi}, \binits{D.}},
\bauthor{\bsnm{Martin}, \binits{S.}},
\bauthor{\bsnm{Chen}, \binits{X.}},
\bauthor{\bsnm{Guittet}, \binits{D.}},
\bauthor{\bsnm{Tumbalam-Gooty}, \binits{R.}}, \betal:
\batitle{Beyond price taker: Conceptual design and optimization of integrated
  energy systems using machine learning market surrogates}.
\bjtitle{Applied Energy}
\bvolume{351},
\bfpage{121767}
(\byear{2023})
\end{barticle}
\endbibitem

\bibitem[\protect\citeauthoryear{L{\'o}pez-Flores
  et~al.}{2024}]{lopez2024process}
\begin{barticle}
\bauthor{\bsnm{L{\'o}pez-Flores}, \binits{F.J.}},
\bauthor{\bsnm{Ram{\'\i}rez-M{\'a}rquez}, \binits{C.}},
\bauthor{\bsnm{Ponce-Ortega}, \binits{J.M.}}:
\batitle{Process systems engineering tools for optimization of trained machine
  learning models: Comparative and perspective}.
\bjtitle{Industrial \& Engineering Chemistry Research}
\bvolume{63}(\bissue{32}),
\bfpage{13966}--\blpage{13979}
(\byear{2024})
\end{barticle}
\endbibitem

\bibitem[\protect\citeauthoryear{McDonald et~al.}{2024}]{mcdonald2024mixed}
\begin{barticle}
\bauthor{\bsnm{McDonald}, \binits{T.}},
\bauthor{\bsnm{Tsay}, \binits{C.}},
\bauthor{\bsnm{Schweidtmann}, \binits{A.M.}},
\bauthor{\bsnm{Yorke-Smith}, \binits{N.}}:
\batitle{Mixed-integer optimisation of graph neural networks for computer-aided
  molecular design}.
\bjtitle{Computers \& Chemical Engineering}
\bvolume{185},
\bfpage{108660}
(\byear{2024})
\end{barticle}
\endbibitem

\bibitem[\protect\citeauthoryear{Schweidtmann and
  Mitsos}{2019}]{schweidtmann2019deterministic}
\begin{barticle}
\bauthor{\bsnm{Schweidtmann}, \binits{A.M.}},
\bauthor{\bsnm{Mitsos}, \binits{A.}}:
\batitle{Deterministic global optimization with artificial neural networks
  embedded}.
\bjtitle{Journal of Optimization Theory and Applications}
\bvolume{180}(\bissue{3}),
\bfpage{925}--\blpage{948}
(\byear{2019})
\end{barticle}
\endbibitem

\bibitem[\protect\citeauthoryear{Fischetti and Jo}{2018}]{fischetti2018deep}
\begin{barticle}
\bauthor{\bsnm{Fischetti}, \binits{M.}},
\bauthor{\bsnm{Jo}, \binits{J.}}:
\batitle{Deep neural networks and mixed integer linear optimization}.
\bjtitle{Constraints}
\bvolume{23}(\bissue{3}),
\bfpage{296}--\blpage{309}
(\byear{2018})
\end{barticle}
\endbibitem

\bibitem[\protect\citeauthoryear{Lomuscio and
  Maganti}{2017}]{lomuscio2017approach}
\begin{botherref}
\oauthor{\bsnm{Lomuscio}, \binits{A.}},
\oauthor{\bsnm{Maganti}, \binits{L.}}:
An approach to reachability analysis for feed-forward relu neural networks.
arXiv preprint arXiv:1706.07351
(2017)
\end{botherref}
\endbibitem

\bibitem[\protect\citeauthoryear{Tjeng et~al.}{2017}]{tjeng2017evaluating}
\begin{bchapter}
\bauthor{\bsnm{Tjeng}, \binits{V.}},
\bauthor{\bsnm{Xiao}, \binits{K.Y.}},
\bauthor{\bsnm{Tedrake}, \binits{R.}}:
\bctitle{Evaluating robustness of neural networks with mixed integer
  programming}.
In: \bbtitle{International Conference on Learning Representations}
(\byear{2017})
\end{bchapter}
\endbibitem

\bibitem[\protect\citeauthoryear{Anderson et~al.}{2020}]{anderson2020strong}
\begin{barticle}
\bauthor{\bsnm{Anderson}, \binits{R.}},
\bauthor{\bsnm{Huchette}, \binits{J.}},
\bauthor{\bsnm{Ma}, \binits{W.}},
\bauthor{\bsnm{Tjandraatmadja}, \binits{C.}},
\bauthor{\bsnm{Vielma}, \binits{J.P.}}:
\batitle{Strong mixed-integer programming formulations for trained neural
  networks}.
\bjtitle{Mathematical Programming}
\bvolume{183}(\bissue{1}),
\bfpage{3}--\blpage{39}
(\byear{2020})
\end{barticle}
\endbibitem

\bibitem[\protect\citeauthoryear{Tsay et~al.}{2021}]{tsay2021partition}
\begin{barticle}
\bauthor{\bsnm{Tsay}, \binits{C.}},
\bauthor{\bsnm{Kronqvist}, \binits{J.}},
\bauthor{\bsnm{Thebelt}, \binits{A.}},
\bauthor{\bsnm{Misener}, \binits{R.}}:
\batitle{Partition-based formulations for mixed-integer optimization of trained
  {ReLU} neural networks}.
\bjtitle{Advances in neural information processing systems}
\bvolume{34},
\bfpage{3068}--\blpage{3080}
(\byear{2021})
\end{barticle}
\endbibitem

\bibitem[\protect\citeauthoryear{Badilla
  et~al.}{2023}]{badilla2023computational}
\begin{botherref}
\oauthor{\bsnm{Badilla}, \binits{F.}},
\oauthor{\bsnm{Goycoolea}, \binits{M.}},
\oauthor{\bsnm{Mu{\~n}oz}, \binits{G.}},
\oauthor{\bsnm{Serra}, \binits{T.}}:
Computational tradeoffs of optimization-based bound tightening in relu
  networks.
arXiv preprint arXiv:2312.16699
(2023)
\end{botherref}
\endbibitem

\bibitem[\protect\citeauthoryear{Sosnin and Tsay}{2024}]{sosnin2024scaling}
\begin{bchapter}
\bauthor{\bsnm{Sosnin}, \binits{P.}},
\bauthor{\bsnm{Tsay}, \binits{C.}}:
\bctitle{Scaling mixed-integer programming for certification of neural network
  controllers using bounds tightening}.
In: \bbtitle{2024 IEEE 63rd Conference on Decision and Control (CDC)},
pp. \bfpage{1645}--\blpage{1650}
(\byear{2024}).
\bcomment{IEEE}
\end{bchapter}
\endbibitem

\bibitem[\protect\citeauthoryear{Zhao et~al.}{2024}]{zhao2024bound}
\begin{bchapter}
\bauthor{\bsnm{Zhao}, \binits{H.}},
\bauthor{\bsnm{Hijazi}, \binits{H.}},
\bauthor{\bsnm{Jones}, \binits{H.}},
\bauthor{\bsnm{Moore}, \binits{J.}},
\bauthor{\bsnm{Tanneau}, \binits{M.}},
\bauthor{\bsnm{Van~Hentenryck}, \binits{P.}}:
\bctitle{Bound tightening using rolling-horizon decomposition for neural
  network verification}.
In: \bbtitle{International Conference on the Integration of Constraint
  Programming, Artificial Intelligence, and Operations Research},
pp. \bfpage{289}--\blpage{303}
(\byear{2024}).
\bcomment{Springer}
\end{bchapter}
\endbibitem

\bibitem[\protect\citeauthoryear{Milgrom and Segal}{2002}]{milgrom2002}
\begin{barticle}
\bauthor{\bsnm{Milgrom}, \binits{P.}},
\bauthor{\bsnm{Segal}, \binits{I.}}:
\batitle{Envelope theorems for arbitrary choice sets}.
\bjtitle{Econometrica}
\bvolume{70}(\bissue{2}),
\bfpage{583}--\blpage{601}
(\byear{2002})
\doiurl{10.1111/1468-0262.00296}
\end{barticle}
\endbibitem

\bibitem[\protect\citeauthoryear{Fiacco}{1983}]{fiacco1983}
\begin{bbook}
\bauthor{\bsnm{Fiacco}, \binits{A.V.}}:
\bbtitle{Introduction to Sensitivity and Stability Analysis in Nonlinear
  Programming}.
\bpublisher{Academic Press},
\blocation{New York}
(\byear{1983})
\end{bbook}
\endbibitem

\bibitem[\protect\citeauthoryear{Xiao et~al.}{2019}]{xiao2019training}
\begin{bchapter}
\bauthor{\bsnm{Xiao}, \binits{K.}},
\bauthor{\bsnm{Tjeng}, \binits{V.}},
\bauthor{\bsnm{Shafiullah}, \binits{N.M.}},
\bauthor{\bsnm{Madry}, \binits{A.}}:
\bctitle{Training for faster adversarial robustness verification via inducing
  {ReLU} stability}.
In: \bbtitle{International Conference on Learning Representations}
(\byear{2019})
\end{bchapter}
\endbibitem

\bibitem[\protect\citeauthoryear{Gowal et~al.}{2018}]{gowal2018effectiveness}
\begin{botherref}
\oauthor{\bsnm{Gowal}, \binits{S.}},
\oauthor{\bsnm{Dvijotham}, \binits{K.}},
\oauthor{\bsnm{Stanforth}, \binits{R.}},
\oauthor{\bsnm{Bunel}, \binits{R.}},
\oauthor{\bsnm{Qin}, \binits{C.}},
\oauthor{\bsnm{Uesato}, \binits{J.}},
\oauthor{\bsnm{Arandjelovic}, \binits{R.}},
\oauthor{\bsnm{Mann}, \binits{T.}},
\oauthor{\bsnm{Kohli}, \binits{P.}}:
On the effectiveness of interval bound propagation for training verifiably
  robust models.
arXiv preprint arXiv:1810.12715
(2018)
\end{botherref}
\endbibitem

\bibitem[\protect\citeauthoryear{Mirman
  et~al.}{2018}]{mirman2018differentiable}
\begin{bchapter}
\bauthor{\bsnm{Mirman}, \binits{M.}},
\bauthor{\bsnm{Gehr}, \binits{T.}},
\bauthor{\bsnm{Vechev}, \binits{M.}}:
\bctitle{Differentiable abstract interpretation for provably robust neural
  networks}.
In: \bbtitle{International Conference on Machine Learning},
pp. \bfpage{3578}--\blpage{3586}
(\byear{2018}).
\bcomment{PMLR}
\end{bchapter}
\endbibitem

\bibitem[\protect\citeauthoryear{Zhang et~al.}{2020}]{zhang2020towards}
\begin{bchapter}
\bauthor{\bsnm{Zhang}, \binits{H.}},
\bauthor{\bsnm{Chen}, \binits{H.}},
\bauthor{\bsnm{Xiao}, \binits{C.}},
\bauthor{\bsnm{Gowal}, \binits{S.}},
\bauthor{\bsnm{Stanforth}, \binits{R.}},
\bauthor{\bsnm{Li}, \binits{B.}},
\bauthor{\bsnm{Boning}, \binits{D.}},
\bauthor{\bsnm{Hsieh}, \binits{C.J.}}:
\bctitle{Towards stable and efficient training of verifiably robust neural
  networks}.
In: \bbtitle{International Conference on Learning Representations}
(\byear{2020})
\end{bchapter}
\endbibitem

\bibitem[\protect\citeauthoryear{Sosnin et~al.}{2025}]{sosnin2025abstract}
\begin{botherref}
\oauthor{\bsnm{Sosnin}, \binits{P.}},
\oauthor{\bsnm{Wicker}, \binits{M.}},
\oauthor{\bsnm{Collyer}, \binits{J.}},
\oauthor{\bsnm{Tsay}, \binits{C.}}:
Abstract gradient training: A unified certification framework for data
  poisoning, unlearning, and differential privacy.
arXiv preprint arXiv:2511.09400
(2025)
\end{botherref}
\endbibitem

\bibitem[\protect\citeauthoryear{Mandi et~al.}{2024}]{mandi2024decision}
\begin{barticle}
\bauthor{\bsnm{Mandi}, \binits{J.}},
\bauthor{\bsnm{Kotary}, \binits{J.}},
\bauthor{\bsnm{Berden}, \binits{S.}},
\bauthor{\bsnm{Mulamba}, \binits{M.}},
\bauthor{\bsnm{Bucarey}, \binits{V.}},
\bauthor{\bsnm{Guns}, \binits{T.}},
\bauthor{\bsnm{Fioretto}, \binits{F.}}:
\batitle{Decision-focused learning: Foundations, state of the art, benchmark
  and future opportunities}.
\bjtitle{Journal of Artificial Intelligence Research}
\bvolume{80},
\bfpage{1623}--\blpage{1701}
(\byear{2024})
\end{barticle}
\endbibitem

\bibitem[\protect\citeauthoryear{Elmachtoub and
  Grigas}{2022}]{elmachtoub2022smart}
\begin{barticle}
\bauthor{\bsnm{Elmachtoub}, \binits{A.N.}},
\bauthor{\bsnm{Grigas}, \binits{P.}}:
\batitle{Smart ``predict, then optimize''}.
\bjtitle{Management Science}
\bvolume{68}(\bissue{1}),
\bfpage{9}--\blpage{26}
(\byear{2022})
\end{barticle}
\endbibitem

\bibitem[\protect\citeauthoryear{Donti et~al.}{2017}]{donti2017task}
\begin{botherref}
\oauthor{\bsnm{Donti}, \binits{P.}},
\oauthor{\bsnm{Amos}, \binits{B.}},
\oauthor{\bsnm{Kolter}, \binits{J.Z.}}:
Task-based end-to-end model learning in stochastic optimization.
Advances in neural information processing systems
\textbf{30}
(2017)
\end{botherref}
\endbibitem

\bibitem[\protect\citeauthoryear{Dvijotham
  et~al.}{2018}]{dvijotham2018training}
\begin{botherref}
\oauthor{\bsnm{Dvijotham}, \binits{K.}},
\oauthor{\bsnm{Gowal}, \binits{S.}},
\oauthor{\bsnm{Stanforth}, \binits{R.}},
\oauthor{\bsnm{Arandjelovic}, \binits{R.}},
\oauthor{\bsnm{O'Donoghue}, \binits{B.}},
\oauthor{\bsnm{Uesato}, \binits{J.}},
\oauthor{\bsnm{Kohli}, \binits{P.}}:
Training verified learners with learned verifiers.
arXiv preprint arXiv:1805.10265
(2018)
\end{botherref}
\endbibitem

\bibitem[\protect\citeauthoryear{Tang and Khalil}{2024}]{TangKhalil2024}
\begin{barticle}
\bauthor{\bsnm{Tang}, \binits{B.}},
\bauthor{\bsnm{Khalil}, \binits{E.B.}}:
\batitle{{PyEPO: a PyTorch-based end-to-end predict-then-optimize library for
  linear and integer programming}}.
\bjtitle{Mathematical Programming Computation}
\bvolume{16}(\bissue{3}),
\bfpage{297}--\blpage{335}
(\byear{2024})
\doiurl{10.1007/s12532-024-00255-x}
\end{barticle}
\endbibitem

\bibitem[\protect\citeauthoryear{Amos et~al.}{2017}]{amos2017input}
\begin{bchapter}
\bauthor{\bsnm{Amos}, \binits{B.}},
\bauthor{\bsnm{Xu}, \binits{L.}},
\bauthor{\bsnm{Kolter}, \binits{J.Z.}}:
\bctitle{Input convex neural networks}.
In: \bbtitle{International Conference on Machine Learning},
pp. \bfpage{146}--\blpage{155}
(\byear{2017}).
\bcomment{PMLR}
\end{bchapter}
\endbibitem

\bibitem[\protect\citeauthoryear{Rosemberg et~al.}{2025}]{rosemberg2025sobolev}
\begin{botherref}
\oauthor{\bsnm{Rosemberg}, \binits{A.W.}},
\oauthor{\bsnm{Garcia}, \binits{J.D.}},
\oauthor{\bsnm{Bent}, \binits{R.}},
\oauthor{\bsnm{Van~Hentenryck}, \binits{P.}}:
Sobolev training of end-to-end optimization proxies.
arXiv preprint arXiv:2505.11342
(2025)
\end{botherref}
\endbibitem

\bibitem[\protect\citeauthoryear{Tsay}{2021}]{tsay2021sobolev}
\begin{barticle}
\bauthor{\bsnm{Tsay}, \binits{C.}}:
\batitle{Sobolev trained neural network surrogate models for optimization}.
\bjtitle{Computers \& Chemical Engineering}
\bvolume{153},
\bfpage{107419}
(\byear{2021})
\end{barticle}
\endbibitem

\bibitem[\protect\citeauthoryear{Goodfellow et~al.}{2016}]{goodfellow2016deep}
\begin{bbook}
\bauthor{\bsnm{Goodfellow}, \binits{I.}},
\bauthor{\bsnm{Bengio}, \binits{Y.}},
\bauthor{\bsnm{Courville}, \binits{A.}},
\bauthor{\bsnm{Bengio}, \binits{Y.}}:
\bbtitle{Deep Learning}
vol. \bseriesno{1}.
\bpublisher{MIT press Cambridge}, \blocation{???}
(\byear{2016})
\end{bbook}
\endbibitem

\bibitem[\protect\citeauthoryear{Manng{\aa}rd
  et~al.}{2018}]{manngaard2018structural}
\begin{barticle}
\bauthor{\bsnm{Manng{\aa}rd}, \binits{M.}},
\bauthor{\bsnm{Kronqvist}, \binits{J.}},
\bauthor{\bsnm{B{\"o}ling}, \binits{J.M.}}:
\batitle{Structural learning in artificial neural networks using sparse
  optimization}.
\bjtitle{Neurocomputing}
\bvolume{272},
\bfpage{660}--\blpage{667}
(\byear{2018})
\end{barticle}
\endbibitem

\bibitem[\protect\citeauthoryear{Alc{\'a}ntara
  et~al.}{2025}]{alcantara2025quantile}
\begin{barticle}
\bauthor{\bsnm{Alc{\'a}ntara}, \binits{A.}},
\bauthor{\bsnm{Ruiz}, \binits{C.}},
\bauthor{\bsnm{Tsay}, \binits{C.}}:
\batitle{A quantile neural network framework for two-stage stochastic
  optimization}.
\bjtitle{Expert Systems with Applications}
\bvolume{284},
\bfpage{127876}
(\byear{2025})
\end{barticle}
\endbibitem

\bibitem[\protect\citeauthoryear{Ghilardi
  et~al.}{2025}]{ghilardi2025integrated}
\begin{barticle}
\bauthor{\bsnm{Ghilardi}, \binits{L.M.}},
\bauthor{\bsnm{Patr{\'o}n}, \binits{G.D.}},
\bauthor{\bsnm{Alc{\'a}ntara}, \binits{A.}},
\bauthor{\bsnm{Tsay}, \binits{C.}}:
\batitle{Integrated design and scheduling of hydrogen processes under
  uncertainty: A quantile neural network approach}.
\bjtitle{Industrial \& Engineering Chemistry Research}
\bvolume{64}(\bissue{44}),
\bfpage{21235}--\blpage{21250}
(\byear{2025})
\end{barticle}
\endbibitem

\bibitem[\protect\citeauthoryear{Amos and Kolter}{2017}]{amos2017optnet}
\begin{bchapter}
\bauthor{\bsnm{Amos}, \binits{B.}},
\bauthor{\bsnm{Kolter}, \binits{J.Z.}}:
\bctitle{Optnet: Differentiable optimization as a layer in neural networks}.
In: \bbtitle{International Conference on Machine Learning},
pp. \bfpage{136}--\blpage{145}
(\byear{2017}).
\bcomment{PMLR}
\end{bchapter}
\endbibitem

\bibitem[\protect\citeauthoryear{Bertsimas and
  Tsitsiklis}{1997}]{bertsimas1997introduction}
\begin{bbook}
\bauthor{\bsnm{Bertsimas}, \binits{D.}},
\bauthor{\bsnm{Tsitsiklis}, \binits{J.N.}}:
\bbtitle{Introduction to Linear Optimization}.
\bpublisher{Athena Scientific},
\blocation{Belmont, MA}
(\byear{1997})
\end{bbook}
\endbibitem

\bibitem[\protect\citeauthoryear{Wilhelm et~al.}{2023}]{wilhelm2023convex}
\begin{barticle}
\bauthor{\bsnm{Wilhelm}, \binits{M.E.}},
\bauthor{\bsnm{Wang}, \binits{C.}},
\bauthor{\bsnm{Stuber}, \binits{M.D.}}:
\batitle{Convex and concave envelopes of artificial neural network activation
  functions for deterministic global optimization}.
\bjtitle{Journal of Global Optimization}
\bvolume{85}(\bissue{3}),
\bfpage{569}--\blpage{594}
(\byear{2023})
\end{barticle}
\endbibitem

\bibitem[\protect\citeauthoryear{Agrawal
  et~al.}{2019}]{agrawal2019differentiable}
\begin{botherref}
\oauthor{\bsnm{Agrawal}, \binits{A.}},
\oauthor{\bsnm{Amos}, \binits{B.}},
\oauthor{\bsnm{Barratt}, \binits{S.}},
\oauthor{\bsnm{Boyd}, \binits{S.}},
\oauthor{\bsnm{Diamond}, \binits{S.}},
\oauthor{\bsnm{Kolter}, \binits{J.Z.}}:
Differentiable convex optimization layers.
Advances in neural information processing systems
\textbf{32}
(2019)
\end{botherref}
\endbibitem

\bibitem[\protect\citeauthoryear{Pineda et~al.}{2022}]{pineda2022theseus}
\begin{barticle}
\bauthor{\bsnm{Pineda}, \binits{L.}},
\bauthor{\bsnm{Fan}, \binits{T.}},
\bauthor{\bsnm{Monge}, \binits{M.}},
\bauthor{\bsnm{Venkataraman}, \binits{S.}},
\bauthor{\bsnm{Sodhi}, \binits{P.}},
\bauthor{\bsnm{Chen}, \binits{R.T.}},
\bauthor{\bsnm{Ortiz}, \binits{J.}},
\bauthor{\bsnm{DeTone}, \binits{D.}},
\bauthor{\bsnm{Wang}, \binits{A.}},
\bauthor{\bsnm{Anderson}, \binits{S.}}, \betal:
\batitle{Theseus: A library for differentiable nonlinear optimization}.
\bjtitle{Advances in Neural Information Processing Systems}
\bvolume{35},
\bfpage{3801}--\blpage{3818}
(\byear{2022})
\end{barticle}
\endbibitem

\bibitem[\protect\citeauthoryear{Besan{\c{c}}on
  et~al.}{2024}]{besanccon2024flexible}
\begin{barticle}
\bauthor{\bsnm{Besan{\c{c}}on}, \binits{M.}},
\bauthor{\bsnm{Dias~Garcia}, \binits{J.}},
\bauthor{\bsnm{Legat}, \binits{B.}},
\bauthor{\bsnm{Sharma}, \binits{A.}}:
\batitle{Flexible differentiable optimization via model transformations}.
\bjtitle{INFORMS Journal on Computing}
\bvolume{36}(\bissue{2}),
\bfpage{456}--\blpage{478}
(\byear{2024})
\end{barticle}
\endbibitem

\bibitem[\protect\citeauthoryear{Rosemberg et~al.}{2025}]{rosemberg2025general}
\begin{botherref}
\oauthor{\bsnm{Rosemberg}, \binits{A.W.}},
\oauthor{\bsnm{Garcia}, \binits{J.D.}},
\oauthor{\bsnm{Pacaud}, \binits{F.}},
\oauthor{\bsnm{Parker}, \binits{R.B.}},
\oauthor{\bsnm{Legat}, \binits{B.}},
\oauthor{\bsnm{Sundar}, \binits{K.}},
\oauthor{\bsnm{Bent}, \binits{R.}},
\oauthor{\bsnm{Van~Hentenryck}, \binits{P.}}:
A general and streamlined differentiable optimization framework.
arXiv preprint arXiv:2510.25986
(2025)
\end{botherref}
\endbibitem

\bibitem[\protect\citeauthoryear{Bengio et~al.}{2013}]{bengio2013estimating}
\begin{botherref}
\oauthor{\bsnm{Bengio}, \binits{Y.}},
\oauthor{\bsnm{L{\'e}onard}, \binits{N.}},
\oauthor{\bsnm{Courville}, \binits{A.}}:
Estimating or propagating gradients through stochastic neurons for conditional
  computation.
arXiv preprint arXiv:1308.3432
(2013)
\end{botherref}
\endbibitem

\bibitem[\protect\citeauthoryear{Yin et~al.}{2019}]{yin2019understanding}
\begin{bchapter}
\bauthor{\bsnm{Yin}, \binits{P.}},
\bauthor{\bsnm{Lyu}, \binits{J.}},
\bauthor{\bsnm{Zhang}, \binits{S.}},
\bauthor{\bsnm{Osher}, \binits{S.}},
\bauthor{\bsnm{Qi}, \binits{Y.}},
\bauthor{\bsnm{Xin}, \binits{J.}}:
\bctitle{Understanding straight-through estimator in training activation
  quantized neural nets}.
In: \bbtitle{International Conference on Learning Representations}
(\byear{2019})
\end{bchapter}
\endbibitem

\bibitem[\protect\citeauthoryear{Paszke et~al.}{2019}]{paszke2019pytorch}
\begin{botherref}
\oauthor{\bsnm{Paszke}, \binits{A.}},
\oauthor{\bsnm{Gross}, \binits{S.}},
\oauthor{\bsnm{Massa}, \binits{F.}},
\oauthor{\bsnm{Lerer}, \binits{A.}},
\oauthor{\bsnm{Bradbury}, \binits{J.}},
\oauthor{\bsnm{Chanan}, \binits{G.}},
\oauthor{\bsnm{Killeen}, \binits{T.}},
\oauthor{\bsnm{Lin}, \binits{Z.}},
\oauthor{\bsnm{Gimelshein}, \binits{N.}},
\oauthor{\bsnm{Antiga}, \binits{L.}}, et al.:
Pytorch: An imperative style, high-performance deep learning library.
Advances in neural information processing systems
\textbf{32}
(2019)
\end{botherref}
\endbibitem

\bibitem[\protect\citeauthoryear{{Gurobi Optimization, LLC}}{2026}]{gurobi}
\begin{botherref}
\oauthor{\bsnm{{Gurobi Optimization, LLC}}}:
{Gurobi Optimizer Reference Manual}
(2026).
\url{https://www.gurobi.com}
\end{botherref}
\endbibitem

\bibitem[\protect\citeauthoryear{Huangfu and
  Hall}{2018}]{huangfu2018parallelizing}
\begin{barticle}
\bauthor{\bsnm{Huangfu}, \binits{Q.}},
\bauthor{\bsnm{Hall}, \binits{J.J.}}:
\batitle{Parallelizing the dual revised simplex method}.
\bjtitle{Mathematical Programming Computation}
\bvolume{10}(\bissue{1}),
\bfpage{119}--\blpage{142}
(\byear{2018})
\end{barticle}
\endbibitem

\bibitem[\protect\citeauthoryear{Applegate
  et~al.}{2021}]{applegate2021practical}
\begin{barticle}
\bauthor{\bsnm{Applegate}, \binits{D.}},
\bauthor{\bsnm{D{\'\i}az}, \binits{M.}},
\bauthor{\bsnm{Hinder}, \binits{O.}},
\bauthor{\bsnm{Lu}, \binits{H.}},
\bauthor{\bsnm{Lubin}, \binits{M.}},
\bauthor{\bsnm{O'Donoghue}, \binits{B.}},
\bauthor{\bsnm{Schudy}, \binits{W.}}:
\batitle{Practical large-scale linear programming using primal-dual hybrid
  gradient}.
\bjtitle{Advances in Neural Information Processing Systems}
\bvolume{34},
\bfpage{20243}--\blpage{20257}
(\byear{2021})
\end{barticle}
\endbibitem

\bibitem[\protect\citeauthoryear{Applegate et~al.}{2023}]{applegate2023faster}
\begin{barticle}
\bauthor{\bsnm{Applegate}, \binits{D.}},
\bauthor{\bsnm{Hinder}, \binits{O.}},
\bauthor{\bsnm{Lu}, \binits{H.}},
\bauthor{\bsnm{Lubin}, \binits{M.}}:
\batitle{Faster first-order primal-dual methods for linear programming using
  restarts and sharpness}.
\bjtitle{Mathematical Programming}
\bvolume{201}(\bissue{1}),
\bfpage{133}--\blpage{184}
(\byear{2023})
\end{barticle}
\endbibitem

\bibitem[\protect\citeauthoryear{Liu et~al.}{2025}]{liu2025icnn}
\begin{botherref}
\oauthor{\bsnm{Liu}, \binits{Y.}},
\oauthor{\bsnm{Oliveira}, \binits{F.}},
\oauthor{\bsnm{Kronqvist}, \binits{J.}}:
{ICNN-enhanced 2SP}: Leveraging input convex neural networks for solving
  two-stage stochastic programming.
arXiv preprint arXiv:2505.05261
(2025)
\end{botherref}
\endbibitem

\bibitem[\protect\citeauthoryear{Cornu{\'e}jols
  et~al.}{1991}]{cornuejols1991comparison}
\begin{barticle}
\bauthor{\bsnm{Cornu{\'e}jols}, \binits{G.}},
\bauthor{\bsnm{Sridharan}, \binits{R.}},
\bauthor{\bsnm{Thizy}, \binits{J.-M.}}:
\batitle{A comparison of heuristics and relaxations for the capacitated plant
  location problem}.
\bjtitle{European Journal of Operational Research}
\bvolume{50}(\bissue{3}),
\bfpage{280}--\blpage{297}
(\byear{1991})
\end{barticle}
\endbibitem

\end{thebibliography}

\end{document}